\documentclass{amsart}
\usepackage{amsmath,amsthm}
\usepackage{amsfonts,amssymb}
\usepackage{accents}
\usepackage{enumerate}
\usepackage{accents,color}
\usepackage{graphicx}

\newcommand{\lvt}{\left|\kern-1.35pt\left|\kern-1.3pt\left|}
\newcommand{\rvt}{\right|\kern-1.3pt\right|\kern-1.35pt\right|}

\newtheorem{thm}{Theorem}[section]
\newtheorem{cor}[thm]{Corollary}
\newtheorem{lem}[thm]{Lemma}
\newtheorem{prop}[thm]{Proposition}

\newtheorem{defn}[thm]{Definition}

\theoremstyle{remark}

 \def\la{{\langle}}
 \def\ra{{\rangle}}

 \def\a{{\alpha}}
 \def\b{{\beta}}
 \def\g{{\gamma}}

 \def\l{{\lambda}}

 \def\s{\sigma}
 \def\la{{\langle}}
 \def\ra{{\rangle}}

 \def\CD{{\mathcal D}}

 \def\CL{{\mathcal L}}

 \def\CV{{\mathcal V}}
 
 \def\CW{{\mathcal W}}
 \def\BB{{\mathbb B}}

 \def\NN{{\mathbb N}}

 \def\RR{{\mathbb R}}
  
 \def\ZZ{{\mathbb Z}}
      \def\proj{\operatorname{proj}}

\def\lla{\langle{\kern-2.5pt}\langle}      
\def\rra{\rangle{\kern-2.5pt}\rangle}      

\newcommand{\wt}{\widetilde}
\newcommand{\wh}{\widehat}

\def\f{\frac}

\graphicspath{{./}}
\begin{document}

\title[Approximation and orthogonality Sobolev spaces on triangle]
{Approximation and orthogonality in Sobolev spaces on a triangle}

\author{Yuan Xu}
\address{  Department of Mathematics\\ University of Oregon\\
    Eugene, Oregon 97403-1222.}
\email{yuan@math.uoregon.edu}
 
\date{\today}
\thanks{The author was supported in part by NSF Grant DMS-1510296}
 
\begin{abstract}
Approximation by polynomials on a triangle is studied in the Sobolev space $W_2^r$ that consists of functions 
whose derivatives of up to $r$-th order have bounded $L^2$ norm. The first part aims at understanding the 
orthogonal structure in the Sobolev space on the triangle, which requires explicit construction of an inner product 
that involves derivatives and its associated orthogonal polynomials, so that the projection operators of the corresponding 
Fourier orthogonal expansion commute with partial derivatives. The second part establishes the sharp estimate for 
the error of polynomial approximation in $W_2^r$, when $r = 1$ and $r=2$, where the polynomials of approximation are
the partial sums of the Fourier expansions in orthogonal polynomials of the Sobolev space.
\end{abstract}

\maketitle


\section{Introduction} \label{sect1}
\setcounter{equation}{0}

The purpose of this paper is to study approximation by polynomials and orthogonality in the Sobolev space on the 
regular triangle in $\RR^2$.

\subsection{Background}
In order to explain our motivation and our results, we start with a discussion of approximation and orthogonality on
an interval in one variable. 

For a function $f \in L^2[0,1]$, it is well-known that the $n$-th partial sum, $S_n f$, of the Fourier-Legendre expansion 
of $f$ is the $n$-th polynomial of best approximation in the norm $\|\cdot\|$ of $L^2[0,1]$. In other words,
$$
   \|f - S_n f\| = \inf_{P \in \Pi_n} \|f-P\|=:  E_n (f),
$$
where $\Pi_n$ denotes the space of polynomials of degree at most $n$. It is also known that polynomial approximation 
on $[0,1]$ is heavily influenced by the boundary behavior of $f$. For a fixed positive integer $r$, it is easy to see that 
the derivatives of $S_n f$ satisfies 
\begin{equation} \label{eq:1.1}
  \left \| \phi^{k} \left[ f^{(k)}- (S_n f)^{(k)} \right] \right \| \le c n^{-r+k} E_{n-k} (f^{(r)}), \qquad 0 \le k\le r,
\end{equation}
where $\phi(x) = \sqrt{x(1-x)}$. However, for approximation in the Sobolev space $W_2^r[0,1]$ that consists of functions 
whose derivatives of up to $r$th order are all in $L^2[0,1]$, the above error estimate is weaker than what is possible. 
Indeed, we can find a polynomial $p_n$ of degree $n$ such that,  see Section 3 below, 
\begin{equation} \label{eq:1.2}
    \| f^{(k)}- p_n^{(k)} \| \le c n^{-r+k} E_{n-k} (f^{(r)}), \qquad 0 \le k\le r. 
\end{equation}
Furthermore, the polynomial $p_n$ can be chosen as the $n$-th partial 
sum of the Fourier expansion of $f$ in terms of orthogonal polynomials associated with an inner product that contains 
derivatives, and these polynomials are extensions of the Jacobi polynomials $P_n^{(\a,\b)}$ so that $\a = \b = -r$ for $n$ 
large. The orthogonality 
of this type has been studied in various forms in orthogonal polynomials (see, for example, the recent survey 
\cite[Section 6]{MX}) under the name Sobolev orthogonal polynomials, but not yet in approximation or in orthogonal 
series. Estimates closely related to \eqref{eq:1.2} have 
been studied in approximation theory under the name simultaneous approximation (see, for example, \cite{K}), but not
associate to orthogonal expansion in the Sobolev space. Analogue estimates were established 
in the study of spectral methods for numerical solution of differential equations (cf. \cite{CHQZ, CQ, GSW1, GSW2}), in which
the orthogonality of $W_2^r[0,1]$ is either implicitly used by assuming zero boundary conditions, or defined differently, 
and the estimate has the norm of $f^{(r)}$ in place of $E_{n-r}(f^{(r)})$ in the right hand side of \eqref{eq:1.2}. As can be
seen by this brief description, the result of this nature is of fundamental nature and has attracted attentions from several areas. 

We aim at establishing results in the spirit of \eqref{eq:1.2} on regular domains in higher dimension. In a recent work
\cite{LX}, we studied the approximation on the unit ball $\BB^d$ in $\RR^d$ and its application in the spectral theory. 
One of the main tasks there is to construct an appropriate inner product in the Sobolev space $W_r^2(\BB^d)$, involving
derivatives, so that its corresponding orthogonal polynomials {\it formerly} extend the orthogonal polynomials with 
respect to the weight function $(1-\|x\|^2)^{\mu}$ to negative integer in $\mu$ parameter and, more importantly, 
the corresponding partial sum operator for the Fourier orthogonal expansion commutes with the powers of gradient, 
or powers of Laplace operator, up to certain order. This commuting property, together with several new recent results on 
approximation on the sphere and on the unit ball, enables us to establish an analogue of \eqref{eq:1.2} 
on the unit ball. 

The current paper is devoted to establish analogous results on the triangle, which turns out to be substantially 
more complex than the problem on the interval and on the unit ball. Throughout this paper we work with the triangle 
$$
\triangle:=\{(x,y): x \ge 0, y \ge 0, x+y \le 1\}
$$
and all functions appear in the paper are real valued. Classical orthogonal polynomials on the triangle are orthogonal with 
respect to the inner product 
\begin{equation}\label{eq:ipd}
  \la f, g \ra_{\a,\b,\g} := \int_{\triangle} f(x,y)g(x,y) \varpi_{\a,\b,\g}(x,y)dxdy,
\end{equation}
where $\a,\b,\g > -1$ and the weight function $\varpi_{\a,\b,\g}$ is defined by 
\begin{equation}\label{eq:weight}
  \varpi_{\a,\b,\g} (x,y) := x^\a y^\b(1-x-y)^\g.
\end{equation}
Several families of orthogonal polynomials can be given in terms of the Jacobi polynomials $P_m^{(\a,\b)}$ \cite[Section 2.4]{DX}. 
One mutually orthogonal family, in particular, consists of polynomials
\begin{equation}\label{eq:basis}
 P_{k,n}^{\a,\b,\g} (x,y) := (x+y)^k P_k^{(\a,\b)}\left( \frac{y-x}{x+y}\right) P_{n-k}^{(2k+\a+\b+1,\g)}(1-2x-2y)
\end{equation}
for $0 \le k \le n$. With respect to this basis, the  $n$-th partial sum operator, denoted by $S_n^{\a,\b,\g} f$, of the Fourier orthogonal 
expansion of $f \in L^2(\varpi_{\a,\b,\g})$ is defined by 
\begin{equation}\label{eq:Fourier}
 S_n^{\a,\b,\g} f : = \sum_{m=0}^n \sum_{k=0}^m \wh f_{k,m}^{\a,\b,\g}  P_{k,m}^{\a,\b,\g}, \qquad 
      \wh f_{k,m}^{\a,\b,\g} : = \frac{\la f,  P_{k,m}^{\a,\b,\g} \ra_{\a,\b,\g}}{\la P_{k,m}^{\a,\b,\g},  P_{k,m}^{\a,\b,\g} \ra_{\a,\b,\g}}. 
\end{equation}
Thus, $S_n^{\a,\b,\g}f$ is the polynomial of best approximation in $L^2(\varpi_{\a,\b,\g})$. More precisely, the standard Hilbert
space result shows that 
\begin{equation}\label{eq:best-approx}
  E_n(f)_{\a,\b,\g}: =  \inf_{p \in \Pi^2_n} \|f - p\|_{L^2(\varpi_{\a,\b,\g})} = \left \|f -S_n^{\a,\b,\g} f \right\|_{L^2(\varpi_{\a,\b,\g})},
\end{equation}
where $\Pi_n^2$ stands for the space of polynomials of total degree at most $n$ in two variables.

We study approximation in the Sobolev space $W_2^r$ on the triangle. The work contains two parts. In the first part, 
we show that the setup in \eqref{eq:Fourier} can be {\it formerly} extended to $\a,\b,\g$ being $-1$ or $-2$, for which we need 
to define inner products for $W_2^r$, instead of for $L^2$, and construct their orthogonal polynomials by first extending   
$P_{k,n}^{\a,\b,\g}$ to negative parameters. In the second part, we use the construction in the first part to study approximation 
in the Sobolev space. Best approximation by polynomials on the triangle has been studied in \cite{X05} for weighted $L^p$ 
spaces and in \cite{To}, where it is studied for all polytopes. With the partial sum $S_n^{\a,\b,\g} f$ as the polynomial of 
approximation, similar behaviors as in \eqref{eq:1.1} appear in the estimate for approximation of derivatives. Our goal is to 
derive analogue of \eqref{eq:1.2}. 

\subsection{Main results}
To state our results, we need further notations. 
For $\mathbf{m}  \in \NN_0^3$, $\mathbf{m} = (m_1,m_2,m_3)$, define $|\mathbf{m}|:=m_1+m_2+m_3$ and  
$$
 \partial^\mathbf{m} := \partial_1^{m_1}  \partial_2^{m_2}  \partial_3^{m_3}, \quad \hbox{where} \quad \partial_3:=\partial_2 - \partial_1.  
$$
Let $L^2 = L^2(\varpi_{0,0,0})$. For a positive integer $r$, the Sobolev space $W_2^{r}$ is defined by  
\begin{equation*}
  W_2^{r}: = \{ f\in L^2:  \partial^\mathbf{m} f \in L^2,  |\mathbf{m}| \le r\},
\end{equation*}
which is the standard Sobolev space since setting $m_3 =0$ in the definition gives the same space. Let  $\|\cdot\|: = \|\cdot\|_{L^2}$. 
We define the norm of $W_2^{r}$ by 
\begin{equation} \label{eq:S-norm}
 \|f \|_{W_2^{r}}: = \sum_{ |\mathbf{m}| \le r} \left\| \partial^\mathbf{m} f \right \|  
     \sim  \sum_{m_1+m_2 \le r} \left\| \partial_1^{m_1} \partial_2^{m_2} f \right \|,
\end{equation}
where $A \sim B$ means $0< c_1 \le A/B\le c_2$ as usual. We define another Sobolev space on the triangle, in
which we define the error of approximation in terms of $E_n(f)_{\a,\b,\g}$.
 
\begin{defn} \label{def:1.1}
For a positive integer $r$, define 
\begin{equation*} 
  \CW_2^{2r}: = \{ f\in L^2:  \partial^{\mathbf{m}}  f \in L^2\left(\varpi_{r-m_2,r-m_1,r-m_3}\right), 
       \, 0 \le m_i \le r, \, |\mathbf{m}| \le 2r\}. 
\end{equation*}
For $f \in \CW_2^{2r}$ and $\mathbf{m}=(m_1,m_2,m_3)$ with $m_i \in \{0,1,\ldots, r\}$, define
\begin{equation}\label{eq:calEn}
 \mathcal{E}_{n-2r}^{(2r)} (f): = \sum_{\substack{|\mathbf{m}| =2r \\ 0 \le m_i \le r}} E_{n-2r}(\partial^\mathbf{m} f)_{r-m_2,r-m_1,r-m_3}.
\end{equation}
\end{defn}

For example, for $r = 1$, we have 
$$
 \mathcal{E}_{n-2}^{(2)} (f) = E_{n-2}(\partial_1\partial_2 f)_{0,0,1} +  E_{n-2}(\partial_2\partial_3 f)_{0,1,0} +
      E_{n-2}(\partial_3\partial_1 f)_{1,0,0}.
$$
 
We now state our main results on approximation, in which $S_n^{-1,-1,-1}$ and $S_n^{-2,-2,-2}$
are polynomials of degree $n$ constructed in Section \ref{sect7}  and Section \ref{sect10}, respectively, and $c$ is a positive 
constant independent of $n$. 
 
\begin{thm} \label{thm:1.1}
Let $f \in \CW_2^2$ and let $p_n =S_n^{-1,-1,-1} f$. Then 
\begin{equation*} 
  \|f - p_n\| \le \frac{c}{n^2}   \mathcal{E}_{n-2}^{(2)} (f) \quad\hbox{and}\quad  
         \|\partial_i f - \partial_i p_n\| \le \frac{c}{n}\mathcal{E}_{n-2}^{(2)} (f), \quad i=1,2.
\end{equation*}
\end{thm} 

\begin{thm} \label{thm:1.2}
Let $f \in \CW_2^4$ and let $p_n =S_n^{-2,-2,-2} f$. Then
\begin{equation*}
  \|f - p_n\| \le \frac{c}{n^4} \mathcal{E}_{n-4}^{(4)} (f),  \qquad 
      \|\partial_i f - \partial_i p_n\| 
      \le  \frac{c}{n^3}\,\mathcal{E}_{n-4}^{(4)} (f), \quad i=1,2,
\end{equation*}
and 
\begin{equation*} 
 \|\partial_i \partial_j f - \partial_i \partial_j p_n\| 
      \le  \frac{c}{n^2}\,\mathcal{E}_{n-4}^{(4)} (f), \qquad  i, j =1,2.
\end{equation*}
\end{thm} 

For more refined estimates, see Section \ref{sect14} and Section \ref{sect16}. Evidently, we can state the left hand side of 
the estimates in the norm of $W_2^r$. Moreover, by iterating the estimate in Corollary \ref{cor:calE-iterate} in Section 11, 
we have the following corollary: 

\begin{cor} \label{cor:1.3}
Let $f \in \CW_2^{2r}$ for an integer $r \ge s$ with $s = 1$ or $s =2$. For $n \ge 2r$,
\begin{equation} \label{eq:cor1.3}
 \| f- S_n^{-s,-s,-s} f \|_{W_2^s} \le c\, n^{-2 r+s} \mathcal{E}_{n- 2r}^{(2r)} (f). 
\end{equation}
\end{cor}

It is worthwhile to emphasis that our result indicates that $\CW_2^{2r}$ is the right Sobolev space for estimating error of 
approximation in $W_2^s$. As an immediate corollary of \eqref{eq:cor1.3}, we can deduce a weaker result 
$$
 \| f- S_n^{-s,-s,-s} f \|_{W_2^s} \le  c\, n^{-2 r+s}  \sum_{\substack{|\mathbf{m}| =2r \\ 0 \le m_i \le r}} 
   \|\partial^\mathbf{m} f\|_{L^2(\varpi_{r-m_2,r-m_1,r-m_3})},
$$
since $E_n(f)_{\a,\b,\g} \le \|f\|_{L^2(\varpi_{\a,\b,\g})}$, in which we can also replace the norm of $\partial^\mathbf{m} f$
by $\|\partial^\mathbf{m} f\|_{L^2}$ since $\varpi_{\a,\b,\g}(x,y) \le 1$ if $\a,\b,\g \ge 0$ and state a still weaker estimate
$$
    \| f- S_n^{-s,-s,-s}f \|_{W_2^s} \le  c\, n^{-2 r+s}  \sum_{\substack{|\mathbf{m}| =2r \\ 0 \le m_i \le r}} \|\partial^\mathbf{m} f\|.
$$

The estimates in these results are useful for the spectral method for numerical solution of partial differential 
equations. Under the assumption that the function vanishes on the boundary of the triangle, a weaker estimate 
in the case of $s =1$ was established in \cite{LS}, in which $\mathcal{E}_{n- 2r}^{(2r)} (f)$ in the right hand side 
of \eqref{eq:cor1.3} is replaced by the norm of partial derivatives. 

One natural question is why stops at $s=1$ and $2$. The reason lies in the complexity of the structure on the 
triangle, which goes far beyond previous cases. In the beginning, it is not even clear what should be the correct 
formulation of the right hand side of our estimates in Theorems \ref{thm:1.1} and \ref{thm:1.2}. Although our approach 
is simple conceptually, there are many technical obstacles and new phenomena, as outlined in the next subsection. 
Now we have the correct formulation, perhaps the problem could be understood at the next level. 

\subsection{Main steps in the proof}
The proof requires a thorough understanding of orthogonal structure in the Sobolev space on the triangle. The task 
is substantially more difficult than the situation on the interval and on the unit ball. We point out major steps and 
main obstacles in our development below. 

Let $\CV_n(\varpi_{\a,\b,\g})$ be the space of polynomials of degree $n$ that are orthogonal to polynomials of lower 
degrees with respect to $\la \cdot,\cdot\ra_{\a,\b,\g}$. The polynomials $P_{k,n}^{\a,\b,\g}$, $0 \le k \le n$, consist of a 
mutually orthogonal basis for $\CV_n(\varpi_{\a,\b,\g})$; moreover, by the symmetry of $\varpi_{\a,\b,\g}$, we can obtain two 
more mutually orthogonal bases for this space by simultaneously permuting $(x,y,1-x-y)$ and $(\a,\b,\g)$. For approximation 
on the triangle, the main tool is often the reproducing kernel of the space $\CV_n(\varpi_{\a,\b,\g})$, which has a close
form formula independent of a particular choice of bases (see \cite{DaiX, DX, KL, X05}). In our current work, 
however, we need detail knowledges of these polynomial bases, such as their various recurrence relations and 
expressions of their derivatives. 

A complication regarding these bases is that we need to extend these polynomials so that their parameters 
can be negative integers. The usual extension in \cite{Sz} reduces the degree of $P_n^{(\a,\b)}$ when $\a, \b$ are 
negative integers and, as a consequence, unsuitable for our purpose. A new extension was defined and used 
successfully in \cite{LX}. Let us denote by $J_n^{\a,\b}$ this extension and denote by $J_{k,n}^{\a,\b,\g}$ the 
polynomial $P_{k,n}^{\a,\b,\g}$ in \eqref{eq:basis} with the Jacobi polynomials replaced by $J_k^{\a,\b}$. Since 
a family of orthogonal polynomials on the unit ball can be given in terms of spherical harmonics and the Jacobi polynomials 
on the radial variable, in spherical-polar coordinates, the simple structure of the ball reduces the main problem on
the negative parameter to one variable. This is, however, no longer the case for the triangle. What we ultimately 
need is $J_{k,n}^{-r,-r,-r}$ for some positive integer $r$, for which we need numerous explicit properties of 
$J_n^{\a,\b}$. Analytic continuation shows that most properties of $J_n^{\a,\b}$ follow from those of the classical
Jacobi polynomials if $\a$ and $\b$ are not negative integers, but what we need are often singular cases that do 
not follow from analytic continuation, for which a careful analysis is required. 

Taking our cue from the situation on $[0,1]$ and on the unit ball, we need to study the orthogonality of 
$\CV_n(\varpi_{\a,\b,\g}) = \mathrm{span} \{J_{k,n}^{\a,\b,\g}: 0 \le k \le n\}$, when $\a,\b,\g$ are negative integers. 
The orthogonality is defined with respect to an inner product that necessarily involves derivatives. When one or all 
parameters are $-1$, such an orthogonality is studied in \cite{AX} from the point of view that orthogonal polynomials 
are eigenfunctions of the characteristic differential operator. There are, however, many such inner products. What we
need for studying approximation in the Sobolev space is an inner product  for which the projection operators of the 
Fourier orthogonal expansion commute with partial derivatives. To construct such an inner product, we need 
detailed knowledge on $J_{k,n}^{\a,\b,\g}$, in particular, its derivatives. Taking a partial derivative of $J_{k,n}^{\a,\b,\g}$ 
increases two parameters by 1, a new complication that does not appear in the case of orthogonal polynomial on the 
unit ball, and this reflects in the definition of $\CW_2^{2r}$.

On a technical level, the partial derivative $\partial_3$ of $J_{k,n}^{\a,\b,\g}$ takes a simple form, which allows us 
to deduce the commuting relation for $\partial_3$, but not so for the partial derivatives $\partial_1$ and $\partial_2$. 
Instead, we work with two other bases of polynomials obtained from $J_{k,n}^{\a,\b,\g}$ by simultaneously permutation. 
For $\a,\b,\g > -1$, each of the three bases is a mutually orthogonal basis of $\CV_n(\varpi_{\a,\b,\g})$, but the 
mutual orthogonality fails to hold when the parameters are negative integers. This forces us to construct not just 
one but three inner products for each Sobolev space, so that each basis is mutually orthogonal for one of the inner 
products and is orthogonal to polynomials of lower degrees with respect to two other inner products. In addition, 
the three partial sum operators need to be the same, so that the partial sum operators commute with all partial 
derivatives up to order $r$.   

Finally, to prove our estimates, we use orthogonal expansions. Taking, for example, the estimate in Theorem \ref{thm:1.1} 
as an example. The commuting of partial derivatives with the projection operators allows us to express the Fourier 
coefficients of partial derivatives of the function, such as $\wh{\partial_1\partial_3  f}_{k,n}^{1,0,0}$, in terms of of 
$\wh f^{-1,-1,-1}_{k,m}$, which allows us to expand the right hand side of the estimate,  by the Parseval identity, as a 
series that contains only Fourier coefficients $\wh f^{-1,-1,-1}_{k,m}$. Furthermore, we can express the 
orthogonal polynomials with one set of parameters in terms those with another set of parameters, such as writing
$J_{k,n}^{-1,-1,-1}$ in terms of $J_{k,n}^{0,0,0}$, which allows us to write the left hand side of the estimate 
as a series that contains only Fourier coefficients $\wh f^{-1,-1,-1}_{k,m}$. These make it possible to compare the 
two sides and derive the desired estimates. 

The work relies heavily on properties and relations of several families of orthogonal polynomials on the triangle, and 
requires manipulations of a plethora of formulas. We have used a computer algebra system in this work, not only for 
verifying formulas, although all recursive and differential relations in the next three sections and all Sobolev inner 
products are so verified, but also for discovering relations and searching for appropriate Sobolev inner products. In fact,
it is hard to image that the work could be carried out without using a computer algebra system. 

\subsection{Organization of the paper}
We start with the extended Jacobi polynomials, $J_n^{\a,\b}$, for which $\a$ and $\b$ can be less than $-1$.
In the next section, we establish relations and identities among these polynomials, for which particular 
attention will be paid to the singular cases that do not follow from analytic continuation. In Section 3, we study the 
orthogonality of $J_n^{-\ell,-m}$ and show how this allows us to establish estimates in \eqref{eq:1.2}. Although
related results have been studied in several areas, as discussed at the end of the Subsection 1.1, we include this 
section since our results, based on a new extension of the Jacobi polynomials with negative parameters, do not 
seem to be exactly the same, and it is a prelude of our study in two variables. 

We develop properties of the family of polynomials $J_{k,n}^{\a,\b,\g}$ and two other families derived from 
simultaneous permutations in Section 4. These include recursive relations for polynomials with different sets of 
parameters, derivatives, and differential operators, all in explicit forms. Many of these relations are developed for 
the first time, which will be useful in future studies. 

The orthogonality of the extended Jacobi polynomials on the triangle is discussed in Sections 5--10. Indexed by the 
parameters $(\a,\b,\g)$, we start from $(\a,\b,-1)$ and its permutations in Section 5, followed by $(-1,-1,\g)$ in Section 6, 
$(-1,-1,-1)$ in Section 7, $(-2,\a,\b)$ in Section 8, $(-2,-1,-1)$ in Section 9 and, finally, $(-2,-2,-2)$ in Section 10. The goal 
is to choose an inner product that has an orthonormal basis in $\CV_n = \mathrm{span} \{J_{k,n}^{\a,\b,\g}: 0\le k \le n\}$, 
so that its corresponding orthogonal projection commutes with certain partial derivatives. Our development is based on 
the realization that inner products for the Sobolev space can be defined somewhat inductively according to the parameters. 
The process is still complicated, since we often have to take linear combinations of polynomials in $\CV_n$ and 
modify the inner product at the same time. Furthermore, in each case, we need to define three inner products, for the 
family of $J_{k,n}^{\a,\b,\g}$ and two other families derived from permutation, respectively. 

We start the proof of approximation part of the results in Section 11, where we prove a result for best approximation 
in $L^2$ norm in the classical setting of $\a,\b,\g > -1$, which provides a prelude for the proof in latter sections. 
The approximation behavior of $S_n^{\a,\b,\g} f$ is studied in Section 12 when $\g =-1$, in Section 13 for
$\a=\b=-1$, and in Section 14 for $\a=\b=-1$ and $\g =-2$. These three sections provide intermediate steps that
are needed for the estimates of derivatives in our main results. The case $\a=\b=\g =-1$ is treated in Section 14,
which gives the proof of Theorem \ref{thm:1.1}. Finally, the case $\a=\b=\g =-2$ is considered in Section 16, which
gives the proof of Theorem \ref{thm:1.2}. 

\section{Jacobi polynomials of negative parameters}
\label{sect2}
\setcounter{equation}{0}

For $\a,\b>-1$, the Jacobi polynomials are defined by 
$$
  P_n^{(\a,\b)}(t) = \frac{(\a+1)_n}{n!} {}_2F_1 \left(-n, n+\a + \b+1; \a+1; \frac{1-t}{2} \right)
$$
in hypergeometric function ${}_2F_1$. They are  orthogonal to each other with respect to the weight function 
$
   w^{\a,\b}(t) := (1-t)^\a (1+t)^\b
$
on $[-1,1]$, and 
\begin{align} \label{eq:normJacobi}
   h_n^{\a,\b}:= & \ \int_{-1}^1 \left[ P^{(\alpha,\beta)}_{n}(t) \right]^2 w^{\a,\b}(t) dt \\
        = &\  \frac{2^{\a+\b+1}}{2n+\a+\b+1} \frac{\Gamma(n+\a+1)\Gamma(n+\b+1)}{n! \Gamma(n+\a+\b+1)}. \notag
\end{align}
As it is shown in \cite{Sz}, writing the ${}_2F_1$ in the following explicit form 
\begin{align} \label{eq:JacobiP}
   P^{(\alpha,\beta)}_{n}(t) 
   =& \sum_{k=0}^n  \frac{(k+\a+1)_{n-k} \, (n+\alpha+\beta+1)_{k} }{ (n-k)!\, k! } \left(\frac{t-1}{2}\right)^k
\end{align}
extends the definition of $P_n^{(\a,\b)}(t)$ for $\a$ and $\b$ being any real numbers. However, a reduction of the
polynomial degree occurs if $-n-\a-\b\in\{1,2,\dots,n\}$ in such an extension. To avoid this degree reduction, we 
renormalized the Jacobi polynomials in \cite{LX} and provided another definition that works for all parameters. More 
precisely, the following definition was given in \cite{LX}: 
 
\begin{defn}
For $\a,\b \in \RR$ and $n=1,2,\ldots$, define
\begin{align} \label{eq:Jab}
& J^{\alpha,\beta}_{n}(t):=
\sum_{k= \iota_0}^n \frac{(k+\a+1)_{n-k}  }{ (n-k)!\, k!  (n+\alpha+\beta+k+1)_{n-k} } \left(\frac{t-1}{2}\right)^k,   \quad n\in \NN_0,
\end{align}
where $\iota_0 = \iota^{\a,\b}_0(n):= -n-\a-\b$ if $-n-\a-\b\in\{ 1,2,\dots,n\}$ and $\iota_0 = 0$ otherwise.
\end{defn}

For convenience, we also define ${P}^{(\alpha,\beta)}_{n}(t)=J^{\alpha,\beta}_{n}(t)=0$ whenever $n$ is a negative integer.
This extends the definition of the Jacobi polynomials to all $\a,\b  \in \RR$, in particular, to $\a, \b$ being negative integers. 
It is easy to see that $J_n^{\a,\b}$ is a polynomial of degree $n$ for all $\a,\b$. For $\a,\b > -1$, these are just renormalization 
of the ordinary Jacobi polynomials, as seen below. 

\begin{prop} \label{prop:Jacobi1st}
If $\iota_0^{\a,\b}(n) =0$, then 
\begin{equation} \label{JacobiP2}
J^{\alpha,\beta}_{n}(t) =  \frac{1 }{(n+\a+\beta+1)_{n}}{P}^{(\alpha,\beta)}_{n}(t).
\end{equation}
Furthermore, $J_n^{\a,\b}$ is an analytic function of $\a,\b$ if $-n - \a-\b \notin \{1,2,\ldots,n\}$. 
\end{prop}

\begin{proof}
If $\iota_0^{\a,\b}(n) =0$, then $-n-\a-\b < 1$ or $-n-\a-\b > n$, so that $(n+\a+\b+1)_n \ne 0$. Consequently, 
the identity 
$$
     \frac{(n+\a+\beta+1)_n}{(n+\alpha+\beta+k+1)_{n-k}} = (n+\alpha+\beta+1)_{k} 
$$ 
establishes \eqref{JacobiP2}. Furthermore, if $-n - \a-\b \notin \{1,2,\ldots,n\}$, then $\iota_0^{\a,\b}(n) =0$ and $J_n^{\a,\b}$ 
depends continuously  on $\a,\b$ by the explicit formula in \eqref{eq:Jab}.
\end{proof}

As a consequence, many identities of the ordinary Jacobi polynomials extend to $J_n^{\a,\b}$ immediately. For example, 
\begin{equation} \label{eq:Jn(-t)}
 J_n^{\a,\b}(-t) = (-1)^n J_n^{\b,\a}(t), \qquad \hbox{if} \quad -n - \a-\b \notin \{1,2,\ldots,n\}.
\end{equation}
This identity, however, fails to hold if $\iota_0^{\a,\b}(n) \ne 0$. We should point out that $J_n^{\a,\b}$ may not be continuous as 
a function of $(\a,\b)$ when $\iota_0^{\a,\b}(n) \ne 0$. For example, $\iota_0^{-2,0}(1) = 1$, and $J_1^{-2,0}(t) = -\frac12 (1-t)$ does
not equal to the limit of $J_1^{-2,\b}(t) =  - \f12 (1 +\frac2 {\b} - t)$ as $\b \to 0$. 

The three-term relation satisfied by the Jacobi polynomials takes the form
\begin{equation}\label{eq:three-termJ}
 t J_n^{\a,\b}(t) = 2 (n + 1) J_{n + 1}^{\a, \b}(t) + \mathfrak{a}_n^{\a,\b}J_n^{\a,\b}(t) + \mathfrak{b}_n^{\a,\b}J_{n-1}^{\a,\b}(t), 
\end{equation}
when writing in $J_n^{\a,\b}$, where 
\begin{align*}
\mathfrak{a}_n^{\a,\b} & := \frac{\b^2 - \a^2}{(2 n + \a + \b) (2 n + \a + \b + 2)}, \\ 
\mathfrak{b}_n^{\a,\b} & := \frac{2 (n + \a) (n +\b) (n + \a + \b)}{(2 n + \a + \b -1) (2 n + \a + \b)^2 (2 n + \a + \b + 1)}.
\end{align*}
 
In the following, we will state a number of properties of these extended Jacobi polynomials. Although analytic continuation
can be used to extend most of these identities from ordinary Jacobi polynomials to almost all parameters 
in $\RR$, we often need to consider exceptional cases of $-n -\a - \b \in \{1,2,\ldots, n\}$. Hence, it is necessary to
give a proof for these identities. 

\begin{prop}
For $\a,\b \in \RR$ and $n \ge 0$, 
\begin{align} \label{derivativeJ}
  \frac{d}{dt} J_n^{\a,\b}(t) = \frac12 J_{n-1}^{\a+1,\b+1}(t). 
\end{align}
\end{prop}

\begin{proof}
If $\iota_0^{\a,\b}(n) =0$, this follows from \eqref{JacobiP2} and the classical result on the derivative of $P_n^{(\a,\b)}$. 
Assume now $\iota_0^{\a,\b}(n) \ne 0$. Then $\iota_0^{\a,\b}(n) = -n - \a -\b$ is an integer in $\{1,2,\ldots, n\}$. A quick check 
shows that $\iota_0^{\a+1,\b+1}(n-1) = \iota_0^{\a,\b}(n) -1$, so that $J_n^{\a,\b}$ and $J_{n-1}^{\a+1,\b+1}$ contain 
the same number of monomial terms.  Taking derivative of $J_n^{\a,\b}$ in \eqref{eq:Jab} and then shifting the summation 
index by 1, so that the summation starts from $\iota_0^{\a+1,\b+1}(n-1)$, it is easy to verify \eqref{derivativeJ} by working
with the hypergeometric expansions.
\end{proof}

\begin{prop} \label{prop:J-mJm}
For $\a  \in \RR$, $m =0,1,2,\ldots$, 
\begin{align} \label{eq:Jam}
J_n^{\a,-m}(t) = \frac{(n-m)!}{2^m n!} (1+t)^m J_{n-m}^{\a,m}(t), \quad n \ge m + \max\{0,\left \lfloor -\a \right \rfloor\}. 
\end{align}
Furthermore, for $\b \in \RR$ and $\ell= 0,1,2, ...$, 
\begin{align}\label{eq:Jmb}
J_n^{-\ell,\b}(t)  = \frac{(-1)^\ell (n-\ell)!}{2^\ell n!} (1-t)^\ell J_{n-\ell}^{\ell,\b}(t), \quad n \ge \ell.
\end{align}
For $m, \ell = 0,1,2,\ldots,$ and $n \ge m+\ell$, 
\begin{align}\label{eq:Jlm}
  J_n^{-\ell,-m}(t) = \frac{(-1)^\ell (n-m-\ell)!}{2^{\ell+m} n!} (1-t)^\ell(1+t)^m J_{n-\ell-m}^{\ell,m}(t).
\end{align}
\end{prop}

\begin{proof}
If $\a > -1$, then $-n-\a +m < -n + m +1 \le 1$ for $n \ge m$, which implies that $\iota_0^{\a,-m}(n) =0$.
If $\a \le -1$, then $-n-\a +m < 1$ if $n \ge m + \lfloor -\a \rfloor$, which again implies that $\iota_0^{\a,-m}(n) =0$.
Consequently, \eqref{eq:Jam} follows from \eqref{eq:Jab} and property of the ordinary Jacobi polynomials 
\cite{Sz}. 

The proof of \eqref{eq:Jmb} requires more work. For $\b \in \RR$, the expansion $J_n^{-\ell,\b}$
in \eqref{eq:Jab} contains a factor $(k-\ell+1)_{n-k}$ in the numerator, which is equal to 0 if $k < \ell$. 
If $\iota_0^{-\ell,\b}(n) \le \ell$, then the summation starts from $\ell$, so that a shift in $k$ leads to 
$$
  J_n^{-\ell,\b}(t) = \sum_{k=0}^{n-\ell} \frac{(k+1)_{n-k-\ell}}{(k+\ell)! (n-k-\ell)! (n+\b+k+1)_{n-k-\ell}}
     \left(\frac{t-1}2\right)^{k+\ell}.
$$
Furthermore, in this case $-n+\ell - \b \le \ell$ is equivalent to $-n - \b < 0$, so that $\iota_0^{\ell,\b}(n-\ell) =0$. 
Rewriting the Pochhammer symbols in the above expression in the form of $(a)_k$, it is easy to see that
\eqref{eq:Jmb} holds in this case. In the remaining case of $\iota_0^{-\ell,\b}(n) > \ell$, we see that $\ell < -n+\ell - \b \le n$ is
equivalent to $\ell < -n - \b \le n-\ell$, which shows immediately that $\iota_0^{\ell,\b} (n-\ell)= \iota_0^{-\ell,\b}(n) - \ell$.  
Hence, shifting the summation index by $\ell$ in the expression of $ J_n^{-\ell,\b}$ in \eqref{eq:Jab}, we can 
verify \eqref{eq:Jmb} as in the first case. 

Finally, \eqref{eq:Jlm} follows from applying \eqref{eq:Jmb} first and then applying \eqref{eq:Jam}.
\end{proof}

It is worthwhile to note that \eqref{eq:Jam} does not hold if $\a <0$ and $m < n < m + \lfloor - \a \rfloor$, whereas 
\eqref{eq:Jmb} holds, in contrast, for all $\b \in \RR$ and $n \ge m$. 

There are numerous identities for the ordinary Jacobi polynomials and all can be generalized to our extended polynomials. We 
shall state a few that will be needed for studying derivatives of the Jacobi polynomials on the triangle. 

\begin{prop} \label{prop:tau}
For $\a,\b \in \RR$ and $n \in \NN$, define
\begin{equation} \label{eq:tau}
 \tau_{n}^{\a,\b}: =   \frac{n+\b}{(2n+\a+\b)(2n+\a+\b+1)},   
\end{equation}
if $2n+\a+\b \ne 0$ and $2n+\a+\b +1 \ne 0$, and define
$$
  \tau_n^{-n,-n} := 1 \quad \hbox{and}\quad  \tau_n^{-n-1,-n}:= -1. 
$$
Then the relations 
\begin{align}\label{eq:JacobiR1}
\begin{split}
 J_n^{\a,\b}(t) & =  J_n^{\a+1,\b}(t) - \tau_{n}^{\a,\b} J_{n-1}^{\a+1,\b}(t), \\ 
 J_n^{\a,\b}(t) & =  J_n^{\a,\b+1}(t) +  \tau_{n}^{\b,\a} J_{n-1}^{\a,\b+1}(t), 
\end{split}
\end{align}
hold if either one of the following conditions satisfies: 
\begin{enumerate}
\item $2n+\a+\b\ne 0$, $2n+\a+\b+1\ne 0$, and either $n+\a+\b+1 \ne 0$ or $n+\a+\b+1=0$ but $\a +1 \in \{-1,-2,\ldots, -n\}$. 
\item $\b = -n$, $\a =-n$ or $\a = -n-1$. 
\end{enumerate}
\end{prop}

\begin{proof}
If $\a,\b > -1$, these are classical identities for the Jacobi polynomials (cf. \cite[(22.7.18)]{AS}). For general $\a,\b$, we consider
the first identity, which holds, by Proposition \ref{prop:Jacobi1st},  if $\iota_0^{\a,\b}(n) = \iota_0^{\a+1,\b}(n) = 
\iota_0^{\a+1,\b}(n-1)$. These three indexes are not equal if $-n -\a-\b = 1$, $n$ and $n+1$.
If $2n +\a+\b \ne 0$ and $2n +\a+\b +1\ne 0$ but $n+\a+\b+1=0$, then  $\iota_0^{\a,\b}(n) =  \iota_0^{\a+1,\b}(n-1) =1$
by $\iota_0^{\a+1,\b}(n) =0$. For the identity to hold, we require the constant term of $J_n^{\a+1,\b}$ to be zero. 
This term is given by $(\a+2)_n /n!^2$, which is zero if and only if $\a+1 =-1,-2,\ldots,-n$. This proves (1). If 
$2n+\a+\b =0$ or $2n+\a+\b=0$ but $n+\b \ne 0$, then $\tau_n^{\a,\b}$ is undefined and \eqref{eq:JacobiR1}
cannot hold. If $2n + \a+\b = 0$ and $n+\b =0$, then $\a=\b=-n$, $\iota_0^{-n,-n}(n) = n$, $\iota_0^{-n+1,-n}(n) =n-1$
and $\iota_0^{-n+1,-n}(n-1) =0$. However, it is easy to verify directly that $ J_n^{-n,-n}(x) = \frac{1}{n!} \left(\frac{t-1}2\right)^n$,
$J_{n-1}^{-n+1,-n}(x) =  \frac{1}{(n-1)!} \left(\frac{t-1}2\right)^{n-1}$, and $
J_n^{-n,-n}(x) = \frac{1}{n!} \left(\frac{t-1}2\right)^n+ \frac{1}{(n-1)!} \left(\frac{t-1}2\right)^{n-1}$, which shows that the
first identity in \eqref{eq:JacobiR1} holds if $\a= \b = -n$. Similarly, if $2n+\a+\b =0$ and $n+\b =0$, then $\a = -n-1$, $\b =-n$, 
and $\iota_0^{-n-1,-n}(n) = 0$, $\iota_0^{-n,-n}(n) =n$ and $\iota_0^{-n,-n}(n-1) =0$. A direct computation shows that 
$$
  J_n^{-n-1,-n}(x) =\sum_{k=0}^n \frac1 {k!(n-k)!} \left(\frac{t-1}2\right)^k, \quad J_{n-1}^{-n,-n}(x) = 
   \sum_{k=0}^{n-1} \frac1 {k!(n-k)!} \left(\frac{t-1}2\right)^k,
$$
so that the first identity in \eqref{eq:JacobiR1} follows from a direct verification. This completes the proof. 
\end{proof}

\begin{prop}
For $\a,\b \in \RR$ and $n \in \NN$, 
\begin{equation}\label{eq:JacobiR2}
  \frac{1-t}{2} J_{n-1}^{\a+1,\b}(t) = \a J_{n}^{\a,\b-1}(t) - (n+\a)J_{n}^{\a-1,\b}(t). 
\end{equation}
\end{prop}

\begin{proof}
If $\a,\b > -1$, the identity can be verified directly by working with the hypergeometric expansions,
which also holds if $\iota_0^{\a+1,\b}(n-1)=\iota_0^{\a,\b-1}(n)=\iota_0^{\a-1,\b}(n)$. That these three 
identities are not equal only if $n+\a+\b =0$, for which $\iota_0^{\a+1,\b}(n-1) =0$ and the other two 
indexes are equal to 1. However, there is a shift of index due to the multiple of $(1-t)/2$ in the left hand side,
from which it follows easily that the identity holds in this case as well. 
\end{proof}

\begin{prop}
For $\a,\b \in \RR$ and $n \in \NN$, 
\begin{align}\label{eq:JacobiR3}
\begin{split}
  \frac{1+t}{2} J_{n-1}^{\a+1,\b+1}(t) & = n J_{n}^{\a,\b}(t) + \frac{n+\b}{2n+\a+\b}J_{n-1}^{\a+1,\b}(t), \\ 
  \frac{1-t}{2} J_{n-1}^{\a+1,\b+1}(t) & = -n J_{n}^{\a,\b}(t) + \frac{n+\a}{2n+\a+\b}J_{n-1}^{\a,\b+1}(t), 
\end{split}
\end{align}
where the first identity holds if $2n+\a+\b \ne 0$ and either $n+\a+\b +1 \ne 0$ or $n+\a+\b+1=0$ and 
$\a + 1 = -1,-2,\ldots, -n$, whereas the second identity holds under the same assumption with $\a$
and $\b$ interchanged. Furthermore, if $2n + \a+ \b =0$ and $n+\b =0$, then 
\begin{align}\label{eq:JacobiR3b} 
\begin{split}
 \frac{1+t}{2} J_{n-1}^{-n+1,-n+1}(t) & = n J_{n}^{-n,-n}(t) + J_{n-1}^{-n+1,-n}(t), \\
  \frac{1-t}{2} J_{n-1}^{-n+1,-n+1}(t) & = - n J_{n}^{-n,-n}(t) + J_{n-1}^{-n,-n+1}(t). 
\end{split}
\end{align}
\end{prop}

\begin{proof}
If the $\iota_0$ indexes of the three polynomials are equal, then the identity can be verified directly by working with 
the hypergeometric expansions. In general, we write $(1+t)/2 = 1+ (t-1)/2$ and comparing the coefficients of
$(t-1)^j$ in two sides of the identity. The three indexes are not equal if $n+\a+\b+1 \ne 0$, or $2n+\a+\b \ne 0$. 
If $n+\a+\b +1 = 0$, then $\iota_0^{\a+1,\b+1}(n-1)=0$ and the other two indexes are equal to 1. We need the 
constant term of $J_{n-1}^{\a+1,\b+1}$ to be zero, which holds if $(\a+2)_{n-1}=0$,  or $\a +1\ne -1, -2, \ldots, -n$.
If $2n+\a+\b=0$, then we need $n+\b =0$, which leads to $\a = -n$ and $\b=-n$. In this case, the identity
\eqref{eq:JacobiR3b} can be verified directly. 
\end{proof}

\section{Sobolev inner product and its orthogonal basis on an interval}
\label{sect3}
\setcounter{equation}{0}

Let $\la \cdot,\cdot\ra$ be an inner product in a Sobolev space that contains the space of polynomials of one real 
variable. Let $\{Q_n\}_{n\ge 0}$ be the sequence of orthogonal polynomials with respect to this inner product, with 
$Q_n$ being a polynomial of degree precisely $n$. Then the Fourier orthogonal expansion with respect to this basis 
is defined by 
$$
  f(t) = \sum_{n=0}^\infty \wh f_n \, Q_n(t), \quad \hbox{where}\quad \wh f_n := \frac{\la f,Q_n\ra}{\la Q_n,Q_n\ra}. 
$$
The $n$-th partial sum of this expansion is defined by 
$$
  S_n f(t) =  S_n(f;t) := \sum_{k=0}^n \wh f_k \, Q_k(t), \qquad n=0,1,2,\ldots.
$$
In the following, we will consider an inner product $\la \cdot,\cdot\ra_{\a,\b}$ with $\a,\b$ being two parameters, for which 
$J_n^{\a,\b}$ are orthogonal when $\a, \b$ are negative integers. This is possible for $n \ge n_0$ for some $n_0$ 
depending on $\a$ and $\b$, but require modification of polynomials for $n< n_0$. We will adopt the convention of 
using $\wh J_k^{\a,\b}$ and $S_n^{\a,\b}(f)$ to denote the orthogonal polynomials and the $n$-th partial sum of the 
orthogonal expansion.  

Denote the derivative by $\partial$ and let $\partial^k f:= f^{(k)}$ for $k=0,1,2,\ldots$. For $m \in \NN_0$, let 
$$
W_2^m: = \{f \in L^2[0,1]: \partial^k f \in L^2[0,1], \, \, k=1,2,\ldots, m\}. 
$$

It should be mentioned that the inner products that we construct below are termed as discrete-continuous Sobolev inner
product. It was considered in \cite{APPR} when one parameter is a negative integer and in \cite{AMR} when both parameters
are negative integers; see also the discussion in \cite[Section 6]{MX}. Because of our modification of the Jacobi polynomials, 
we need to study them below. 

\subsection{Sobolev Inner product for $J_n^{\a,-m}$ and $J_n^{-m,\b}$}

Let $\a > -m-1$ and $m =1,2,\ldots$. For $\l_k > 0$, $k =0,1,\ldots, m-1$, define $\a^\sharp: = \max\{0, \lfloor - \a \rfloor\}$ and
\begin{equation}\label{eq:ipd-a-m}
 \la f,g\ra_{\a,-m} := \int_0^1 f^{(m)}(t)g^{(m)}(t) (1-t)^{\a+m} dt + \sum_{k=0}^{m -1} \l_k f^{(k)}(0) g^{(k)}(0).
\end{equation}
For $n =0,1,\ldots,$ we define the polynomials $\wh J_n^{\a,-m}$ by 
$$
  \wh J_n^{\a,-m}(t) = \begin{cases} J_n^{\a,-m} (2t-1), &  n \ge m+\a^\sharp, \\
      \displaystyle{ J_n^{\a,-m} (2t-1) - \sum_{k=0}^{\min\{n,m\}-1} J_{n-k}^{\a+k,-m+k} (-1) \frac{t^k}{k!}}, & n < m+\a^\sharp.
      \end{cases}
$$ 
We shall prove that $\wh J_n^{\a,-m}$ consists of a mutually orthogonal basis and we denote by $S_n^{\a,-m}$ the $n$-th 
partial sum operator of the Fourier orthogonal expansion with respect to this basis. 

\begin{prop} \label{prop:Qam}
Let $m=1,2,\ldots$ and $\a > -m-1$, the polynomials $\{\wh J_n^{\a,-m}: n=0,1, 2, \ldots\}$ are mutually orthogonal with 
respect to the inner product $\la \cdot,\cdot \ra_{\a,-m}$. Furthermore, for $n \ge k$,
\begin{equation} \label{eq:dS=Sam}
  \partial^k S_n^{\a,-m} (f) = S_{n-k}^{\a+k,-m+k}  (\partial^k f), \qquad k =1,\ldots, m. 
\end{equation}
\end{prop} 

\begin{proof}
For $n \ge m$, we claim that $\partial^k \wh J_n^{\a, -m}(0) =0$ for $0 \le k \le m-1$. Indeed, if $n \ge m+\a^\sharp$, this 
follows readily from \eqref{eq:Jam} whereas if $m < n < m+ \a^\sharp$, this follows from the definition of $\wh J_n^{\a,-m}$ 
and \eqref{derivativeJ}. Hence, for $n \ge m$, we only need to consider the integral part of $\la f, \wh J_n^{\a, -m}\ra_{\a,-m}$.
Consequently, applying \eqref{derivativeJ} shows that 
\begin{equation} \label{eq:fQam}
  \la f, \wh J_n^{\a, -m}\ra_{\a,-m} =  \la f', \wh J_{n-1}^{\a+1, -m+1}\ra_{\a+1,-m+1} = \cdots =  
    \la f^{(m)}, \wh J_{n-m}^{\a+m,0}\ra_{\a+m,0},  
\end{equation}
where the last inner product is the regular one for $L^2(\varpi_{\a+m,0})$ since $\a + m> -1$. Since $\wh J_{n-m}^{\a+m, 0}$
is, up to a multiple constant, the ordinary Jacobi polynomial, the orthogonality of $\wh J_n^{\a,-m}$ follows immediately from
\eqref{eq:fQam}. If $ n < m$, the definition of $\wh J_n^{\a,-m}$ implies that $\partial^k \wh J_n^{\a,-m}(0) = \delta_{k,n}$ 
for $0\le k\le n-1$ and, since $\partial \wh J_n^{\a,-m} = \wh J_{n-1}^{\a+1,-m+1}$ by \eqref{derivativeJ}, the same argument 
shows that, for $n < m+\a^\sharp$,   $\la f, \wh J_n^{\a, -m}\ra_{\a,-m} =   \cdots =  \la f^{(n)}, \wh J_{0}^{\a+n, -m+n}\ra_{\a+n,-m+n},$ 
from which the orthogonality $\wh J_n^{\a,-m}$ follows immediately. 

Setting $f =  \wh J_n^{\a,-m}$ in the relation \eqref{eq:fQam} also shows that 
$$
      \la \wh J_n^{\a,-m},\wh J_n^{\a, -m}\ra_{\a,-m} 
       = \cdots =   \la\wh J_{n-m}^{\a+m,0}, \wh J_{n-m}^{\a+m,0}\ra_{\a+m,0},
$$ 
which implies, combining with \eqref{eq:fQam} that $\wh f_n^{\a,-m} = \wh{\partial f}_{n-1}^{\a+1,-m+1}$. Consequently, it 
follows immediately from \eqref{derivativeJ} that $\partial S_n^{\a,-m} (f) = S_{n-1}^{\a+1,-m+1} (\partial f)$. This is 
\eqref{eq:dS=Sam} when $k=1$ and its iteration proves the general case. 
\end{proof}

Similarly, we can define the Sobolev inner product for the parameter $(-\ell,\b)$. Let $\b > -\ell-1$ and $\ell =1,2,\ldots$. 
For $\l_k > 0$, $k =0,1,\ldots, \ell-1$, define
\begin{equation}\label{eq:ipd-ell-b}
 \la f,g\ra_{-\ell,\b} := \int_0^1 f^{(\ell)}(t)g^{(\ell)}(t) t^{\b+\ell} dt + \sum_{k=0}^{\ell-1} \l_k f^{(k)}(1) g^{(k)}(1).
\end{equation}
For $n =0,1,\ldots,$ we define the polynomials $\wh J_n^{-\ell,\b}$ by 
$$
  \wh J_n^{-\ell,\b}(t) = \begin{cases} J_n^{-\ell,\b} (2t-1), &  n \ge \ell, \\
      \displaystyle{  J_n^{-\ell,\b} (2t-1) - \sum_{k=0}^{n-1} J_{n-k}^{-\ell+k,\b+k} (1) \frac{(t-1)^k}{k!}}, & n < \ell.
            \end{cases}
$$ 
This definition is simpler than that of $\wh J_n^{\a,-m}$ since \eqref{eq:Jmb} is less restrictive on $n$ than that of
\eqref{eq:Jam}. Similar to Proposition \ref{prop:Qam}, we can prove the following proposition:

\begin{prop}\label{prop:Qlb}
For $\ell = 1,2,\ldots$ and $\b > -\ell -1$, the polynomials $\{\wh J_n^{-\ell,\b}: n=0,1, 2, \ldots\}$ are mutually orthogonal 
with respect to the inner product $\la \cdot,\cdot \ra_{-\ell,\b}$. Furthermore, for $n \ge k$,
\begin{equation}\label{eq:Qlb}
  \partial^k S_n^{-\ell,\b} (f) = S_{n-k}^{-\ell+k,\b+k}  (\partial^k f), \qquad k =1,\ldots, \ell. 
\end{equation}
\end{prop} 
  
\begin{proof}
If $n \ge \ell$, then $\partial^k \wh J_n^{-\ell, \b}(1) = 0$ for $0 \le k \le \ell-1$ by \eqref{eq:Jmb} and 
$\partial^k \wh J_n^{-\ell,\b}(t) =  \wh J_{n-k}^{-\ell+k,k+\b}(t)$ by \eqref{derivativeJ}. Consequently, 
$$
 \la f, \wh J_n^{-\ell,\b} \ra_{-\ell,\b} = \la f', \wh J_{n-1}^{-\ell+1,\b+1} \ra_{-\ell+1,\b+1} 
        =\cdots =  \la \partial^\ell f, \wh J_{n-\ell}^{0, \b+\ell} \ra_{0, \b+\ell},
$$
from which the orthogonality follows readily. If $n <\ell$, we only need to consider the discrete part in the inner product
$ \la f, \wh J_n^{-\ell,\b} \ra_{-\ell,\b}$, and our definition of $\wh J_n^{-\ell,\b}$ gives orthogonality immediately. Furthermore,
it is easy to verify that $\partial \wh J_n^{-\ell,\b}(t) = \wh J_{n-1}^{-\ell+1,\b+1}(t)$, from which \eqref{eq:Qlb} follows. 
\end{proof}
 
It should be pointed out that, together, the two propositions cover the case of $(-\ell, -m)$ for all nonnegative integers
$\ell$ and $m$. Let us also note that in the case of $J_n^{-m,-m}$, we end up with two different constructions and two
inner products.

\subsection{Inner product for $J_n^{\a,\b}(1-2t)$ and further discussions}
Even though $J_n^{\a,\b}(t) = (-1)^n J_n^{\b,\a}(-t)$ does not always hold when $\a$ and $\b$ are nonnegative integers, 
inner products for $J_n^{\a,\b}(1-2t)$ can be derived, at lease conceptually, from those for $J_n^{\a,\b}(1-2t)$ by a 
simple change of variables which requires, however, that the point evaluations in the inner product will need to be 
evaluated at the point with $0$ and $1$ exchanged, and so does the Taylor expansion in the modification $\wh J_n^{\a,\b}$. 
It turns out that we do not obtain new inner products and only need to modify orthogonal polynomials. 

Let us start with $J_n^{\a,-m}(1-2t)$. For $n =0,1,\ldots,$ we define $\wt J_n^{\a,-m}$ by 
$$
  \wt J_n^{\a,-m}(t) = \begin{cases} (-1)^n J_n^{\a,-m} (1-2t), &  n \ge m + \a^\sharp, \\
  \displaystyle{ (-1)^n J_n^{\a,-m} (1-2t) - \sum_{k=0}^{n-1} (-1)^{n-k}J_{n-k}^{\a+k,-m+k} (-1) \frac{(t-1)^k}{k!}}, & n <m + \a^\sharp.
            \end{cases}
$$ 
As in Proposition \ref{prop:Qam}, we can prove the following proposition:

\begin{prop} \label{prop:Qam2}
For $m = 1,2,\ldots$ and $\a > - m -1$, the polynomials $\wt J_n^{\a,-m}$, $n=0,1, 2, \ldots$, are mutually orthogonal 
with respect to the inner product $\la \cdot,\cdot \ra_{-m,\a}$ defined in \eqref{eq:ipd-ell-b}. Furthermore, denote again
by $S_n^{\a,-m} f$ the corresponding $n$-th partial sum, then, for $n \ge k$,
\begin{equation}\label{eq:Qam2}
  \partial^k S_n^{\a,-m} (f) = S_{n-k}^{\a+k,-m+k}  (\partial^k f), \qquad k =1,\ldots, m. 
\end{equation}
\end{prop} 
 
Similarly,  for $J_n^{-\ell, \b}(1-2t)$, we define the polynomials $\wh J_n^{-\ell,\b}$ by 
$$
  \wt J_n^{-\ell,\b}(t) =  \begin{cases} 
     (-1)^n J_n^{-\ell,\b} (1-2t), &  n \ge \ell, \\
  \displaystyle{ (-1)^n J_n^{-\ell,\b} (1-2t) - \sum_{k=0}^{n-1} (-1)^{n-k}J_{n-k}^{-\ell+k,\b+k} (1) \frac{t^k}{k!}}, & n \le  \ell.
            \end{cases}
$$
We can then prove, as in Proposition \ref{prop:Qam}, the following proposition. 

\begin{prop} \label{prop:Qlb2}
For $\ell = 1,2,\ldots$ and $\b > - \ell -1$, the polynomials $\wt J_n^{-\ell,\b}$, $n=0,1, 2, \ldots$, are mutually orthogonal 
with respect to the inner product $\la \cdot,\cdot \ra_{\b,-\ell}$ defined in \eqref{eq:ipd-a-m}. Furthermore, denote again
by $S_n^{-\ell,\b} f$ the corresponding $n$-th partial sum, then, for $n \ge k$,
\begin{equation}\label{eq:Qlb2}
  \partial^k S_n^{-\ell,\b} (f) = S_{n-k}^{-\ell+k,\b+k}  (\partial^k f), \qquad k =1,\ldots, \ell. 
\end{equation}
\end{prop} 

Both $\{\wh J_{n}^{\a,-m}\}$ in Proposition \ref{prop:Qam} and $\{\wt J_{n}^{-m,\a}\}$ in Proposition \ref{prop:Qlb2} are
mutually orthogonal polynomials with respect to the inner product $\la \cdot,\cdot\ra_{\a,-m}$ in \eqref{eq:ipd-a-m}. Since
the dimension of the space of orthogonal polynomials of a fixed degree is 1, we must then have 
$\wh J_{n}^{\a,-m}(t) = c_n \wt J_{n}^{-m,\a}(t)$, where $c_n$ is a multiple constant. By the definition of these two families
of polynomials, this last identity captures the failure of the identity $J_n^{\a,\b}(-t) = (-1)^n J_n^{\b,\a}(t)$ when $\a$ and $\b$
are certain negative integer values. 


\subsection{Approximation by polynomials in the Sobolev space} 

As an application of the Sobolev orthogonality defined in this section, we prove a result on approximation
by polynomials in the Sobolev space. Let $m$ be a positive integer and let $\la \cdot,\cdot\ra_{-m,-m}$ and 
$\wh J_n^{-m,-m}$ be defined as in Proposition \ref{prop:Qlb}. Let $S_n^{-m} f := S_n^{-m,-m} f$ and 
$\wh J_n^{-m} := \wh J_n^{-m,-m}$. Recall that 
$$
   S_n^{-m}f = \sum_{k=0}^n \wh f_k^m \wh J_k^{-m}(t), \qquad \wh f_k^m = \frac{\la f,J_k^{-m}\ra_{-m,-m}}
     {\la J_k^{-m},J_k^{-m}\ra_{-m,-m}}.
$$
In particular, $S_n^0$ is the $n$-th partial sum of the usual Legendre series. 

Let $\|\cdot\|$ denote the norm of $L^2[0,1]$ and denote the error of the best approximation by polynomials
in $L^2[0,1]$ by 
$$
  E_n(f)_2: = \inf_{\mathrm{deg} p \le n} \|f- p\| = \Big(\sum_{k=n+1}^\infty |\wh f_k^0|^2 h_k^0\Big)^{1/2},
$$
where the second equality is the standard Parseval identity and, by \eqref{eq:normJacobi} and \eqref{JacobiP2}, 
$$
\wh h_k^0 := \int_0^1 \left[\wh J_k^0(t)\right]^2 dt = \frac{2}{(2k+1) (k+1)_k^2}  \sim k^{-2 k-1}. 
$$

\begin{thm}
For $m = 1,2,\ldots$, $k =0, 1, \ldots, m$ and $n \ge m$,  
\begin{equation} \label{eq:error1d}
  \|\partial^k f - \partial^k S_n^{-m} f \|_2 \le c\, n^{- m+k} E_{n-m}(\partial^m f)_2,
\end{equation}
where $c$ is a constant depending only on $m$. 
\end{thm}

\begin{proof}
By \eqref{derivativeJ}, it is easy to see that $\partial^m S_n^{-m} f = S_{n-m}^ 0 (\partial^m f)$ implies that
$$
    \wh f_{k+m}^{-m} =  \wh{\partial^m f} {}_k^0, \qquad k =0,1,\ldots. 
$$
Using both equations in \eqref{eq:JacobiR1}, it is easy to deduce that, for $k > m$, 
$$
   J_k^{-m}(t) = J_k^{-m+1}(t) - c_k^m J_{k-2}^{-m+1}, \quad c_k^m = \frac{1}{4 ( 4(k-m)^2-1)}   \sim k^{-2}.
$$ 
Iterating the above identity shows that 
$$
 \wh J_k^{-m}(t) = \sum_{0\le j \le m} c_{k,j}^m \wh J_{k-2j}^0(t) \quad \hbox{with}\quad c_{k,j}^m  \sim k^{-2j},
$$
where $\wh J_n^0(t) = J_n^{0,0} (2t-1)$. Consequently, we see that 
\begin{align*}
 f(t) - S_n^{-m} f(t) & = \sum_{k=n+1}^\infty \wh f_k^{-m} \wh J_k^{-m}(t) =
          \sum_{k=n+1}^\infty \wh{\partial^m f}{}_{k-m}^0 \sum_{0\le j \le m} c_{k,j}^m \wh J_{k-2j}^0(t) \\
          & = \sum_{0 \le j \le m} \sum_{k=n+1-2j}^\infty \wh{\partial^m f}{}_{k-m+2j}^0 c_{k+2j,j}^m \wh J_k^0(t).
\end{align*}
Exchanging the order of the summation and using the orthogonality of $\wh J_k^0$, we obtain
\begin{align*}
 \|f-S_n^{-m} f \|_2^2 \le & \sum_{k = n+1-2m}^{n} \Big | \sum_{n+1-k \le 2j \le m} |\wh{\partial^m f}{}_{k-m+2j}^0| c_{k+2j,j}^m
   \Big |^2 \wh h_k^0 \\
   & + \sum_{k = n+1}^\infty \Big | \sum_{0 \le j \le m} |\wh{\partial^m f}{}_{k-m+2j}^0| c_{k+2j,j}^m
   \Big |^2 \wh h_k^0.
\end{align*}
It is easy to see that $(c_{k+2j,k}^m )^2 \wh h_k^0 / \wh h_{k-m+2j}^0 \sim k^{-2m}$ for $0 \le j \le m/2$, with
a constant of proportion depending only on $m$, so that, by the Cauchy-Schwarz inequality, 
\begin{align*}
 \|f-S_n^{-m} f \|_2^2 \le & \ c \sum_{k = n+1-2m}^\infty k^{-2m} \sum_{0 \le j \le m} \left|\wh{\partial^m f}{}_{k-m+2j}^0\right|^2
      \wh h_{k-m+2j}^0 \\
      \le & \ c\, n^{-2m}  \sum_{k = n+1-m}^\infty  \left |\wh{\partial^m f}{}_k^0 \right|^2 \wh h_k^0 
       = c \, n^{-2m} \|\partial^m f - S_{n-m}^0 (\partial^m f) \|_2^2. 
\end{align*}  
This proves the case $k=0$. The case $k =m$ follows  from \eqref{eq:Qlb}. The
intermediate cases of $1 \le k \le m$ follow from the standard Landau-Kolmogorov inequality. 
\end{proof}

This result, and the proof, provides a precursor for what lies ahead for our main results on the triangle. We did not 
strike for the most general result and chose a proof that is straight forward but more conceptual. 
As far as we know, the estimate \eqref{eq:error1d} is new since similar results in the literature have a modulus of smoothness 
in the right hand side (\cite{K}).

\section{Jacobi polynomials on the triangle and their properties}
\label{sect4}
\setcounter{equation}{0}

For $\a,\b,\g  \in \RR$, let $X= (x,y,1-x-y)$ and define, as in \eqref{eq:weight},
$$
  \varpi_{\a,\b,\g}(x,y) := X^{\a,\b,\g} = x^\a y^\b (1-x-y)^\g, \qquad (x,y) \in \triangle.
$$

\subsection{Jacobi polynomials on the triangle}
If $\a,\b,\g  > -1$, then $\varpi_{\a,\b,\g}$ is integrable on $\triangle$ and we define
an inner product
$$
  \la f, g \ra_{\a,\b,\g} := \int_{\triangle} f(x,y) g(x,y) \varpi_{\a,\b,\g}(x,y) dxdy,
$$
which is the same as \eqref{eq:ipd}. Let $\CV_n(\varpi_{\a,\b,\g})$ be the space of orthogonal polynomials of degree $n$ 
with respect to this inner product. 

Explicit bases of $\CV_n(\varpi_{\a,\b,\g})$ can be given in terms of the Jacobi polynomials $J_n^{\a,\b}$. 
A mutually orthogonal basis of $\CV_n(\varpi_{\a,\b,\g})$ is given by (cf. \cite[Section 2.4]{DX})
\begin{equation}\label{eq:J-basis}
 J_{k,n}^{\a,\b,\g} (x,y) := (x+y)^k J_k^{\a,\b}\left( \frac{y-x}{x+y}\right) J_{n-k}^{2k+\a+\b+1,\g}(1-2x-2y),
    \quad 0 \le k \le n. 
\end{equation}
More precisely, 
\begin{equation*} 
 \la J_{k,n}^{\a,\b,\g},  J_{j,m}^{\a,\b,\g} \ra_{\a,\b,\g} = h_{k,n}^{\a,\b,\g} \delta_{j,k}\delta_{n,m}, 
\end{equation*}
where 
\begin{align}  \label{eq:hkn}
 h_{k,n}^{\a,\b,\g} = \, & \frac{\Gamma(k+\alpha+1)\Gamma(k+\beta+1)\Gamma(k+\a+\b+1)\Gamma(n-k+\g+1)}
   {k!(n-k)!} \\
    & \times  \frac{\Gamma(n+k+\a+\b+2)\Gamma(n+k+\a+\b+\g+2)} {\Gamma(2 k+\a+ \b +2)^2\Gamma(2n+\a+\b+\g+3)^2}
      \notag \\
    &\times  (2k + \a + \b + 1) (2n +\a + \b +\g +2). \notag
\end{align}
For convenience, we also define $J_{k,n}^{\a,\b,\g}(x,y) =0$ for $k < 0$ or $k > n$. 
The triangle $\triangle$ is symmetric under the permutation of $(x,y,1-x-y)$. Furthermore, the inner product 
$\la \cdot,\cdot\ra_{\a,\b,\g}$ is invariant under simultaneous permutation of $(\a,\b,\g)$ and $(x,y,1-x-y)$. 
As a consequence, define 
\begin{align}
 K_{k,n}^{\a,\b,\g} (x,y) 
      := &\ (1-y)^k J_k^{\g,\a}\left( \frac{2x}{1-y} -1 \right)
        J_{n-k}^{2k+\a+\g+1,\b}(2y-1), \quad 0 \le k \le n,\label{eq:K-basis}\\
  L_{k,n}^{\a,\b,\g} (x,y) 
      := & \ (1-x)^k J_k^{\b,\g}\left(1- \frac{2y}{1-x}\right)
        J_{n-k}^{2k+\b+\g+1,\a}(2x-1), \quad 0 \le k \le n,\label{eq:L-basis}
\end{align} 
then both $\{ K_{k,n}^{\a,\b,\g} : 0 \le k \le n \}$ and $\{  L_{k,n}^{\a,\b,\g} : 0 \le k \le n \}$ are mutually orthogonal 
basis of $\CV_n^d(\varpi_{\a,\b,\g})$; moreover, 
$$
   \la K_{k,n}^{\a,\b,\g},  K_{j,m}^{\a,\b,\g} \ra_{\a,\b,\g} = h_{k,n}^{\g,\a,\b} \quad \hbox{and}\quad
   \la   L_{k,n}^{\a,\b,\g},  L_{j,m}^{\a,\b,\g} \ra_{\a,\b,\g} = h_{k,n}^{\b,\g,\a}, \quad 0\le k \le n.
$$
For latter use, we collect the relations between the three bases below:
\begin{align}\label{eq:3bases-transform}
\begin{split}
  J_{k,n}^{\a,\b,\g}(x,y) &\ = K_{k,n}^{\b,\g,\a}(y,1-x-y), \qquad  J_{k,n}^{\a,\b,\g}(x,y)  = L_{k,n}^{\g,\a,\b}(1-x-y,x), \\ 
  K_{k,n}^{\a,\b,\g}(x,y) &\ = L_{k,n}^{\b,\g,\a}(y,1-x-y), \qquad  K_{k,n}^{\a,\b,\g}(x,y) = J_{k,n}^{\g,\a,\b}(1-x-y,x), \\ 
  L_{k,n}^{\a,\b,\g}(x,y) &\ = J_{k,n}^{\b,\g,\a}(y,1-x-y), \qquad   L_{k,n}^{\a,\b,\g}(x,y) = K_{k,n}^{\g,\a,\b}(1-x-y,x). 
\end{split}
\end{align}

With our extension of the Jacobi polynomials, the above definitions hold for all $\a, \b, \g \in \RR$ and so is 
\eqref{eq:3bases-transform}. In the following we shall give the following definition: 

\begin{defn}
For $\a,\b,\g\in \RR$ and $n =0,1,2,\ldots$, let $\CV_n(\varpi_{\a,\b,\g})_J$ be the space spanned by 
$\{J_{k,n}^{\a,\b,\g}: 0 \le k \le n\}$; $\CV_n(\varpi_{\a,\b,\g})_K$ and $\CV_n(\varpi_{\a,\b,\g})_L$ are defined similarly. 
\end{defn}

\begin{prop}
If $\a,\b,\g> -1$ or $-\a, -\b, - \g \notin \NN$,  then
\begin{equation} \label{eq:CV=CV}
\CV_n(\varpi_{\a,\b,\g})_J=\CV_n(\varpi_{\a,\b,\g})_K=\CV_n(\varpi_{\a,\b,\g})_L = :
\CV_n(\varpi_{\a,\b,\g}),
\end{equation}
where $=:$ means we remove the subscript $J, K, L$ when they are equal. 
\end{prop}

\begin{proof}
If $\a,\b,\g > -1$, this follows from the fact that all three families of polynomials are orthogonal with respect to
$\la \cdot,\cdot \ra_{\a,\b,\g}$. In particular, this means that, for example, each $J_{k,n}^{\a,\b,\g}$ can
be expressed as a linear combination of $L_{j,n}^{\a,\b,\g}$, $0\le j \le n$, 
$$
   J_{k,n}^{\a,\b,\g} (x,y)= \sum_{j=0}^n c_{j,k,n}^{\a,\b,\g} L_{j,n}^{\a,\b,\g}(x,y),
$$
where the coefficient $c_{j,k,n}^{\a,\b,\g}$ is given explicitly by a multiple of a  ${}_3F_2$ Racah polynomial
(\cite{D, IX}). If $-\a,-\b,-\g \notin \NN$, then $\iota_0$ indexes for both factors of the Jacobi polynomials in 
$J_{k,n}^{\a,\b,\g}$,  $K_{k,n}^{\a,\b,\g}$ and $L_{k,n}^{\a,\b,\g}$ are zero, and we can extend these 
relations by analytic continuation. 
\end{proof}

The relation \eqref{eq:CV=CV} still holds for some triplets of $\a,\b,\g$ that contain negative integers, but it 
does not hold for all such triplets, as we shall see in later sections. 

From Proposition \ref{prop:J-mJm}, we can deduce the following relations: 

\begin{prop}
Let $\wh \a = \max\{0, -\a\}$ for $\a > -1$ or $\a = -1$. 
\begin{align}
  J_{k,n}^{\a,\b,\g} (x,y) & = c_{k,n}^{\a,\b,\g} X^{\wh \a,\wh \b,\wh \g} 
     J_{k - \wh \a - \wh \b,n -\wh \a -\wh \b - \wh \g}^{\a+2\wh \a,\b+2 \wh \b,\g+2\wh \g} (x,y), \quad  \wh \a + \wh \b \le k \le n- \wh \g,  \label{eq:J+J-}  \\
  K_{k,n}^{\a,\b,\g} (x,y) & = c_{k,n}^{\g,\a,\b} X^{\wh \a,\wh \b,\wh \g} 
      K_{k - \wh \a - \wh \g,n -\wh \a - \wh \b - \wh \g}^{\a+2\wh \a,\b+2 \wh \b,\g+2\wh \g} (x,y), 
       \quad  \wh \a + \wh \g \le k \le n - \wh \b,     \label{eq:K+K-} \\ 
  L_{k,n}^{\a,\b,\g} (x,y) & = c_{k,n}^{\b,\g,\a} X^{\wh \a,\wh \b,\wh \g} 
      L_{k - \wh \b - \wh \g,n -\wh \a- \wh \b -\wh \g}^{\a+2\wh \a,\b+2 \wh \b,\g+2\wh \g} (x,y),
        \quad  \wh \b + \wh \g \le k \le n- \wh\a, \label{eq:L+L-}
\end{align}
where $c_{\a,\b,\g}$ is a constant which can be determined by the relations in Proposition \ref{prop:J-mJm}. 
\end{prop}

\subsection{Recursive relations}

It is known that the Jacobi polynomials on the triangle satisfies two three-term relations \cite[p. 80]{DX}. 
What we will need are different types of recursive relations between Jacobi polynomials of different parameters. Let 
$\tau_n^{\a,\b}$ be defined as in \eqref{eq:tau}. We further define
\begin{equation}\label{eq:sigma}
      \sigma_n^{\a,\b}:= \frac{(n+\a) (n+\a + \b)}{(2n+\a +\b) (2n+\a+\b+1)}. 
\end{equation}

\begin{prop}
For $\a,\b,\g \in \RR$, assume $n > -\g$ and $n > -(\a+\b+1)/2$. Then
\begin{align}
   J_{k,n}^{\a,\b,\g}(x,y) =  J_{k,n}^{\a,\b,\g+1}(x,y) + \tau_{n-k}^{\g,2k+\a+\b+1} J_{k,n-1}^{\a,\b,\g+1}(x,y);\label{eq:recurJc}
\end{align}
moreover, if the conditions of Proposition \ref{prop:tau} holds with $k$ in place of $n$, then
\begin{align}
J_{k,n}^{\a,\b,\g}(x,y) =  &\ J_{k,n}^{\a+1,\b,\g}(x,y) - \tau_{n-k}^{2k+\a+\b+1,\g} J_{k,n-1}^{\a+1,\b,\g}(x,y) \label{eq:recurJa}\\
  + & \ \tau_k^{\a,\b}\left((n-k+1)J_{k-1,n}^{\a+1,\b,\g}(x,y)- \sigma_{n-k}^{2k+\a+\b+1,\g}J_{k-1,n-1}^{\a+1,\b,\g}(x,y)\right); 
     \notag \\
J_{k,n}^{\a,\b,\g}(x,y) =  &\ J_{k,n}^{\a,\b+1,\g}(x,y) - \tau_{n-k}^{2k+\a+\b+1,\g} J_{k,n-1}^{\a,\b+1,\g}(x,y)\label{eq:recurJb}\\
  - & \ \tau_k^{\a,\b}\left((n-k+1)J_{k-1,n}^{\a,\b+1,\g}(x,y)- \sigma_{n-k}^{2k+\a+\b+1,\g}J_{k-1,n-1}^{\a,\b+1,\g}(x,y)\right). 
      \notag
\end{align}
\end{prop} 

\begin{proof}
The first identity follows from applying \eqref{eq:tau} to $J_{n-k}^{2k+\a+\b+1,\g}$. It is easy to check that 
$n > -\g$ and $n > -(\a+\b+1)/2$ guarantee that we can apply \eqref{eq:tau} in this setting. For the second identity, 
we first apply \eqref{eq:tau} on $J_k^{\a,\b}$ in $J_{k,n}^{\a,\b}$ and deduce 
\begin{align*}
 J_{k,n}^{\a,\b,\g}(x,y) =  & J_{k,n}^{\a+1,\b,\g}(x,y) - \tau_{n-k}^{2k+\a+\b+1,\g} J_{k,n-1}^{\a+1,\b,\g}(x,y) \\
     & - \tau_k^{\a,\b}(x+y)^k J_{k-1}^{\a+1,\b}\left(\frac{y-x}{x+y}\right)J_{n-k}^{2k+\a+\b+1,\g}(1-2x-2y),
\end{align*}
where we have also applied \eqref{eq:tau} on $J_{n-k}^{2k+\a+\b+1,\g}$ once. For the last term, we need the 
identity
$$
 \frac{1-s}{2}J_m^{\a,\b}(s) = - (m+1) J_{m+1}^{\a-1,\b}(s) + \s_m^{\a,\b} J_m^{\a-1,\b}(s),
$$
which follows from the second identity in \eqref{eq:JacobiR3b} and the second identity in \eqref{eq:JacobiR1}. 
Applying this identity on $(x+y)J_{n-k}^{2k+\a+\b+1,\g}(1-2x-2y)$ completes the proof of \eqref{eq:recurJa}.
The proof of \eqref{eq:recurJb} follows similarly. 
\end{proof}

Evidently, by \eqref{eq:3bases-transform}, each identity in the proposition also holds if we replace all  
$J_{k,n}^{\a,\b,\g}$ by $K_{k,n}^{\b,\g,\a}$ or by $L_{k,n}^{\g,\a,\b}$, respectively, throughout the identity. 

\begin{prop}\label{prop:Ja+1b+1}
Let $\tau_n^{\a,\b}$ and $\s_n^{\a,\b}$ be defined as in \eqref{eq:tau} and \eqref{eq:sigma}, respectively, and let 
$\mathfrak{a}_k^{\a,\b}$ and $\mathfrak{b}_k^{\a,\b}$ be defined as in \eqref{eq:three-termJ}. Then
\begin{align*}
& J_{k,n}^{\a,\b,\g} (x,y) =   J_{k,n}^{\a+1,\b+1,\g}(x,y) + d_{1,k,n}^{\a,\b} J_{k,n-1}^{\a+1,\b+1,\g}(x,y)+d_{2,k,n}^{\a,\b}
 J_{k,n-2}^{\a+1,\b+1,\g}(x,y) \\
 &\quad  +  c_{0,k,n}^{\a,\b} J_{k-1,n}^{\a+1,\b+1,\g}(x,y) + c_{1,k,n}^{\a,\b} J_{k-1,n-1}^{\a+1,\b+1,\g}(x,y)+c_{2,k,n}^{\a,\b}
   J_{k-1,n-2}^{\a+1,\b+1,\g}(x,y) \\
  &\quad  +  e_{0,k,n}^{\a,\b} J_{k-2,n}^{\a+1,\b+1,\g}(x,y) + e_{1,k,n}^{\a,\b} J_{k-2,n-1}^{\a+1,\b+1,\g}(x,y)+e_{2,k,n}^{\a,\b}
   J_{k-2,n-2}^{\a+1,\b+1,\g}(x,y),
\end{align*}
where 
\begin{align*}
 & d_{1,k,n}^{\a,\b}:=  - \left(\tau_{n-k}^{2k+\a+\b+2,\g}+\tau_{n-k}^{2k+\a+\b+1,\g}\right), \quad 
     d_{2,k,n}^{\a,\b}:=   \tau_{n-k}^{2k+\a+\b+1,\g}\tau_{n-k-1}^{2k+\a+\b+2,\g}, \\  
  &  c_{0,k,n}^{\a,\b}:=  - c_k^{\a,\b}  (n-k+1), \quad 
      c_{1,k,n}^{\a,\b}:=  \frac12 c_k^{\a,\b}  \left(1- \mathfrak{a}_{n-k}^{2k+\a+\b+1,\g} \right),    \\
  &  c_{2,k,n}^{\a,\b}:=  - \frac12 c_k^{\a,\b}  \mathfrak{b}_{n-k}^{2k+\a+\b+1,\g}, \quad
      e_{0,k,n}^{\a,\b}:=   d_k^{\a,\b} (n-k+1) (n-k+2), \\
  &  e_{1,k,n}^{\a,\b}:=   -d_k^{\a,\b}  (n-k+1)  \left(\s_{n-k+1}^{2k+\a+\b,\g}+\s_{n-k}^{2k+\a+\b+1,\g}\right), \\ 
  &  e_{2,k,n}^{\a,\b}:=     d_k^{\a,\b} \s_{n-k}^{2k+\a+\b+1,\g} \s_{n-k}^{2k+\a+\b,\g}. 
\end{align*}
in which 
$$
c_k^{\a,\b} := \tau_{k}^{\b,\a+1} -\tau_{k}^{\a,\b} \quad \hbox{and}\quad
 d_k^{\a,\b}: = \tau_{k}^{\a,\b} \tau_{k-1}^{\b,\a+1}.
$$
\end{prop}

\begin{proof}
By combining the two identities in \eqref{eq:JacobiR1}, it is easy to see that 
$$
  J_n^{\a,\b}(t) = J_n^{\a+1,\b+1}(t)+ (\tau_n^{\b,\a+1} -\tau_n^{\a,\b})J_{n-1}^{\a+1,\b+1}(t) -
   \tau_n^{\a,\b}\tau_{n-1}^{\b,\a+1} J_{n-2}^{\a+1,\b+1}(t). 
$$ 
Applying this identity to the second Jacobi polynomial in $J_{k,n}^{\a,\b,\g}$, we obtain
\begin{align} \label{eq:prop4.5a}
J_{k,n}^{\a,\b,\g}(x,y) =  & \ (x+y)^k J_k^{\a+1,\b+1}(s) J_{n-k}^{2k+\a+\b+1,\g}(t) \\
  & + (\tau_{k}^{\b,\a+1} -\tau_{k}^{\a,\b}) (x+y) J_{k-1,n-1}^{\a+1,\b+1,\g}(x,y) \notag \\
  & +  \tau_{k}^{\a,\b} \tau_{k-1}^{\b,\a+1}(x+y)^k J_{k-2}^{\a+1,\b+1}(s) J_{n-k}^{2k+\a+\b+1,\g}(t), \notag
\end{align}
where $s = (y-x)/(x+y)$ and $t = 1-2x-2y$. For the first term in the right hand side, we use \eqref{eq:JacobiR1} twice to
obtain 
\begin{align*}
  (x+y)^k & J_k^{\a+1,\b+1}(s) J_{n-k}^{2k+\a+\b+1,\g}(t) 
    =  \tau_{n-k}^{2k+\a+\b+1,\g}\tau_{n-k-1}^{2k+\a+\b+2,\g}  J_{k,n-2}^{\a+1,\b+1,\g}(x,y) \\
     &- (\tau_{n-k}^{2k+\a+\b+2,\g}+\tau_{n-k}^{2k+\a+\b+1,\g}) J_{k,n-1}^{\a+1,\b+1,\g}(x,y) + 
         J_{k,n}^{\a+1,\b+1,\g}(x,y).
\end{align*}
For the second term in the right hand side of \eqref{eq:prop4.5a}, we use the three-term relation \eqref{eq:three-termJ}
on $(x+y) J_{n-k}^{2k+\a+\b+1,\g}(1-2x-2y)$ to obtain that 
\begin{align*}
    (x+y) & J_{k-1,n-1}^{\a+1,\b+1,\g}(x,y) = - (n-k+1)J_{k-1,n}^{\a+1,\b+1,\g}(x,y) \\
     &  +\frac12 \left(1- \mathfrak{a}_{n-k}^{2k+\a+\b+1,\g} \right)J_{k-1,n-1}^{\a+1,\b+1,\g}(x,y)  
        - \frac12 \mathfrak{b}_{n-k}^{2k+\a+\b+1,\g} J_{k-1,n-2}^{\a+1,\b+1,\g}(x,y).
\end{align*}
For the third term in the right hand side of \eqref{eq:prop4.5a}, we use the identity
$$
  \frac{1-t}{2} J_n^{\a,\b}(t) = -(n+1)J_{n+1}^{\a-1,\b}(t) + \s_n^{\a,\b} J_n^{\a-1,\b}(t),
$$
which is \cite[(22.7.15)]{AS} written for $J_{n}^{\a,\b}$, twice to write
\begin{align*}
     (x+y)^k J_{k-2}^{\a+1,\b+1}(s) & J_{n-k}^{2k+\a+\b+1,\g}(t) = (n-k+1) (n-k+2)J_{k-2,n}^{\a+1,\b+1,\g}(x,y) \\
     &  
      - (n-k+1) (\s_{n-k+1}^{2k+\a+\b,\g}+\s_{n-k}^{2k+\a+\b+1,\g}) J_{k-2,n-1}^{\a+1,\b+1,\g}(x,y)  \\
     &   +\s_{n-k}^{2k+\a+\b+1,\g} \s_{n-k}^{2k+\a+\b,\g} J_{k-2,n-2}^{\a+1,\b+1,\g}(x,y).
\end{align*}
Putting these together and computing the constants, we complete the proof. 
\end{proof}

\subsection{Derivatives of the Jacobi polynomials on the triangle}
We will need a number of properties on the derivatives of these polynomials. To emphasis the symmetry of 
the triangle, we make the following definition: 
$$
   \partial_z : = \partial_y - \partial_x, \qquad z = 1-x-y. 
$$ 
Furthermore, when variables are not specified, we shall denote partial derivatives with respect to the first and the 
second variable by $\partial_1$ and $\partial_2$, respectively, and write $\partial_3 = \partial_2-\partial_1$ 
accordingly.  

For $\a,\b \in \RR$ and $0 \le k \le n$, define
\begin{align} \label{eq:akn-bkn}
  a_{k,n}^{\a,\b} :=  \frac{(k+\b)(n+k+\a+\b+1)}{(2k+\a+\b)(2k+\a+\b+1)}
\end{align}
for $2k+\a+\b \ne 0$ and $2k+\a+\b+1 \ne 0$, and define
\begin{align} \label{eq:akn-bkn2}
a_{k,n}^{-k,-k} :=  n-k+1 \quad \hbox{and} \quad a_{k,n}^{-k-1,-k} := - (n - k).   
\end{align}
  
\begin{prop} \label{prop:partialJ}
Let $\a,\b,\g \in \RR$ and $k,n \in \NN_0$ such that $0 \le k \le n$. Then,
\begin{align}\label{eq:diff1J}
\begin{split}
  \partial_x  J_{k,n}^{\a,\b,\g}(x,y)  =\, & -  a_{k,n}^{\a,\b}  J_{k-1,n-1}^{\a+1,\b,\g+1}(x,y) 
    -   J_{k,n-1}^{\a+1,\b,\g+1}(x,y),  \\
  \partial_y  J_{k,n}^{\a,\b,\g}(x,y)  =\, & a_{k,n}^{\b,\a}  J_{k-1,n-1}^{\a,\b+1,\g+1}(x,y)  - 
          J_{k,n-1}^{\a,\b+1,\g+1}(x,y), \\
     \partial_z J_{k,n}^{\a,\b,\g}(x,y)  =\, & J_{k-1,n-1}^{\a+1,\b+1,\g}(x,y), \quad 0 \le k\le n,
\end{split}
\end{align}
where the third equation holds for all $n \in \NN_0$, the first identity holds 
if either one of the following conditions satisfies: 
\begin{enumerate}[\quad (a)]
\item $2k+\a+\b\ne 0$, $2k+\a+\b+1\ne 0$, and either $k+\a+\b+1 \ne 0$ or $k+\a+\b+1=0$ but $\a +1 \in \{-1,-2,\ldots, -k\}$. 
\item $\a = -k$, $\b =-k$ or $\a = -k-1$, $\b=-k$,
\end{enumerate}
whereas the second identity holds if $\a$ and $\b$ are exchanged in (a) and (b).
\end{prop}

\begin{proof}
Since $\partial_z f(x+y) =0$, the third identity follows directly from \eqref{derivativeJ}. Since 
$J_{n,n}(x,y) = (x+y)^n J_n^{\a,\b}((y-x)/(x+y))$, taking partial derivatives with respect to $x$ by \eqref{derivativeJ},
we see that 
$$
  \partial_x J_{n,n}^{\a,\b,\g}(x,y) = (x+y)^{n-1} \left[n J_n^{\a,\b}(t) - \frac{1+t}{2} J_{n-1}^{\a+1,\b+1}(t) \right]
    = -a_n^{\a,\b} J_{n-1,n-1}^{\a+1,\b,\g+1}(x,y),
$$
where $t = (y-x)/(x+y)$ and the second identity follows from \eqref{eq:JacobiR3} or \eqref{eq:JacobiR3b}. 
If $0 \le k \le n-1$, taking derivatives with respect to $x$ and using the same identity, with $n$ replaced by $k$, 
we see that 
\begin{align*}
   \partial_x J_{k,n}^{\a,\b,\g}(x,y) = & -\frac{k+\b}{2k+\a+\b} (x+y)^{k-1} J_{k-1}^{\a+1,\b}\left(\frac{y-x}{x+y}\right)
         J_{n-k}^{2k+\a+\b+1,\g}(1-2x-2y) \\
         & - (x+y)^{k} J_{k}^{\a,\b}\left(\frac{y-x}{x+y}\right)
         J_{n-k-1}^{2k+\a+\b+2,\g+1}(1-2x-2y)
\end{align*}
if $2k+\a+\b \ne 0$. Now we apply \eqref{eq:JacobiR2} to obtain 
\begin{align*}
 J_{n-k}^{2k+\a+\b+1,\g}(1-2x-2y) = & \frac{n+k+\a+\b+1}{2k+\a+\b+1} J_{n-k}^{2k+\a+\b,\g+1}(1-2x-2y)\\
  &+ \frac{x+y}{2k+\a+\b+1} J_{n-k-1}^{2k+\a+\b+2,\g+1}(1-2x-2y)
\end{align*}
and apply the first identity in \eqref{eq:JacobiR1} on $J_{k}^{\a,\b}\left(\frac{y-x}{x+y}\right)$, and then 
substitute the two identities into the right hand side of $\partial_z J_{k,n}^{\a,\b,\g}$ in the above. After
cancellation of two terms, we have proved the first identity of \eqref{eq:diff1J} under assumption (a). The
other cases can be dealt with similarly. 
\end{proof}

As an example, we consider the case $J_{1,n}^{-2,-1,-1}$, for which (b) is satisfied with $k=1$ for $\partial_x$. 
However, when $\a$ and $\b$ exchange positions, the condition is not satisfied for $\partial_y$. In fact, it is not
difficult to check that 
$$
   \partial_y J_{1,n}^{-2,-1,-1}(x,y) = n J_{0,n-1}^{-2, 0,0}(x,y) - J_{1,n-1}^{-2,0,0}(x,y)
       + \frac{1}{2n-3} J_{0,n-2}^{-1,0,0}(x,y),
$$ 
which shows that the identity \eqref{eq:diff1J} for $\partial_y$ does not hold in this case. 

\begin{prop}\label{prop:partialK}
Let $\a,\b,\g \in \RR$ and $k,n \in \NN_0$ such that $0 \le k \le n$. Then,
\begin{align} \label{eq:diffK}
\begin{split}
  \partial_x K_{k,n}^{\a,\b,\g} (x,y)  = \,&   K_{k-1,n-1}^{\a+1,\b,\g+1}(x,y), \quad 1 \le k\le n,\\
  \partial_y  K_{k,n}^{\a,\b,\g}(x,y)  =\, &  a_{k,n}^{\g,\a}  K_{k-1,n-1}^{\a,\b+1,\g+1}(x,y) 
    +    K_{k,n-1}^{\a,\b+1,\g+1}(x,y),  \\
  \partial_z  K_{k,n}^{\a,\b,\g}(x,y)  =\, & - a_{k,n}^{\a,\g}  K_{k-1,n-1}^{\a+1,\b+1,\g}(x,y)  +
          K_{k,n-1}^{\a+1,\b+1,\g}(x,y), 
\end{split}
\end{align}
where the first equation holds for all $n \in \NN_0$, the second identity holds 
if either one of the following conditions satisfies: 
\begin{enumerate}[\quad (a)]
\item $k+\g+\a\ne 0$, $2k+\g+\a+1\ne 0$, and either $k+\g+\a+1 \ne 0$ or $k+\g+\a+1=0$ but $\g +1 \in \{-1,-2,\ldots, -k\}$. 
\item $\g = -k$, $\a =-k$ or $\g = -k-1$, $\a=-k$,
\end{enumerate}
whereas the third identity holds if $\g$ and $\a$ are exchanged in (a) and (b).
\end{prop}

\begin{proof}
This follows from \eqref{eq:3bases-transform} and Proposition \ref{prop:partialJ}. Indeed, if $g(x,y) = f(z,x)$ with
$z = 1-x-y$, then 
\begin{equation} \label{eq:d-transform1}
 \partial_x g(x,y) = \partial_3 f(z,x), \quad  \partial_y g(x,y) =- \partial_1 f(z,x), \quad \partial_z g(x,y) = -\partial_2 f(z,x).
\end{equation}
Applying this to $K_{k,n}^{\a,\b,\g}(x,y) =  J_{k,n}^{\g,\a,\b}(1-x-y,x)$, we deduce the stated results from
Proposition \ref{prop:partialJ}.
\end{proof}

\begin{prop}\label{prop:partialL}
Let $\a,\b,\g \in \RR$ and $k,n \in \NN_0$ such that $0 \le k \le n$. Then,
\begin{align} \label{eq:diffL}
\begin{split}
  \partial_x  L_{k,n}^{\a,\b,\g}(x,y)  =\, & - a_{k,n}^{\g,\b}  L_{k-1,n-1}^{\a+1,\b,\g+1}(x,y) 
    +   L_{k,n-1}^{\a+1,\b,\g+1}(x,y),  \\
 \partial_y L_{k,n}^{\a,\b,\g} (x,y)  = \,& -  L_{k-1,n-1}^{\a,\b+1,\g+1}(x,y), \quad  1 \le k\le n,\\
  \partial_z  L_{k,n}^{\a,\b,\g}(x,y)  =\, & - a_{k,n}^{\b,\g}  L_{k-1,n-1}^{\a+1,\b+1,\g}(x,y)  - 
         L_{k,n-1}^{\a+1,\b+1,\g}(x,y), 
\end{split}
\end{align}
where the second equation holds for all $n \in \NN_0$, the first identity holds if either one of the following conditions satisfies: 
\begin{enumerate}[\quad (a)]
\item $k+\g+\b \ne 0$, $2k+\g+\b+1\ne 0$, and either $k+\g+\b+1 \ne 0$ or $k+\g+\b+1=0$ but $\g +1 \in \{-1,-2,\ldots, -k\}$. 
\item $\g = -k$, $\b =-k$ or $\g = -k-1$, $\b=-k$,
\end{enumerate}
whereas the third identity holds if $\g$ and $\a$ are exchanged in (a) and (b).
\end{prop}

\begin{proof}
This again follows from \eqref{eq:3bases-transform} and Proposition \ref{prop:partialJ}. Indeed, if $h(x,y) = f(y,z)$ with
$z = 1-x-y$, then 
\begin{equation} \label{eq:d-transform2}
 \partial_x h(x,y) = - \partial_2 f(y,z), \quad  \partial_y h(x,y) = - \partial_3 f(y,z), \quad \partial_z h(x,y) = \partial_1 f(y,z).
\end{equation}
Applying this to $L_{k,n}^{\a,\b,\g}(x,y) =  J_{k,n}^{\b,\g,\a}(y,1-x-y)$, we deduce the stated results from
Proposition \ref{prop:partialJ}.
\end{proof}
 
\subsection{Differential operators}
Also useful are differential operators that have the Jacobi polynomials on the triangle as eigenfunctions. For $\a,\b,\g > -1$, 
there is a second order partial differential operator that has all elements of $\CV_n(\varpi^{\a,\b,\g})$ as eigenfunctions with an
eigenvalue that depends only on $n$ (see, for example, \cite[(8.2.3)]{DX}). Our first proposition below contains a 
generalization of this operator from the second order to $2r$ order for $r =1,2,\ldots$, which first appeared in \cite{LS}. 

\begin{prop}
Let $\a,\b,\g \in \RR$. For $r =1,2,\ldots$ and $\mathbf{n} = (n_1,n_2,n_3) \in \NN_0^3$, define 
$$
   \CL_r^{\a,\b,\g} : = X^{-\a,-\b,-\g} \sum_{|\mathbf{n}| = r} \binom{r}{\mathbf{n}} \partial^{\mathbf{n}} 
   X^{\a +n_1+n_3, \b+n_2+n_3, \g+n_1+n_2}  \partial^{\mathbf{n}}.  
$$
Then, for  $\a,\b,\g > -1$,  
\begin{equation}\label{eq:eigen1}
   \CL_r^{\a,\b,\g} P  =  \l_{r,n}^{\a,\b,\g} P, \qquad  \forall P \in \CV_n(\varpi_{\a,\b,\g}),
\end{equation}
where 
$$
       \l_{r,n}^{\a,\b,\g} := (-n)_r (n+\a+\b+\g+2)_r.  
$$
Moreover, if $n > - \a-\b-\g-2$, then \eqref{eq:eigen1} holds for $\{J_{k,n}^{\a,\b,\g}: 0\le k \le n\}$
if $\a,\b > -1$ or $\a+\b \notin \ZZ$, for $\{K_{k,n}^{\a,\b,\g}: 0\le k \le n\}$ if $\g,\a > -1$ or $\g+\a \notin \ZZ$, 
for $\{L_{k,n}^{\a,\b,\g}: 0\le k \le n\}$ if $\b,\g > -1$ or $\b+\g \notin \ZZ$.
\end{prop}

\begin{proof}
For $\a,\b,\g > -1$, this is proved in \cite{LS} by induction on $r$. If $n > - \a-\b-\g-2 >0$, then
$\iota_0^{2k+\a+\b+1,\g}(n-k) =0$ for all $0 \le k \le n$, so that the second Jacobi polynomial 
$J_{n-k}^{2k+\a+\b+1,\g}(1-2x-2y)$ in $J_{k,n}^{\a,\b,\g}$ can be extended analytically to $\g \in \RR$. 
Similarly, for $\a+\b \notin \ZZ$, $\iota_0^{\a,\b}(k) =0$, so that the first Jacobi polynomial in $J_{k,n}^{\a,\b,\g}$ 
is analytically extended to such $\a,\b$. Consequently, the identity \eqref{eq:eigen1} holds for 
$J_{k,n}^{\a,\b,\g}$, $0\le k \le n$, by analytic continuation. The same proof works for $K_{k,n}^{\a,\b,\g}$ 
and $L_{k,n}^{\a,\b,\g}$. 
\end{proof}

The identity \eqref{eq:eigen1} could also work for some other parameters. For example, it works for
$\a = \b = \g = -1$ by using, for example, Proposition \ref{prop:partialJ}. However, it does not hold
for all triplets of $\a,\b,\g$, since the relation \eqref{eq:CV=CV} does not hold for all triplets.  

Our next family of differential operators are pertinent to the three families of polynomials and have eigenvalues 
depending on both $k$ and $n$. 

\begin{prop}
For $\a, \b,\g \in \RR$ and $0 \le k \le n$, 
\begin{align}  \label{eq:eigen2}
\begin{split}
   \partial_z \left (X^{\a+1,\b+1,0} \partial_z J_{k,n}^{\a,\b,\g}(x,y) \right) & = 
              \mu_k^{\a,\b} J_{k,n}^{\a,\b,\g}(x,y) X^{\a,\b,0}, \\
  \partial_x \left(X^{\a+1,0,\g+1} \partial_x K_{k,n}^{\a,\b,\g}(x,y)\right) & = 
           \mu_k^{\g,\a} K_{k,n}^{\a,\b,\g}(x,y) X^{\a, 0, \g}, \\
      \partial_y \left(X^{0,\b+1,\g+1} \partial_y L_{k,n}^{\a,\b,\g}(x,y)\right) &= 
         \mu_k^{\b,\g} L_{k,n}^{\a,\b,\g}(x,y) X^{0,\b, \g},   
\end{split}           
\end{align}
where 
$$
   \mu_k^{\a,\b}: = -k(k+\a+\b+1), 
$$
and the first identity holds if $-\a, -\b \notin \ZZ$ or $\a+\b> -1$, the second identity
holds if $-\g, -\a \notin \ZZ$ or $\g+\a > -1$, and the third identity holds if $-\b, -\g \notin \ZZ$ or $\b+\g > -1$.
\end{prop}
 
\begin{proof}
For the first identity, using \eqref{eq:diff1J} twice, it is easy to see that the identity is equivalent to, after dropping the common term and setting $(y-x)/(x+y) =t$, 
\begin{align} \label{eq:JP4}
    \left( - (\a+1)+(\a+\b+1) \frac{1-t}{2}\right) & J_{k-1}^{\a+1,\b+1}(t) +  \frac{(1-t^2)}{4}J_{k-2}^{\a+2,\b+2}(t) \\
     & = -k(k+\a+\b+1) J_k^{\a,\b}(t). \notag
\end{align}
Either the condition $-\a, -\b \notin \ZZ$ or the condition $\a+\b> -1$ warrants that all  three $\iota_0$ indexes for the three
Jacobi polynomials are zero, for which the above identity can be verified directly by working with hypergeometric 
expansions. 
\end{proof}

The first  case for which the identity \eqref{eq:JP4}, hence the first identity in \eqref{eq:eigen2}, fails is $k=1$, 
$\a=-2$ and $\b = 0$, for which $\a+\b =-2 <-1$.  For a fixed $k$, it is possible to relax the restriction on the 
parameters. For example, a quick check shows that \eqref{eq:JP4}, hence the first identity of \eqref{eq:eigen2},
holds for all $k$ if $\a =\b=-1$. 

Combining the two types of differential operators in \eqref{eq:eigen1} and \eqref{eq:eigen2}, we could identify other
types of differential operators that have the Jacobi polynomials on the triangle as eigenfunctions. 

\begin{prop}
For $\a, \b,\g \in \RR$ and $0 \le k \le n$, 
\begin{align}  \label{eq:eigen3}
\begin{split}
     & \left( \partial_x^2 X^{\a+2,\b, \g+2} \partial_x^2  + 2  \partial_x \partial_y X^{\a+1,\b+1, \g+2} \partial_x \partial_y 
      +  \partial_y^2 X^{\a,\b+2,\g+2} \partial_y^2 \right)J_{k,n}^{\a,\b,\g}(x,y), \\
       & \qquad\qquad\qquad\quad  =  \s_{k,n}^{\a,\b,\g} X^{\a,\b,\g} J_{k,n}^{\a,\b,\g}(x,y)\\ 
     & \left( \partial_y^2 X^{\a+2,\b+2, \g} \partial_y^2  + 2  \partial_y \partial_z X^{\a+2,\b+1, \g+1} \partial_y \partial_z 
      +  \partial_z^2 X^{\a+2,\b,\g+2} \partial_z^2 \right) K_{k,n}^{\a,\b,\g} (x,y),\\
       & \qquad\qquad\qquad\qquad  =  \s_{k,n}^{\g,\a,\b} X^{\a,\b,\g} K_{k,n}^{\a,\b,\g}(x,y)\\ 
      & \left( \partial_z^2 X^{\a,\b+2, \g+2} \partial_z^2  + 2  \partial_z \partial_x X^{\a+1,\b+2, \g+1} \partial_z \partial_x 
      +  \partial_x^2 X^{\a+2,\b+2,\g} \partial_x^2 \right)J_{k,n}^{\a,\b,\g}(x,y) \\
       & \qquad\qquad\qquad\qquad  =  \s_{k,n}^{\b,\g,\a} X^{\a,\b,\g} J_{k,n}^{\a,\b,\g}(x,y),          
\end{split}           
\end{align}
where 
$$
     \s_{k,n}^{\a,\b,\g}: = \l_{2,n}^{\a,\b,\g} - 2 \l_{1,n}^{\a+1,\b+1,\g}\mu_{k}^{\a,\b} + \mu_{k-1}^{\a+1,\b+1}\mu_{k}^{\a,\b},
$$
and the first identity holds if $-\a, -\b \notin \ZZ$ or $\a+\b> -1$, the second identity holds if $-\g, -\a \notin \ZZ$ or $\g+\a > -1$,
and the third identity holds if $-\b, -\g \notin \ZZ$ or $\b+\g > -1$.
\end{prop}

\begin{proof}
By the definition of $\CL_r^{\a,\b,\g}$,   
\begin{align*}
& X^{\a,\b,\g} \CL_1^{\a,\b,\g}  
   = \partial_x X^{\a+1,\b,\g+1}\partial_x + \partial_y X^{\a,\b+1,\g+1}\partial_y + \partial_z X^{\a+1,\b+1,\g}\partial_z \\
 & X^{\a,\b,\g} \CL_2^{\a,\b,\g} 
    = \partial_x^2 X^{\a+2,\b,\g+2}\partial_x^2 + \partial_y^2 X^{\a,\b+2,\g+2}\partial_y^2 + \partial_z^2 X^{\a+2,\b+2,\g}\partial_z^2 \\
 & \quad + 2 \partial_x \partial_y X^{\a+1,\b+1,\g+2}\partial_x \partial_y+ 2\partial_y  \partial_z X^{\a+1,\b+2,\g+1}\partial_y \partial_z 
   +2 \partial_z  \partial_x X^{\a+2,\b+1,\g+1}\partial_z  \partial_x.
\end{align*}
Consequently, it follows that 
\begin{align*}
X^{\a,\b,\g} & \left( \partial_x^2 X^{\a+2,\b, \g+2} \partial_x^2  + 2  \partial_x \partial_y X^{\a+1,\b+1,\g+2} \partial_x \partial_y 
      +  \partial_y^2 X^{\a,\b+2,\g+2} \partial_y^2 \right) \\
   &=   X^{\a,\b,\g} \CL_2^{\a,\b,\g}  -  2 \partial_z X^{\a+1,\b+1,\g} \CL_1^{\a+1,\b+1,\g} \partial_z + \partial_z X^{\a+2,\b+2,\g}\partial_z.
\end{align*}
Since $\partial_z J_{k,n}^{\a,\b,\g} = J_{k-1,n-1}^{\a+1,\b+1,\g}$, it follows from \eqref{eq:eigen1} and \eqref{eq:eigen2} that 
\begin{align*}
  \partial_z X^{\a+1,\b+1,\g} \CL_1^{\a+1,\b+1,\g} \partial_z J_{k,n}^{\a,\b,\g} & = \l_{1,n-1}^{\a+1,\b+1,\g} \partial_z  X^{\a+1,\b+1,\g} \partial_z J_{k,n}^{\a,\b,\g}\\
&  = \l_{1,n-1}^{\a+1,\b+1,\g} \mu_{k,n}^{\a,\b} X^{\a,\b,\g} J_{k,n}^{\a,\b,\g},
\end{align*}
where we have used the fact that $\partial_z g(x+y)=0$, so that $\partial_z z^\g F = z^\g \partial F$, which
allows us to use \eqref{eq:eigen2}. Furthermore, applying \eqref{eq:eigen2} twice, we obtain 
\begin{align*}
\partial_z X^{\a+2,\b+2,\g}\partial_z J_{k,n}^{\a,\b,\g}& = \mu_{k-1}^{\a+1,\b+1} \partial_z X^{\a+1,\b+1,\g}\partial_z J_{k,n}^{\a,\b,\g} \\
  & =\mu_{k-1}^{\a+1,\b+1} \mu_{k}^{\a,\b}X^{\a,\b,\g} J_{k,n}^{\a,\b,\g}. 
\end{align*}
Putting these identities together and using \eqref{eq:eigen1} again proves the first identity. The conditions under which 
these identitoes hold follow from those in \eqref{eq:eigen2}.
\end{proof}

\subsection{Orthogonal expansions on the triangle}
If $\a,\b,\g > -1$ and $f \in L^2(\varpi_{\a,\b,\g})$, then $f$ has the orthogonal expansion 
$$
  f  = \sum_{n=0}^\infty \proj_n^{\a,\b,\g} f, \qquad\quad \proj_n^{\a,\b,\g}: L^2(\varpi_{\a,\b,\g}) \mapsto \CV_n^d(\varpi_{\a,\b,\g}).
$$
The projection operator $\proj_n^{\a,\b,\g}$ is independent of the particular choice of orthogonal basis of 
$\CV_n(\varpi_{\a,\b,\g})$. For a given basis, see $J_{k,n}^{\a,\b,\g}$, it can be written explicitly as 
\begin{equation} \label{eq:proj-abc}
  \proj_n^{\a,\b,\g} f(x) = \sum_{k=0}^n \wh f_{k,n}^{\a,\b,\g} J_{k,n}^{\a,\b,\g}, \qquad
      \wh f_{k,n}^{\a,\b,\g}: = \frac{\la f,  J_{k,n}^{\a,\b,\g}\ra_{\a,\b,\g}} {h_{k,n}^{\a,\b,\g}},   
\end{equation}
where $h_{k,n}^{\a,\b,\g} = \la J_{k,n}^{\a,\b,\g}, J_{k,n}^{\a,\b,\g}\ra_{\a,\b,\g}$  is given in \eqref{eq:hkn}. 
We can also replace $J_{k,n}^{\a,\b,\g}$ by $K_{k,n}^{\a,\b,\g}$ or $L_{k,n}^{\a,\b,\g}$ in this expression of 
the projection operator. The $n$-th partial sum of the orthogonal expansion is defined by 
$$
 S_n^{\a,\b,\g} f(x,y): = \sum_{k=0}^n \proj_k^{\a,\b,\g} f(x,y),
$$ 
which agrees with \eqref{eq:Fourier}; it is the $n$-th polynomial of best approximation to $f$ from $\Pi_n^2$ in  
the norm of $L^2(\varpi_{\a,\b,\g})$. 


As an application of our three bases, we prove that the partial derivatives commute with the projection operators 
defined in \eqref{eq:proj-abc}. 

\begin{thm} 
Let $\a,\b,\g > -1$. Then
\begin{align}\label{eq:dproj}
\begin{split}
   \partial_1 \proj_n^{\a,\b,\g} f = \proj_{n-1}^{\a+1,\b,\g+1} \partial_1 f,   \\
   \partial_2 \proj_n^{\a,\b,\g} f = \proj_{n-1}^{\a,\b+1,\g+1} \partial_2 f,   \\
   \partial_3 \proj_n^{\a,\b,\g} f = \proj_{n-1}^{\a+1,\b+1,\g} \partial_3 f.
\end{split}
\end{align}
\end{thm}

\begin{proof}
We consider $\partial_3$ first. By \eqref{eq:diff1J} and \eqref{eq:eigen2}, an integration by parts shows
that 
$$
  \la \partial_3 f, J_{k-1,n-1}^{\a+1,\b+1,\g} \ra_{\a+1,\b+1,\g} 
  =  \la \partial_3 f, \partial_3 J_{k,n}^{\a,\b,\g} \ra_{\a+1,\b+1,\g}    
  =  k ( k+\a+\b+1)\la f,  J_{k,n}^{\a,\b,\g} \ra_{\a,\b,\g}, 
$$
where we have used the fact that $\partial_z h(x+y) =0$. Moreover, it follows from \eqref{eq:hkn} that 
$k ( k+\a+\b+1) h_{k,n}^{\a,\b,\g} = h_{k-1,n-1}^{\a+1,\b+1,\g}$, so that 
$\wh f_{k,n}^{\a,\b,\g} =   \wh {\partial_3 f}_{k-1,n-1}^{\a+1,\b+1,\g}$ for $1 \le k \le n$. Consequently, 
by \eqref{eq:diff1J}, 
\begin{align*}
 \partial_3 \proj_n^{\a,\b,\g} f = & \sum_{k=1}^n  \wh f_{k,n}^{\a,\b,\g} J_{k-1,n-1}^{\a+1,\b+1,\g} \\ 
 = & \sum_{k=1}^n  \wh {\partial_3 f}_{k-1,n-1}^{\a+1,\b+1,\g}J_{k-1,n-1}^{\a+1,\b+1,\g}
= \proj_{n-1}^{\a+1,\b+1,\g} \partial_3 f. 
\end{align*}
Since the projection operator is independent of a particular choice of orthogonal bases, we can follow 
the same proof for $\partial_1$ using $K_{k,n}^{\a,\b,\g}$ basis and for $\partial_2$ using $L_{k,n}^{\a,\b,\g}$ basis.
\end{proof}

We shall refer the relations in \eqref{eq:dproj} as commuting relations of partial derivatives and projection,
or partial sum, operators. They play an essential role in establishing the approximation behavior of $S_n^{\a,\b,\g}$ 
in Section \ref{sect11}. For approximation in the Sobolev space, we look for inner products for which such commuting
relations hold for its respective orthogonal polynomials

\section{Sobolev orthogonality with one negative parameter} 
\label{sect5}
\setcounter{equation}{0}

In this section, we consider the Sobolev orthogonality when only one parameter is $-1$. Since many elements of the 
further development will appear in this first case, we shall include full details and discussions.

\subsection{The space $\CV_n(\varpi_{\a,\b,-1})$}
For $\a, \b > -1$, let 
$$
  W_2^1(\varpi_{\a,\b,0}): = \left \{f \in L^2(\varpi_{\a,\b,0}): x \partial_x f \in L^2(\varpi_{\a,\b,0}), 
   y \partial_y f \in L^2(\varpi_{\a,\b,0})\right \}.
$$
We start with the following bilinear form on $W_2^1(\varpi_{\a,\b,0})$ defined by  
\begin{align*} 
\la f, g \ra_{\a,\b,-1}^J: =& \int_{\triangle} \left [ x \partial_x f(x,y) \partial_x g(x,y) + y \partial_y f(x,y) \partial_y g(x,y) \right]
    \varpi_{\a,\b,0}(x,y) dxdy \\
        & + \l \int_0^1 f(x,1-x) g(x,1-x)\varpi_{\a,\b,0}(x,1-x) dx, 
\end{align*}
where $\l > 0$. In terms of classical inner product for the ordinary Jacobi polynomials, we can also write
\begin{equation} \label{eq:ipd-ab1J}
 \la f,g\ra_{\a,\b,-1}^J = \la \partial_1 f, \partial_1 g\ra_{\a+1,\b,0}+\la \partial_2 f, \partial_2 g\ra_{\a,\b+1,0}
    + \lambda \la f \vert_{z=0}, g\vert_{z=0} \ra_{\a,\b},
\end{equation}
where the first two inner products are the ordinary inner products on the triangle whereas the third one is
the inner product of one variable with $f\vert_{z=0}(x,y) = f(x,y)\vert_{z=0} =  f(x,1-x)$. 

It is easy to see that it is indeed an inner product. Recall that $\CV_n(\varpi_{\a,\b,-1})$ denotes the space 
spanned by $J_{k,n}^{\a,\b,-1}$, $0 \le k \le n$.  

\begin{prop}\label{prop:00-1}
The space $\CV_n(\varpi_{\a,\b,-1})$ consists of orthogonal polynomials of degree $n$ with respect to 
$\la \cdot,\cdot\ra_{\a,\b,-1}$ and $\{J_{k,n}^{\a,\b,-1}, 0 \le k \le n\}$ is a mutually orthogonal basis of this
space.
\end{prop}

\begin{proof}
For $0 \le k \le n-1$,  $J_{k,n}^{\a,\b,-1}(x,y) = c (1-x-y)J_{k,n-1}^{\a,\b,1}(x,y)$ by \eqref{eq:J+J-}, which shows that 
$J_{k,n}^{\a,\b,-1}(x,1-x) =0$. Hence, for $0 \le k \le n-1$ and $0 \le l \le n$, integrating by parts shows that 
\begin{align*}
  \la J_{k,n}^{\a,\b,-1}, J_{l,n}^{\a,\b,-1} \ra_{\a,\b,-1}^J = - \lambda_{l,n}^{\a,\b,-1} 
    \int_{\triangle}  J_{k,n}^{\a,\b,-1}(x,y) J_{l,n}^{\a,\b,-1}(x,y) \varpi_{\a,\b,-1}(x,y) dxdy,
\end{align*}
where we have used \eqref{eq:eigen1}, which extends to $J_{n,n}^{\a,\b,-1}$ as well by a direct computation. 
If $0 \le k \ne l \le n-1$, then the above integral is zero by \eqref{eq:J+J-} and the orthogonality of $J_{k,n}^{\a,\b,1}$. 
If $l = n$, then the eigenvalue $\l_{n,n}^{\a,\b,-1} = 0$ so that the inner product is zero. 
\end{proof}

For $\a,\b,\g > -1$, we have two other mutually orthogonal bases for $\CV_n(\varpi_{\a,\b,\g})$, which follows immediately
from the invariance of the measure. The same holds for  $\CV_n(\varpi_{\a,\b,-1})$ but it does require a proof now.   

\begin{prop}\label{prop:Vn00-1}
The set $\{K_{k,n}^{\a,\b,-1}: 0 \le k \le n\}$ is a basis of $\CV_n(\varpi_{\a,\b,-1})$, so is
$\{L_{k,n}^{\a,\b,-1}: 0 \le k \le n\}$.
\end{prop}

\begin{proof}
We give the proof only for $K_{k,n}^{\a,\b,-1}$, the other case is similar. Let us write the inner product 
$\la f, g \ra_{\a,\b,-1}^J = \la f,g\ra_T + \la f,g\ra_B$, where the first one stands for the integral over $\triangle$  
and the second one for the integral over $[0,1]$. Let $g$ be a generic polynomial of degree at most $n-1$. 
By \eqref{eq:diffK} and the orthogonality of $K_n^{\a+1,\b,0}$ and $K_n^{\a,\b+1,0}$, it follows immediately 
that $\la K_{k,n}^{\a,\b,-1}, g\ra_T =0$. By \eqref{eq:K+K-}, $K_{k,n}^{\a,\b,-1}$ vanishes when $x + y =1$ for
$1 \le k \le n$, so that $\la K_{k,n}^{\a,\b,-1}, g\ra_B =0$ for $1 \le k \le n$. Finally, since $K_{0,n}^{\a,\b,-1}(x,1-x) =
 J_n^{\a,\b}(1-2x)$, it follows that $\la K_{0,n}^{\a,\b,-1}, g\ra_B =0$ as well. Since $K_{k,n}^{\a,\b,-1}$ are 
 polynomials of degree $n$ that are orthogonal with respect to $\la \cdot,\cdot \ra_{\a,\b,-1}^J$ to lower degree 
 polynomials, they form a basis of $\CV_n^d(\varpi_{\a,\b,-1})$. 
 \end{proof}
 
Although $\{K_{k,n}^{\a,\b,-1}: 0 \le k \le n\}$ is a basis of $\CV_n(\varpi_{\a,\b,1})$, it is not a mutually orthogonal one
and neither is $\{L_{k,n}^{\a,\b,-1}: 0 \le k \le n\}$. However, more can be said about those two bases. Let us define
\begin{align} \label{eq:ipd-ab1K}
  \la f, g \ra_{\a,\b,-1}^{K} :=    \la \partial_1 f, \partial_1 g \ra_{\a+1,\b,0} + \lambda \la f\vert_{z=0}, g\vert_{z=0} \ra_{\a,\b}.   
\end{align}

\begin{prop}\label{prop:00-1A}
The set $\{K_{k,n}^{\a,\b,-1}: 0 \le k \le n\}$ consists of a mutually orthogonal basis of $\CV_n(\varpi_{\a,\b,-1})$
under the inner product $\la \cdot, \cdot \ra_{\a,\b,-1}^K$. 
\end{prop} 
 
\begin{proof}
For $k \ge 1$, by \eqref{eq:J+J-}, $K_{k,n}^{\a,\b,-1}(x,y) =  c (1-x-y) K_{k-1,n-1}^{\a,\b,1}(x,y)$, which vanishes
if $x = 1-y$. By \eqref{eq:diffK}, $\partial_x K_{k,n}^{\a,\b,-1} (x,y)= K_{k-1,n-1}^{\a+1,\b,0}(x,y)$ for $k \ge 1$
and $\partial_x  K_{0,n}^{\a,\b,-1} (x,y) =0$ by the definition of $K_{0,n}^{\a,\b,-1}$. Hence, 
for $1 \le k \le n$ and $0 \le l \le m$, 
\begin{align*}
  &\la K_{k,n}^{\a,\b,-1}, K_{l,m}^{\a,\b,-1} \ra_{\a,\b,-1}^{K} 
     =  \la  K_{k-1,n-1}^{\a+1,\b,0}, K_{l-1,n-1}^{\a+1,\b,0}\ra_{\a+1,\b,0} = 0,
\end{align*}
if $n \ne m$ or if $n =m$ but $k \ne l$. This also holds if we exchange $k$ and $l$. In the remanning case of
$k=0$ or $l = 0$, we use $\partial_x K_{0,n}^{\a,\b,-1} (x,y) = 0$ and $K_{0,n}^{\a,\b,-1} (x,y) = J_n^{\a,\b}(2y-1)$, so that
the orthogonality follows from that of the Jacobi polynomials on $[0,1]$. This completes the proof. 
\end{proof}

In exactly the same way, we can define the inner product 
\begin{align} \label{eq:ipd-ab1L}
  \la f, g \ra_{\a,\b,-1}^{L} :=    \la \partial_2 f, \partial_2 g \ra_{\a,\b+1,0} + \lambda \la f\vert_{z=0}, g\vert_{z=0} \ra_{\a,\b}
\end{align}
and deduce the following result: 

\begin{prop}\label{prop:00-1B}
The set $\{L_{k,n}^{\a,\b,-1}: 0 \le k \le n\}$ consists of a mutually orthogonal basis of $\CV_n(\varpi_{\a,\b,-1})$
under the inner product $\la \cdot, \cdot \ra_{\a,\b,-1}^L$. 
\end{prop} 

We observe that $\la f,g\ra_{\a,\b,-1}^K+\la f,g\ra_{\a,\b,-1}^L$ is exactly $\la f,g \ra_{\a,\b,-1}^J$ with $2\l$ 
in place of $\l$.  It is interesting that the subspace $\CV_n(\varpi_{\a,\b,-1})$ is the space of orthogonal polynomials for 
three different inner products. This does not usually happen for ordinary (that is, non-Sobolev) inner products and it is 
useful in the study of orthogonal expansions. 

For $ f \in W_2^1(\varpi_{\a,\b,0})$, we can expand $f$ in terms of $J_{k,n}^{\a,\b,-1}$ and its inner product,
$$
   f = \sum_{n=0}^\infty  \mathrm{\proj_J}_n^{\a,\b,-1} f, \qquad \proj_n: W_2^1(\varpi_{\a,\b,0}) \mapsto \CV_n(\varpi_{\a,\b,-1}),
$$ 
where the projection operator 
\begin{equation} \label{eq:proj-ab1}
  \mathrm{\proj_J}_n^{\a,\b,-1} f = \sum_{k=0}^n \wh f_{k,n, J}^{\a,\b,-1} J_{k,n}^{\a,\b,-1}, \quad 
  \wh f_{k,n, J}^{\a,\b,-1} = \frac{\la  f, J_{k,n}^{\a,\b,-1}\ra_{\a,\b,-1}^J} {\la  J_{k,n}^{\a,\b,-1}, J_{k,n}^{\a,\b,-1}\ra_{\a,\b,-1}^J}
\end{equation}
We can also expand $f$ in terms of $K_{k,n}^{\a,\b,-1}$ and its inner product or in terms of  $L_{k,n}^{\a,\b,-1}$ and 
its inner product. Let us denote the corresponding projection operators by $ \mathrm{proj_K}_{n}^{\a,\b,-1}$ and
$ \mathrm{proj_L}_{n}^{\a,\b,-1}$, respectively. For example, 
$$
  \mathrm{proj_K}_{n}^{\a,\b,-1} f = \sum_{k=0}^n \wh f_{k,n,K}^{\a,\b,-1} K_{k,n}^{\a,\b,-1}, \quad 
  \wh f_{k,n,K}^{\a,\b,-1} = \frac{\la  f, K_{k,n}^{\a,\b,-1}\ra_{\a,\b,-1}^K} {\la  K_{k,n}^{\a,\b,-1}, K_{k,n}^{\a,\b,-1}\ra_{\a,\b,-1}^K}.
$$

Once an inner product is chosen, the projection operator is characterized by $\proj_n f \in \CV_n$ and $\proj_n f = f$ 
if $f \in \CV_n$. Since $\CV_n(\varpi_{\a,\b,-1})$ is the space of orthogonal polynomials for all three inner products, a  moment reflection shows that the following holds: 

\begin{prop} \label{prop:proj=A=B}
For $n \ge 0$ and $f \in W_2^1(\varpi_{\a,\b,0})$,  
$$
 \mathrm{proj_J}_n^{\a,\b,-1} f  =  \mathrm{proj_K}_{n}^{\a,\b,-1} f  =  \mathrm{proj_L}_{n}^{\a,\b,-1} f =: \proj_n^{\a,\b,-1}f, 
$$
where $=:$ means that we remove the subscript $J, K, L$ when they are equal. 
\end{prop}

As a consequence of this proposition, we can now deduce the commutativity of partial derivatives with 
the projection operators. 

\begin{thm}\label{thm:Dproj00-1}
For $f \in W_2^1(\varpi_{\a,\b,0})$,
\begin{align} \label{eqLDproj001}
\begin{split}
    \partial_1 \proj_n^{\a,\b,-1} f = & \proj_{n-1}^{\a+1,\b,0} \partial_1 f,\\ 
        \partial_2 \proj_n^{\a,\b,-1} f =&  \proj_{n-1}^{\a,\b+1,0} \partial_2 f. 
\end{split}
\end{align}
\end{thm}

\begin{proof} 
For $\partial_1$ we use the expansion in $K_{k,n}^{\a,\b,-1}$. If $k \ge 1$, then 
$\partial_x K_{k,n}^{\a,\b,-1}(x,y) =  K_{k-1,n-1}^{\a+1,\b,0}$ by \eqref{eq:diffK} and 
$$
 \la  K_{k,n}^{\a,\b,-1},  K_{k,n}^{\a,\b,-1}\ra_{\a,\b,-1}^K = 
  \la  K_{k-1,n-1}^{\a+1,\b,0},  K_{k-1,n-1}^{\a+1,\b,0}\ra_{\a+1,\b,0}^K,
$$
which implies immediately that $\wh f_{k,n, K}^{\a,\b,-1} =  \wh{\partial_1 f}_{k-1,n-1,K}^{\a+1,\b,0}$. 
If $k = 0$, then $\partial_x K_{0,n}^{\a,\b,-1}(x,y)=0$. Hence, it follows that
\begin{align*}
  \partial_x \proj_n^{\a,\b,-1}f(x,y) & = \sum_{k=1}^n  \wh f_{k,n,K}^{\a,\b,-1} K_{k-1,n-1}^{\a+1,\b,0}(x,y)
         =   \proj_{n-1}^{\a+1,\b,0}(\partial_1 f)(x,y). 
\end{align*}
The proof for $\partial_2$ is similar if we use $L_{k,n}^{\a,\b,-1}$. 
\end{proof}
  
\subsection{The spaces $\CV_n(\varpi_{\a,-1,\g})$ and $\CV_n(\varpi_{-1,\b,\g})$} \label{sect:a-1g}
The orthogonality in these two cases follows from simultaneous permutations of parameters and variables. Indeed,
the polynomials are related exactly as in \eqref{eq:3bases-transform} and the partial derivatives are related as in 
\eqref{eq:d-transform1} and \eqref{eq:d-transform2}, accordingly. The definitions of $\CW(\varpi_{\a,0, \g})$ and
$\proj_n^{\a,-1,\g}$ etc. should all be straightforward. We record the inner product and basic facts below for future
use without proof. 

For $\CV(\varpi_{\a,-1,\g})$, we define inner products 
\begin{align} \label{eq:ipd-a1g}
\begin{split}
\la f, g \ra_{\a,-1,\g}^K: = & \la \partial_2 f, \partial_2 g\ra_{\a,0,\g+1} +  \la \partial_3 f, \partial_3 g\ra_{\a+1,0,\g} +
  \l  \la f(\cdot,0),g(\cdot,0)\ra_{\a,\g},  \\ 
\la f, g \ra_{\a,-1,\g}^L: =  & \la \partial_2 f, \partial_2 g\ra_{\a,0,\g+1} +
     \l  \la f(\cdot,0),g(\cdot,0)\ra_{\a,\g},    \\
\la f, g \ra_{\a,-1,\g}^J: = &  \la \partial_3 f, \partial_3 g\ra_{\a+1,0,\g} +
  \l  \la f(\cdot,0),g(\cdot,0)\ra_{\a,\g}. 
\end{split}
\end{align}
Then families $\{K_{k,n}^{\a,-1,\g}: 0 \le k \le n\}$, $\{L_{k,n}^{\a,-1,\g}: 0 \le k \le n\}$, $\{J_{k,n}^{\a,-1,\g}: 0 \le k \le n\}$ 
are mutually orthogonal bases of 
$\CV_n(\varpi_{\a,-1,\g})$ with respect to the inner produc $\la \cdot, \cdot \ra_{\a,-1,\g}^K$, 
$\la \cdot, \cdot \ra_{\a,-1,\g}^L$, $\la \cdot, \cdot \ra_{\a,-1,\g}^J$, respectively.   

\begin{thm}
For $f \in W_2^1(\varpi_{\a,0,\g})$,
\begin{align*}
    \partial_2 \proj_n^{\a,-1,\g} f = &   \proj_{n-1}^{\a,0,\g+1} \partial_2 f, \\
      \partial_3 \proj_n^{\a,-1,\g} f = &  \proj_{n-1}^{\a+1,0,\g} \partial_3 f. 
\end{align*}
\end{thm}

For $\CV_n(\varpi_{-1,\b,\g})$, we define inner products 
\begin{align} \label{eq:ipd-1bg}
\begin{split}
\la f, g \ra_{-1,\b,\g}^L: = & \la \partial_3 f, \partial_3 g\ra_{0,\b+1,\g} +  \la \partial_1 f, \partial_1 g\ra_{0,\b,\g+1} +
  \l  \la f(0,\cdot),g(0,\cdot)\ra_{\b,\g},  \\ 
\la f, g \ra_{-1,\b,\g}^J: =  &\la \partial_3 f, \partial_3 g\ra_{0,\b+1,\g} \l  \la f(0,\cdot),g(0,\cdot)\ra_{\b,\g},   \\
\la f, g \ra_{-1,\b,\g}^K: = &   \la \partial_1 f, \partial_1 g\ra_{0,\b,\g+1}+
  \l   \la f(0,\cdot),g(0,\cdot)\ra_{\b,\g}. 
\end{split}
\end{align}
Then 
the families $\{L_{k,n}^{-1,\b,\g}: 0\le k \le n\}$, $\{J_{k,n}^{-1,\b,\g}: 0\le k \le n\}$, $\{K_{k,n}^{-1,\b,\g}: 0\le k \le n\}$ 
are orthogonal bases of $\CV_n(\varpi_{-1,\b,\g})$ with respect to the inner products $\la \cdot, \cdot \ra_{-1,\b,\g}^L$,
$\la \cdot, \cdot \ra_{-1,\b,\g}^J$, $\la \cdot, \cdot \ra_{-1,\b,\g}^K$, respectively. 

 
\begin{thm}
For $f \in W_2^1(\varpi_{0,\b,\g})$,
\begin{align*}
    \partial_3 \proj_n^{-1,\b,\g} f =  & \proj_{n-1}^{0,\b+1,\g} \partial_3 f, \\ 
      \partial_1 \proj_n^{-1,\b,\g} f = &  \proj_{n-1}^{0,\b,\g+1} \partial_1 f. 
\end{align*}
\end{thm}

\section{Sobolev orthogonality with two negative parameters}
\label{sect6}
\setcounter{equation}{0}

In this section we consider the case when there are two negative parameters among $\a,\b,\g$. Some of the
inner products in this section are defined in terms of the inner products in the previous section. We shall
use $\l  >0$ as a generic number, which could take different values even when it appears, explicitly or implicitly, 
in the same line. 
	
\subsection{The space $\CV_n(\varpi_{-1,-1,\g})$}\label{sec:-1-1g}
Let $W_2^1(\varpi_{-1,-1,\g})=\{f\in L^2: \partial_3 f \in L^2(\varpi_{0,0,\g})\}$. 
For $\g > -1$, we define an inner product for $W_2^1(\varpi_{-1,-1,\g})$ by
\begin{align} \label{eq:ipd11gJ}
  \la f, g \ra_{-1,-1,\g}^J: = \la \partial_3 f,\partial_3 g\ra_{0,0,\g}+ \l \la f(0,\cdot),g(0,\cdot)\ra_{\g,-1},
 \end{align}
where $\la f(0,\cdot),g(0,\cdot)\ra_{\g,-1}$ denotes the inner product of one variable, see \eqref{eq:ipd-a-m}, 
$$
     \la f,g\ra_{\g,-1}:=  \int_0^1   f'(y)  g'(y) (1-y)^{\g+1} dy + \l_0 f(0)g(0)
$$
applied on $f(0,\cdot)$ and $g(0,\cdot)$. This is indeed an inner product, since 
$\la f,f\ra_{-1,-1,\g}^J=0$ implies that $\partial_z f(x,y) =0$,
which implies that $f(x,y) = f_0(x+y)$, and $\la f_0, f_0\ra_{\g,-1}^J =0$. Since $\la \cdot, \cdot\ra_{\g,-1}$ is an 
inner product of one variable, it follows that $f_0$, hence $f$, is identically zero. 

We need the polynomials $J_{k,n}^{-1,-1,\g}$, which satisfies
\begin{align} \label{eq:J-1-1g}
 \begin{split}
 J_{0,n}^{-1,-1,\g}(x,y) & = - \frac1n (x+y)J_{0,n-1}^{0,0,\g}(x,y),   \quad J_{1,n}^{-1,-1,\g}(x,y)   = - x J_{0,n-1}^{0,0,\g}(x,y),\\
  J_{k,n}^{-1,-1,\g}(x,y) & =  \frac{1}{k(k-1)} x y  J_{k-2,n-2}^{1,1,\g}(x,y), \qquad 2 \le k \le n, \\
 \end{split}
\end{align}

\begin{prop} \label{prop:J-1-1g}
The space $\CV_n(\varpi_{-1,-1,g})$ consist of orthogonal polynomials of degree $n$ with respect to 
$\la \cdot,\cdot\ra_{-1,-1,\g}$ and $\{J_{k,n}^{-1,-1,\g}: 0 \le k \le n\}$ consists of a mutually orthogonal basis 
of this space.
\end{prop}

\begin{proof}
By \eqref{eq:diff1J},  $\partial_z J_{k,n}^{-1,-1,\g}(x,y) =  J_{k-1,n-1}^{0,0,\g}(x,y)$ for $1 \le k \le n$ and 
$\partial_z  J_{0,n}^{-1,-1,\g}(x,y) =0$. Furthermore, by \eqref{eq:J-1-1g}, $J_{k,n}^{-1,-1,\g}(0,y) = 0$ for 
$1 \le k \le n$.  Consequently, it follows that 
\begin{equation}\label{eq:ipdJ-1-1g-a}
  \la f, J_{k,n}^{-1,-1,\g}\ra_{-1,-1,\g}^J  = \la \partial_z f, J_{k-1,n}^{0,0,\g}\ra_{0,0,\g}^J, \qquad 1 \le k \le n.
\end{equation}
Together with $\partial_z J_{0,n}^{-1,-1,\g}(x,y) =0$, this shows that  
$\la J_{l,n}^{-1,-1,\g},  J_{k,n}^{-1,-1,\g}\ra_{-1,-1,\g}^J =0$ if
$k \ne l$, $0 \le l \le n$ and $1 \le k \le n$. Furthermore, $J_{0,n}^{-1,-1,\g}(0,y) = J_n^{-1,\g}(1-2y)$, which is
orthogonal with respect to $\la \cdot,\cdot\ra_{\g,-1}$, so that $J_{0,n}^{-1,-1,\g}$ is indeed orthogonal to
polynomials of lower degrees. This completes the proof.
\end{proof}

In the definition of $\la \cdot,\cdot\ra_{-1,-1,\g}$ we can replace the term of the integral over $[0,1]$ by 
$$
  \l_1 \int_0^1 f(x,0) g(x,0) (1-x)^{\g+1} dx +  \l_2 \int_0^1 f(0,y) g(0,y) (1-y)^{\g+1} dy, 
$$
which is the inner product defined in \cite{AX}. For this modified inner product, $\{J_{k,n}^{-1,-1,\g}: 0\le k \le n\}$
is still an orthogonal basis, as can be seen by the proof of Proposition \ref{prop:J-1-1g} and 
$$
  \partial_x  J_{0,n}^{-1,-1,\g}(x,0) =  - J_{n-1}^{0,\g+1}(1-2x), \quad 
  \partial_x  J_{1,n}^{-1,-1,\g}(x,0) =  n J_{n-1}^{0,\g+1}(1-2x), 
$$
but $J_{0,n}^{-1,-1,\g}$ and $J_{1,n}^{-1,-1,\g}$ are no longer  mutually orthogonal. 

Following the same argument of Proposition \ref{prop:Vn00-1}, we prove the following proposition:

\begin{prop} \label{prop:bases11g}
The set $\{K_{k,n}^{-1,-1,\g}: 0 \le k \le n\}$ consists of a basis for $\CV_n(\varpi_{-1,-1,\g})$ and so
is $\{L_{k,n}^{-1,-1,\g}: 0 \le k \le n\}$.
\end{prop}

\begin{proof}
We give a proof only for $K_{k,n}^{-1,-1,\g}$. Let $g$ be a generic polynomial of degree at most $n-1$.
By \eqref{eq:diffK}, it is easy to see that 
$\la K_{k,n}^{-1,-1,\g},g\ra_{-1,-1,\g}^J =0$ for $1 \le k \le n$. Moreover, by \eqref{eq:K+K-} and \eqref{eq:diffK}, 
it is not hard to verify that $\partial_y K_{k,n}^{-1,-1,\g}(0,y) = 0$ for $1\le k \le n$, which implies that 
$\la K_{k,n}^{-1,-1,\g},g\ra_{\g,-1} =0$ for $1 \le k \le n$. Finally, $K_{0,n}^{-1,-1,\g} (x,y) = J_n^{\g,-1}(2y-1)$,
which is the orthogonal polynomial for $\la \cdot,\cdot\ra_{\g,-1}$, so that 
$\la K_{0,n}^{-1,-1,\g},g\ra_{\g,-1} =0$ as well. Consequently, $K_{k,n}^{-1,-1,\g} \in \CV_n(\varpi_{-1,-1,\g})$. 
\end{proof}

As in the case of $(\a,\b,-1)$, we can define two other inner products for which $\{K_{k,n}^{-1,-1,\g}\}$ and 
$\{L_{k,n}^{-1,-1,\g}\}$ are mutually orthogonal, respectively. These inner products, however, involve second
order derivatives. For $\l >0$, we define
\begin{align} \label{eq:ipd11gK}
     \la f,g\ra_{-1,-1,\g}^K :=  \la \partial_1 f, \partial_1 g\ra_{0,-1,\g+1}^K +  \l \la f(0,\cdot),g(0,\cdot)\ra_{\g,-1} 
\end{align}
on the space $W_2^2$, where $\la \cdot,\cdot\ra_{-1,-1,\g}$ is defined in \eqref{eq:ipd11gK}. It is easy to see 
that this is an inner product. Indeed, if $\la f, f\ra_{-1,01,\g} =0$, then 
$f(x,y) = g(y) + x h(y)$ by $\partial_y \partial_x  f(x,y) = 0$ and  $\partial_z \partial_x f(x,y) =0$, but then 
$h(y)=0$ by $\partial_x f(0,y)=0$, and $g(y) = 0$ by  $\la g(0,y),g(0,y)\ra_{\g,-1} =0$.
 
The polynomials $K_{k,n}^{-1,-1,\g}$ are given by
$$
  K_{k,n}^{-1,-1,\g} (x,y) = (1-y)^k J_k^{\g,-1}\left(\frac{2x}{1-y} -1\right)J_{n-k}^{2k+\g,-1}(2y-1).
$$
In particular, it follows that 
$$
   K_{n-1,n}^{-1,-1,\g} (x,y) = \frac1{n} y J_{n-1}^{\g,1}(2y-1), \quad  K_{n,n}^{-1,-1,\g} (x,y)= \frac1{n} x(1-y)^{n-1}
       J_{n-1}^{\g,1}\left(\frac{2x}{1-y} -1\right).
$$
By \eqref{eq:K+K-}, $K_{k,n}^{-1,-1,\g}(0,0) =0$ for $0\le k \le n$. 

\begin{prop}
The system $\{K_{k,n}^{-1,-1,\g}: 0\le k \le n, \,\, n=0,1,2, ...\}$ is mutually orthogonal with respect to the
inner product $\la \cdot ,\cdot\ra_{-1,-1,\g}^K$.
\end{prop}

\begin{proof}
By \eqref{eq:K+K-}, it is easy to verify that $\partial_y K_{k,n}^{-1,-1,\g}(0,y) =0$ for $1\le k \le n$. Combining with 
$K_{k,n}^{-1,-1,\g}(0,0) =0$ for $0\le k \le n$ and $\partial_x K_{k,n}^{-1,-1,\g}(x,y) = K_{k-1,n-1}^{0,-1,\g+1}(x,y)$, 
we see immediately that for $1 \le k \le n$,
\begin{equation}\label{eq:K-1-1gA}
  \la K_{k,n}^{-1,-1,\g}, f\ra_{-1,-1,\g}^K = \la K_{k-1,n-1}^{0,-1,\g+1}, \partial_x f \ra_{0,-1,\g+1}^K. 
\end{equation}
In particular, it follows that $\la K_{k,n}^{-1,-1,\g}, K_{l,m}^{-1,-1,\g}, f\ra_{-1,-1,\g}^K =0$ when
either $k \ge 1$ or $l \ge 1$. In the remaining case of $k=l =0$, we note that $K_{0,n}^{-1,-1,\g}(x,y) =  J_n^{\g,-1}(2y-1)$,  
which is orthogonal with respect to the inner product $\la f, g\ra_{\g,-1}^K$. Together with  
$\partial_x  K_{0,n}^{-1,-1,\g}(x,y) =0$, this shows that $\la K_{0,n}^{-1,-1,\g}, K_{0,m}^{-1,-1,\g}\ra_{-1,-1,\g}^K = 0$ 
if $n \ne m$. 
\end{proof}

Similarly, for $\l > 0$, we can define
\begin{align} \label{eq:ipd11gL}
 \la f,g\ra_{-1,-1,\g}^L : =  \la \partial_2 f, \partial_2 g\ra_{-1,0,\g+1}^L +\l  \la f(\cdot,0),g(\cdot,0)\ra_{\g,-1}
\end{align}
on the space $W_2^2$, where $\la \cdot,\cdot\ra_{-1,0,\g+1}$ is defined in \eqref{eq:ipd-1bg}. 
As an analogue of the case $\la \cdot, \cdot\ra_{-1,-1,\g}^K$,  the following holds: 

\begin{prop}
The system $\{L_{k,n}^{-1,-1,\g}: 0\le k \le n, \,\, n=0,1,2, ...\}$ is mutually orthogonal with respect to the
inner product $\la \cdot ,\cdot\ra_{-1,-1,\g}^L$.
\end{prop}

For $ f \in W_2^1(\varpi_{0,0,\g})$, we can expand $f$ in terms of $J_{k,n}^{-1,-1,\g}$ as 
$$
   f = \sum_{k=0}^\infty \mathrm{\proj_J}_n^{-1,-1,\g} f, \qquad \proj_n: W_2^1(\varpi_{0,0,\g}) 
    \mapsto \CV_n(\varpi_{-1,-1,\g}),
$$ 
where the projection operator is defined by
\begin{equation}\label{eq:Fourier11g}
\mathrm{\proj_J}_n^{-1,-1,\g} f = \sum_{k=0}^n \wh f_{k,n,J}^{-1,-1,\g} J_{k,n}^{-1,-1,\g}, \quad 
\wh f_{k,n,J} = \frac{\la  f, J_{k,n}^{-1,-1,\g}\ra_{-1,-1,\g}^J} {\la  J_{k,n}^{-1,-1,\g}, J_{k,n}^{-1,-1,\g}\ra_{-1,-1,\g}^J}.
\end{equation}
We can also define the orthogonal expansions and the projection operators in terms of $\la\cdot, \cdot\ra_{-1,-1,\g}^K$ 
and $\la\cdot, \cdot\ra_{-1,-1,\g}^L$, respectively, in a similar manner. The corresponding projection operators are 
denoted by $\proj_{n,K}^{-1,-1,\g}$ and $\proj_{n,L}^{-1,-1,\g}$, respectively. Since all three families of polynomials 
are bases of $\CV_n(\varpi_{-1,-1,\g})$, it follows that 
$$
      \mathrm{\proj_J}_n^{-1,-1,\g} f = \mathrm{\proj_K}_n^{-1,-1,\g} f = \mathrm{\proj_L}_n^{-1,-1,\g} f=:\proj_n^{-1,-1,\g} f,
$$
where $=:$ means that we remove the subscripts $J, L, K$when these three are equal. 

\begin{thm} \label{thm:proj-1-1g}
For $f \in W_2^1(\varpi_{0,0,\g})\cap W_2^2$,
\begin{align*}
      \partial_1 \proj_n^{-1,-1,\g} f & =  \proj_{n-1}^{0,-1,\g+1} \partial_1 f, \\
      \partial_2 \proj_n^{-1,-1,\g} f & =  \proj_{n-1}^{-1,0,\g+1} \partial_2 f, \\
      \partial_3 \proj_n^{-1,-1,\g} f & =  \proj_{n-1}^{0,0,\g} \partial_3 f. 
\end{align*}
\end{thm}

\begin{proof}
We prove the case $\partial_3$ and will suppress $J$ in the subscript and the superscript. Taking derivative and 
using $\partial_3 J_{k,n}^{-1,-1,\g} = J_{k-1,n-1}^{0,0,g}(x,y)$ for $1 \le k \le n$ and $=0$ for $k=0$, we obtain 
\begin{equation} \label{proj3f-1-1g}
   \partial_3 \proj_n^{-1,-1,\g} f(x,y) =  \sum_{k=1}^{n} \wh f_{k,n}^{-1,-1,\g} J_{k-1,n}^{0,0,\g}
    =  \sum_{k=0}^{n-1} \wh f_{k+1,n}^{-1,-1,\g} J_{k,n}^{0,0,\g}.
\end{equation}
For $1 \le k \le n$, setting $f = J_{k,n}^{-1,-1,\g}$ in \eqref{eq:ipdJ-1-1g-a} shows that 
$$
  \la J_{k,n}^{-1,-1,\g}, J_{k,n}^{-1,-1,\g} \ra_{-1,-,1,\g} =    
 \la J_{k-1,n-1}^{0,0,\g}, J_{k-1,n-1}^{0,0,\g} \ra_{0,0,\g}, 
$$
which implies that $\wh f_{k+1,n}^{-1,-1,\g} =  \wh {\partial_3 f}_{k,n-1}^{0,0,\g}$ for 
$0 \le  k \le n-1$. Substituting these relations into \eqref{proj3f-1-1g} completes the proof. The proof
of the other two cases are similar, using the identity \eqref{eq:K-1-1gA} and its analogues for 
$\la \cdot,\cdot\ra_{-1,-1,\g}^K$ and $\la \cdot,\cdot\ra_{-1,-1,\g}^L$. 
\end{proof}

\subsection{The spaces $\CV_n(\varpi_{-1,\b,-1})$ and $\CV_n(\varpi_{\a,-1,-1})$}\label{sect:-1b-1}
As in the Subsection \ref{sect:a-1g}, we only record, without proof, the basic fact that will be needed later. 

For $\CV(\varpi_{-a,\b,-1})$, we define inner product
\begin{align} \label{eq:ipd1b1}
\begin{split}
  \la f, g \ra_{-1,\b,-1}^K: = &  \la \partial_1 f, \partial_1 g\ra_{0,\b,0} + \l \la f\vert_{z=0}, g\vert_{z=0}\ra_{-1,\b},
  \\
 \la f,g\ra_{-1,\b,-1}^L: = & \la \partial_2 f, \partial_2 g\ra_{-1,\b+1,0} + \l \la f\vert_{z=0}, g\vert_{z=0}\ra_{-1,\b}, \\
 \la f,g\ra_{-1,\b,-1}^J: =& \la \partial_3 f, \partial_3 g\ra_{0,\b+1,-1} + \l \la f(0,\cdot), g(0,\cdot)\ra_{-1,\b},
\end{split}
\end{align}
where $\la \cdot,\cdot\ra_{0,\b,0}$ is the ordinary inner product on the triangle, $\la \cdot,\cdot\ra_{-1,\b+1,0}$
and $\la \cdot,\cdot\ra_{0,\b+1,-1}$ are defined in \eqref{eq:ipd-1bg} and \eqref{eq:ipd-ab1J}, respectively; and
$\la f\vert_{z=0}, g\vert_{z=0}\ra_{-1,\g}$ is the inner product \eqref{eq:ipd-a-m} with $m=1$ applied on 
$f\vert_{z=0}(x,y) = f(x,1-x)$ and $g\vert_{z=0}(x,y) = g(x,1-x)$, and  $\la \cdot,\cdot\ra_{-1,\g}$
is the inner product \eqref{eq:ipd-ell-b} applied on $f(0,y)$ and $g(0,y)$.
Then 
the families $\{K_{k,n}^{-1,\b,-1}: 0 \le k \le n\}$, $\{L_{k,n}^{-1,\b,-1}: 0 \le k \le n\}$, $\{J_{k,n}^{-1,\b,-1}: 0 \le k \le n\}$
are mutually orthogonal bases of $\CV_n(\varpi_{-1,\b,-1})$ with respect to the inner product 
$\la \cdot, \cdot \ra_{-1,\b,-1}^K$, $\la \cdot, \cdot \ra_{-1,\b,-1}^L$, $\la \cdot, \cdot \ra_{-1,\b,-1}^J$, respectively. 

 
\begin{thm}
For $f \in W_2^1(\varpi_{0,\b,0})$,
\begin{align*}
      \partial_1 \proj_n^{-1,\b,-1} f & =  \proj_{n-1}^{0,\b, 0} \partial_1 f, \\
      \partial_2 \proj_n^{-1,\b,-1} f & =  \proj_{n-1}^{-1,\b+1,0} \partial_2 f, \\
      \partial_3 \proj_n^{-1,\b,-1} f & =  \proj_{n-1}^{0,\b+1,-1} \partial_3 f. 
\end{align*}\end{thm}

For the space $\CV_n(\varpi_{\a,-1,-1})$, we define inner product 
\begin{align} \label{eq:ipd-a11}
\begin{split}
  \la f, g \ra_{\a,-1,-1}^L: = &  \la \partial_2 f, \partial_2 g\ra_{\a,0,0} + \l \la f(\cdot,0), g(\cdot,0)\ra_{-1,\a},
  \\
\la f,g\ra_{\a,-1,-1}^J: = & \la \partial_3 f, \partial_3 g\ra_{\a+1,0,-1} + \l  \la f(\cdot,0), g(\cdot,0)\ra_{-1,\a}, \\
\la f,g\ra_{\a,-1,-1}^K: =& \la \partial_1 f, \partial_1 g\ra_{\a+1,-1,0} + \l  \la f\vert_{z=0}, g\vert_{z=0}\ra_{-1,\a},
\end{split}
\end{align}
where $\la \cdot,\cdot\ra_{\a,0,0}$ is the ordinary inner product on the triangle, $\la \cdot,\cdot\ra_{\a+1,0,-1}$
and $\la \cdot,\cdot\ra_{\a+1,-1,0}$ are defined in \eqref{eq:ipd-ab1J} and \eqref{eq:ipd-a1g}, respectively, 
and $\la \cdot,\cdot\ra_{-1,\g}$ is the inner product of one variable \eqref{eq:ipd-ell-b} with $\ell =1$,
and $\la f\vert_{z=0}, g\vert_{z=0}\ra_{-1,\a}$ is the inner product of one variable \eqref{eq:ipd-a-m} applied 
on $f\vert_{z=0}(x) = f(x,1-x), g\vert_{z=0}(x) = f(x,1-x)$.
Then 
the families $\{L_{k,n}^{-\a,-1,-1}: 0 \le k \le n\}$, $\{J_{k,n}^{-\a,-1,-1}: 0 \le k \le n\}$, $\{K_{k,n}^{-\a,-1,-1}: 0 \le k \le n\}$
are mutually orthogonal bases of $\CV_n(\varpi_{-\a,-1,-1})$ with respect to the inner product 
$\la \cdot, \cdot \ra_{-\a,-1,-1}^L$, $\la \cdot, \cdot \ra_{-\a,-1,-1}^J$, $\la \cdot, \cdot \ra_{-\a,-1,-1}^K$, respectively. 
 
 
\begin{thm}
For $f \in W_2^1(\varpi_{\a,0,0})$,
\begin{align*}
      \partial_1 \proj_n^{\a,-1,-1} f & =  \proj_{n-1}^{\a+1,-1,0} \partial_1 f, \\
      \partial_2 \proj_n^{\a,-1,-1} f & =  \proj_{n-1}^{-\a, 0, 0} \partial_2 f, \\
      \partial_3 \proj_n^{\a,-1,-1} f & =  \proj_{n-1}^{\a+1,0,-1} \partial_3 f. 
\end{align*}
\end{thm}

\section{Sobolev orthogonality with parameter $(-1,-1,-1)$}
\label{sect7}
\setcounter{equation}{0}

For positive $\l$, we define an inner product for $W_2^2$ by 
\begin{align} \label{eq:ipd111J} 
  \la f, g \ra_{-1,-1,-1}^J: =  \la \partial_3 f, \partial_3 g\ra_{0,0,-1}^J + \la f(0,\cdot), g(0,\cdot) \ra_{-1,-1},
\end{align}
where $\la \cdot,\cdot\ra_{-1,-1}$ is the inner product of one variable in \eqref{eq:ipd-a-m} with $\ell = m =1$,
$$  
      \la f,g\ra_{-1,-1}:= \int_0^1  f'(t) g'(t) dt + \lambda_0 f(0) g(0) 
$$
applied to $f(0,y) g(0,y)$. It is easy to see that this is indeed an inner product. 

Let $\CV_n(\varpi_{-1,-1,-1})$ be the space of orthogonal polynomials of degree $n$ with respect to this inner
product. We need the polynomials $J_{k,n}^{-1,-1,-1}$, which satisfies, for $n \ge 2$, 
\begin{align} \label{eq:J-1-1-1}
 \begin{split}
 J_{0,n}^{-1,-1,-1}(x,y) & = (x+y)J_{0,n-1}^{0,0,-1}(x,y),   \quad J_{1,n}^{-1,-1,-1}(x,y)   = - x J_{0,n-1}^{0,0,-1}(x,y),\\
  J_{k,n}^{-1,-1,-1}(x,y) & =  -\frac{x y (1-x-y) }{(n-k)k(k-1)}J_{k-2,n-3}^{1,1,1}(x,y), \qquad 2 \le k \le n-1, \\
  J_{n,n}^{-1,-1,-1}(x,y) & = \frac{-1}{n(n-1)} x y   J_{n-2,n-2}^{1,1,0}(x,y).
\end{split}
\end{align}

\begin{prop} \label{prop:J-1-1-1}
The set $\{J_{k,n}^{-1,-1,-1}: 0 \le k \le n\}$ consists of a mutually orthogonal basis of $\CV_n(\varpi_{-1,-1,-1})$. 
\end{prop}

\begin{proof}
From \eqref{eq:J-1-1-1}, $J_{k,n}^{-1,-1,-1}(0,0) =0$ as long as $n \ge 1$. From \eqref{eq:J-1-1-1}, it is easy
to see that $\partial_y J_{k,n}^{-1,-1,-1}(y,0) = 0$ for $1 \le k \le n$. Hence, it follows readily that, for $1 \le k \le n$, 
\begin{align} \label{eq:ipdJ-1-1-1a}
  \la f, J_{k,n}^{-1,-1,-1} \ra_{-1,-1,-1}^J = \la \partial_3 f, J_{k-1,n-1}^{0,0,-1} \ra_{0,0,-1}^J. 
\end{align}
Now, since $\partial_z J_{0,n}^{-1,-1,-1}(x,y) =0$, this implies immediately that $J_{k,n}^{-1,-1,-1}$ is mutually 
orthogonal to $J_{l,n}^{-1,-1,-1}$ if $1 \le k \le n$ and $0 \le l \le n$ or $0\le k \le n$ and $1 \le l \le n$. In the 
remaining case of $k = l= 0$, using $\partial_z J_{0,n}^{-1,-1,-1}(x,y) =0$ and 
$J_{0,n}^{-1,-1,-1}(0,y) = J_n^{-1,-1}(1-2y)$, we deduce that 
\begin{align} \label{eq:ipdJ-1-1-1b}
  \la f, J_{0,n}^{-1,-1,-1} \ra_{-1,-1,-1}^J = \la f, J_n^{-1,-1}(1-2\{\cdot\}) \ra_{-1,-1}. 
\end{align}
Since $J_n^{-1,-1}(1-2y)$ are orthogonal polynomials with respect to $\la \cdot, \cdot\ra_{-1,-1}$, this completes the proof.  
\end{proof}

\begin{prop}\label{prop:KL-1-1-1basis}
For $n \ge 2$, both $\{K_{k,n}^{-1,-1,-1}:0 \le k \le n\}$ and $\{L_{k,n}^{-1,-1,-1}:0 \le k \le n\}$ are bases of 
$\CV_n(\varpi_{-1,-1,-1})$. 
\end{prop}

\begin{proof}
We give a proof only for $K_{k,n}^{-1,-1,-1}$. Let $g$ be a generic polynomial of degree at most $n-1$. 
By \eqref{eq:diffK}, $\partial_z K_{1,n}^{-1,-1,-1} (x,y) = K_{1,n-1}^{0,0,-1}(x,y)$ and 
$$
  \partial_z K_{k,n}^{-1,-1,-1} (x,y) = -\frac{n+k-1}{2(2k-1)} K_{k-1,n-1}^{0,0,-1}(x,y)+K_{k,n-1}^{0,0,-1}(x,y), \quad k \ne 1.
$$
Hence, we see that $\la \partial_z K_{k,n}^{-1,-1,-1}, \partial_z g\ra_{0,0,-1}^J =0$ by Proposition
\ref{prop:Vn00-1}. Now, by \eqref{eq:K+K-}, it is easy to check that $K_{k,n}^{-1,-1,-1}(0,0) = 0$ for $0 \le k \le n$ 
and $n \ge 2$. Moreover, by \eqref{eq:diffK} again and by 
\eqref{eq:K+K-}, 
it is not hard to verify that $\partial_y K_{k,n}^{-1,-1,-1}(0,y) = 0$  and  for $2\le k \le n$, which
shows that $\la K_{k,n}^{-1,-1,-1}(0,\cdot),g\ra_{-1,-1} =0$ for $2 \le k\le n$. For the remaining two cases, 
we have $\partial_y K_{1,n}^{-1,-1,-1}(0,y) = n J_n^{0,0}(2y-1)$ and $\partial_y K_{1,n}^{-1,-1,-1}(0,y) =
 J_n^{0,0}(2y-1)$, which shows that  $\la K_{k,n}^{-1,-1,-1}(0,\cdot),g\ra_{-1,-1} =0$ holds for $k=0$ and $1$ 
 as well. Putting together, we have shown that $\la  K_{k,n}^{-1,-1,-1}, g\ra_{-1,-1,-1}^J =0$ for $0 \le k \le n$
 and $n \ge 2$, so that $K_{k,n}^{-1,-1,-1} \in \CV_n(\varpi_{-1,-1,-1})$. The proof for $L_{k,n}^{-1,-1,-1}$ is 
 similar. 
\end{proof}

We can define two more inner products on $W_2^2$ for which $K_{k,n}^{-1,-1,-1}$ and $L_{k,n}^{-1,-1,-1}$
are mutually orthogonal, respectively. Let
\begin{align} \label{eq:ipd111KL}
\begin{split}
  \la f, g \ra_{-1,-1,-1}^K: = & \la \partial_1 f, \partial_1 g\ra_{0,-1,0}^K + \la f\vert_{z=0}, g\vert_{z=0} \ra_{-1,-1}, \\
  \la f, g \ra_{-1,-1,-1}^L: = & \la \partial_1 f, \partial_1 g\ra_{-1,0,0}^L + \la f(\cdot,0), g(\cdot,0) \ra_{-1,-1},
\end{split}
\end{align}
where $ \la f\vert_{z=0}, g\vert_{z=0} \ra_{-1,-1}$ is the inner product $\la \cdot,\cdot\ra_{-1,-1}$ applied on $f(x,y)|_{z=0}$,
$$
\la f\vert_{z=0}, g\vert_{z=0} \ra_{-1,-1}:= \int_0^1 \partial_3 f(x,1-x) \partial_3 g(x,1-x)dx + \l_0 f(0,1)g(0,1).
$$
and $\la \cdot,\cdot\ra_{-1,-1}$ is the inner product of one variable in \eqref{eq:ipd-a-m} with $\ell = m =1$,
$$  
      \la f,g\ra_{-1,-1}:= \int_0^1  f'(t) g'(t) dt + \lambda_0 f(1) g(1).  
$$

\begin{prop}\label{prop:KL-1-1-1}
The set $\{K_{k,n}^{-1,-1,-1}: 0 \le k \le n\}$ and $\{L_{k,n}^{-1,-1,-1}: 0 \le k \le n\}$ are mutually orthogonal bases of 
$\CV_n(\varpi_{-1,-1,-1})$ with respect to the inner product $\la \cdot,\cdot\ra_{-1,-1,-1}^K$ and 
$\la \cdot,\cdot\ra_{-1,-1,-1}^L$, respectively. 
\end{prop}

\begin{proof}
By their definitions, both $K_{k,n}^{-1,-1,-1}$ and $L_{k,n}^{-1,-1,-1}$ are defined in terms of $J_{k,n}^{-1,-1,-1}$
by permutations of $(x,y,1-x-y)$, from which the stated results follow from Proposition \ref{prop:J-1-1-1}. 
\end{proof}

For $ f \in W_2^2$, we can expand $f$ in terms of $J_{k,n}^{-1,-1,-1}$ and its inner product,
$$
   f = \sum_{k=0}^\infty  \mathrm{proj_J}_{n}^{-1,-1,-1} f, \qquad \mathrm{proj_J}_{n}: W_2^1(\varpi_{-1,-1,-1}) 
     \mapsto \CV_n(\varpi_{-1,-1,-1}),
$$ 
where the projection operator is defined by, with $(\a,\b,\g) = (-1,-1,-1)$,
\begin{equation}\label{eq:Fourier111}
   \mathrm{proj_J}_{n}^{-1,-1,-1} f = \sum_{k=0}^n \wh f_{k,n,J}^{-1,-1,-1} J_{k,n}^{-1,-1,-1}, 
\end{equation}
with $\wh f_{k,n,J}^{-1,-1,-1}$ defined in terms of $J_{k,n}^{-1,-1,-1}$. 
We can also expand $f$ in terms of $K_{k,n}^{-1,-1,-1}$ and its inner product or in terms of  $L_{k,n}^{-1,-1,-1}$ and 
its inner product. Let us denote the corresponding projection operators by $ \mathrm{proj_K}_{n}^{-1,-1,-1}$ and
$ \mathrm{proj_L}_{n}^{-1,-1,-1}$, respectively. Since $\CV_n(\varpi_{-1,-1,-1})$ is the space of orthogonal polynomials 
for all three inner products for $n \ge 2$, these projection operators are equal. 

\begin{prop} \label{prop:proj=A=B-1-1-1}
For $n \ge 2$ and $f \in W_2^1$,  
$$
  \mathrm{proj_J}_{n}^{-1,-1,-1} f  =  \mathrm{proj_K}_{n}^{-1,-1,-1} f  =  \mathrm{proj_L}_{n}^{-1,-1,-1} f =:
  \proj_{n}^{-1,-1,-1} f. 
$$
\end{prop}
 
We can again drop the subscript $J, K, L$ when the three terms are equal, and we shall do so also for 
the corresponding partial sum operators by $S_n^{-1,-1,-1}f$. We are now ready to show that derivatives and the 
partial sum operators commute. 

\begin{thm}\label{thm:Dproj-1-1-1}
For $n \ge 1$ and $f \in W_2^2$, 
\begin{align*}
  &  \partial_1 S_n^{-1,-1,-1} f  = S_{n-1}^{0,-1,0} \partial_1 f,\\
  &  \partial_2 S_n^{-1,-1,-1} f  = S_{n-1}^{-1,0,0} \partial_2 f,  \\ 
  &  \partial_3 S_n^{-1,-1,-1} f  = S_{n-1}^{0,0,-1} \partial_3 f.
\end{align*}
\end{thm}
 
\begin{proof}
Since $f - S_n^{-1,-1,-1}f =  \sum_{m=n+1} \proj_m^{-1,-1,-1} f$, it is sufficient to prove that, for $n \ge 2$ and $f \in W_2^1$, 
\begin{align*}
  &  \partial_1 \proj_n^{-1,-1,-1} f  = \proj_{n-1}^{0,-1,0} \partial_1 f,\\
  &  \partial_2 \proj_n^{-1,-1,-1} f  = \proj_{n-1}^{-1,0,0} \partial_2 f,  \\ 
  &  \partial_3 \proj_n^{-1,-1,-1} f  = \proj_{n-1}^{0,0,-1} \partial_3 f.
\end{align*}
In the case of $\partial_3$, we suppress $J$ in the subscript and superscript. Since 
$\partial_3 J_{k,n}^{-1,-1,-1} = J_{k-1,n-1}^{0,0,-1}$ for $1 \le k\le n$ and $=0$ for $k=0$, and
$\wh f_{k,n}^{-1,-1,-1} =  \wh {\partial_3 f}_{k-1,n-1}^{0,0,-1}$ by \eqref{eq:ipdJ-1-1-1a}, the third identity 
on the projection operator follows immediately. For $\partial_1$ and $\partial_2$, we use the formula of 
the projection operator in terms of $K_{k,n}^{-1,-1,-1}$ and  $L_{k,n}^{-1,-1,-1}$, respectively. 
\end{proof}

\section{Sobolev orthogonality with parameters $(\a,\b,-2)$ and its permutations}
\label{sect8}
\setcounter{equation}{0}

In this section we consider the orthogonality when the parameters are $(\a,\b,-2)$. This turns out to be a particular 
subtle case. For the space $\CV_n(\varpi_{\a,\b,-2})$ spanned by $J_{k,n}^{-1,-1,-2}$, we failed to find, after 
arduous effort, a full-fledged inner product that satisfies both our requirements; namely, has an orthonormal basis in 
$\CV_n(\varpi_{\a,\b,-2})$ and has its partial sum operators commute with partial derivatives. There are many 
inner products in $W_2^2$ that will have a slight modification of $J_{k,n}^{-1,-1,-2}$ as an orthogonal basis, but it
is not good enough for our purpose. What we found at the end almost works and it is sufficient for our study in 
the next section. 

\subsection{The space $\CV_n(\varpi_{\a, \b,-2})$}
For $\a,\b> -1$, we define 
\begin{align} \label{eq:ipd-ab2}
 \la f,g\ra_{\a,\b,-2}^J: =   \la f,g\ra_{\a,\b,-2}^\circ  & + \int_0^1 \partial_x f(x,1-x) \partial_x g(x,1-x) x^{\a+1}(1-x)^\b dx   \\
     & + \int_0^1 \partial_y f(x,1-x)  \partial_y g(x,1-x) x^\a (1-x)^{\b+1} dx, \notag
\end{align}
where 
\begin{align*}
  \la f,g\ra_{\a,\b,-2}^\circ := & \int_{\triangle} (x^2 \partial_x^2 f(x,y) \partial_x^2 g(x,y) +
 2 xy \partial_x\partial_y f(x,y) \partial_x\partial_y g(x,y) \\
   & \qquad  + y^2 \partial_y^2 f(x,y) \partial_y^2 g(x,y)) x^\a y^\b dxdy.
\end{align*}
Notice that this is not an inner product, since $\la f,f\ra_{\a,\b,-2} =0$ implies that $f(x,y)$ is a constant. However, it is an
inner product when applying to the space of functions that is zero at the origin. Indeed, if $\la f,f\ra_{\a,\b,-2} =0$,
then $\la f,f\ra_{\a,\b,-2}^\circ =0$ implies that $f(x,y)$ is a linear polynomial. Hence, by the remaining part of
$\la f,f \ra_{\a,\b,-2}$, it has to be a constant. 

We work with $J_{k,n}^{\a,\b,-2}$, which are given by
\begin{align*} 
\begin{split}
J_{k,n}^{\a,\b,-2}(x,y) & = (x+y)^{k} J_{k}^{\a,\b}\left(\frac{y-x}{x+y}\right) J_{n-k}^{2k+\a+\b+1,-2}(1-2x-2y), \quad 0 \le k \le n-2, \\
J_{n-1,n}^{\a,\b,-2}(x,y) &=(x+y)^{n-1} J_{n-1}^{\a,\b}\left(\frac{y-x}{x+y}\right) \left (\frac{2n+\a+\b}{2n+\a+\b-1} - x- y\right), \\
J_{n,n}^{\a,\b,-2}(x,y) &=(x+y)^n J_{n}^{\a,\b}\left(\frac{y-x}{x+y}\right).
\end{split}
\end{align*}
First we prove a lemma that will be needed in the sequel. 

\begin{lem}
For $\a,\b > -1$,  $J_{k,n}^{\a,\b,-2}$ are mutually orthogonal with respect to $\la \cdot,\cdot,\ra_{\a,\b,-2}^\circ$. Moreover,
let $H_{k,n}^{\a,\b, -2}: =  \la J_{k,n}^{\a,\b,-2},J_{k,n}^{\a,\b,-2}\ra_{\a,\b,-2}^\circ$; then
$$
   H_{k,n}^{\a,\b, -2}: = h_k^{\a,\b} \begin{cases}
        \displaystyle{ \frac{(2n+\a+\b)(\a+\b+1)_{n+k+1}^2}{(\a+\b+1)_{2n}^2}}, & 0 \le k \le n-2,   \vspace{.05in} \\
          \displaystyle{ \frac{ (n-1) (3n+2\a+2\b)}{2n-1+\a+\b}}, & k = n-1, \\
            n(n-1)(2n+\a+\b+1), & k=n, 
      \end{cases}
$$
where $h_k^{\a,\b}$ is defined in \eqref{eq:normJacobi}. 
\end{lem}

\begin{proof}
Integration by parts twice for each of the three terms in $\la \cdot,\cdot \ra_{\a,\b,-2}^\circ$, which leaves no boundary
terms because of the weight functions, we see that 
\begin{align*}
  \la f, J_{k,n}^{\a,\b,-2} \ra_{\a,\b,-2} = \int_{\triangle} \CD^{\a,\b,-2} f(x,y)  J_{k,n}^{\a,\b,-2}(x,y) dxdy.
\end{align*}
where 
$$
\CD^{\a,\b,-2}: = \partial_x^2 X^{\a+2,\b,0} + 2 \partial_x\partial_y X^{\a+1,\b+1,0} \partial_x\partial_y +\partial_y^2 X^{\a,\b+2,0} \partial_y^2. 
$$
Applying \eqref{eq:eigen3} with $\g = -2$, we then obtain that 
$$
 \la J_{l,m}^{\a,\b,-2}, J_{k,n}^{\a,\b,-2} \ra_{\a,\b,-2} = \s_{k,n}^{\a,\b,-2} \int_{\triangle} J_{l,m}^{\a,\b,-2}(x,y)J_{k,n}^{\a,\b,-2}(x,y)\varpi_{\a,\b,-2}(x,y) dxdy,
$$
where $\s_{k,n}^{\a,\b,-2} = (n-k-1)_2(n+k+\a+\b)_2$, as can be verified from the formula of $\s_{k,n}^{\a,\b,\g}$. 
Using the explicit formula of $J_{k,n}^{\a,\b,-2}$, we can reduce the evaluation of the last integrals to the product of two 
integrals of one variable. This way we verify that the above quantity is zero if $k\ne l$ or $n\ne m$. Moreover, using the
norm of ordinary Jacobi polynomials in \eqref{eq:normJacobi} and a direct computation when $k =1$, we can then evaluate 
the formula for $H_{k,n}^{\a,\b,-2}$. We omit the details. 
\end{proof}

By \eqref{eq:J+J-}, $J_{k,n}^{\a,\b,-2}$ vanishes when $x+y=1$ for $0\le k \le n-2$. We need to modify the system 
$J_{k,n}^{\a,\b,-2}$ to obtain a mutually orthogonal basis for $\la \cdot, \cdot\ra_{\a,\b,-2}$. We start with the following lemma: 

\begin{lem}
For $n \ge 2$, define 
\begin{align*}
  F_n^{\a,\b,-2}(x,y) & = J_{n-1,n}^{\a,\b,-2}(x,y) -  \frac{2n+\a+\b}{n+\b} J_{n,n}^{\a,\b,-2}(x,y)\\ 
  G_n^{\a,\b,-2}(x,y) & = J_{n-1,n}^{\a,\b,-2}(x,y) +  \frac{2n+\a+\b}{n+\a} J_{n,n}^{\a,\b,-2}(x,y).
\end{align*}
Then $F_n^{\a,\b,-2}, G_n^{\a,\b,-2} \in \CV_n(\varpi_{\a,\b,-2})$ satisfy
\begin{align*}
 & \partial_x F_n^{\a,\b,-2}(x,1-x) =0, \qquad  \partial_y F_n^{\a,\b,-2}(x,1-x) = - \frac{2n+\a+\b}{n+\b} J_{n-1}^{\a,\b+1}(1-2x),\\
  & \partial_y G_n^{\a,\b,-2}(x,1-x) =0, \qquad  \partial_x G_n^{\a,\b,-2}(x,1-x) = - \frac{2n+\a+\b}{n+\a} J_{n-1}^{\a+1,\b}(1-2x). 
\end{align*}
\end{lem}

\begin{proof}
These follow from identities in \eqref{eq:diff1J}, which give, for example, that 
\begin{align*}
 \partial_x F_n^{\a,\b,-2}(x,y)& = - a_{n-1,n}^{\a,\b} J_{n-2,n-1}^{\a+1,\b,-1}(x,y),\\
 \partial_y F_n^{\a,\b,-2}(x,y)& = a_{n-1,n}^{\b,\a} J_{n-2,n-1}^{\a,\b+1,-1}(x,y) - \left(1+\frac{n+\a}{n+\b}\right)
      J_{n-1,n}^{\a,\b+1,-1}(x,y). 
\end{align*}
The stated results then follow from \eqref{eq:J+J-}. 
\end{proof}

We define a system of polynomials 
$\{\wh J_{k,n}^{\a,\b,-2}: 0\le k\le n, \, n=0,1,2,\ldots\}$ as follows: 
$$
 \wh J_{0,0}^{\a,\b,-2}(x,y) =1,\quad \wh J_{0,1}^{\a,\b,-2}(x,y) = x+y, \quad \wh J_{1,1}^{\a,\b,-2}(x,y) = x-y,
$$
and, for $n \ge 2$, 
\begin{align*}
    \wh J_{k,n}^{\a,\b,-2} (x,y)&= J_{k,n}^{\a,\b,-2}(x,y), \qquad 0 \le k \le n-2,\\
  \wh J_{n-1,n}^{\a,\b,-2} (x,y)& =  G_n^{\a,\b,-2}(x,y),\\
  \wh J_{n,n}^{\a,\b,-2} (x,y)& =  F_n^{\a,\b,-2}(x,y) - \frac{(n+\a)(n-1)}{\a(2n-1)+n(2n+\b-1)} G_n^{\a,\b,-2}(x,y).
\end{align*}

\begin{prop}
For $n =0,1,\ldots$, the polynomials $\wh J_{k,n}^{\a,\b,-2}$, $0\le k \le n$, are mutually orthogonal with respect to 
$\la f, g\ra_{\a,\b,-2}^J$. Furthermore, for $n \ge 2$, $\{\wh J_{k,n}^{\a,\b,-2}, 0\le k \le n\}$ is a basis for 
$\CV_n(\varpi_{\a,\b,-2})$. 
\end{prop}

\begin{proof}
Since $\wh J_{k,n}^{\a,\b,-2}$ are linear combinations of elements in $\CV_n(\varpi_{\a,\b,-2})$ for $n \ge 2$, 
they belong to this space. 

Let  $\la f,g\ra_{\a,\b,-2}^B = \la f,g\ra_{\a,\b,-2} - \la f, g\ra_{\a,\b,-2}^\circ$. As shown in the lema, 
$J_{k,n}^{\a,\b,-2}$ are mutually orthogonal with respect to $\la \cdot,\cdot\ra_{\a,\b,-2}^\circ$. 
For $0 \le k \le n-2$, the polynomials $J_{k,n}^{\a,\b,-2}(x,y)$ contains a factor $(1-x-y)^2$, so that 
$$
     \la f, \wh J_{k,n}^{\a,\b,-2} \ra_{\a,\b,-2}^J =  \la f,J_{k,n}^{\a,\b,-2}\ra_{\a,\b,-2}^\circ,
$$
from which the orthogonality of these polynomials follows. For $k = n-1$ and $n$, the orthogonality of 
$\wh J_{k,n}^{\a,\b,-2}$ with respect to the polynomials of lower degree follow from the lemma. We only have to 
verify the mutually orthogonality of $\wh J_{n-1,n}^{\a,\b,-2}$ and $\wh J_{n,n}^{\a,\b,-2}$. Write 
$\wh J_{n,n}^{\a,\b,-2}(x,y) = F_n^{\a,\b,-2} (x,y) - c_n^{\a,\b} G_n^{\a,\b,-2}(x,y)$. A straightforward computation shows that 
\begin{align*}
 \la \wh J_{n-1,n}^{\a,\b,-2}, \wh J_{n,n}^{\a,\b,-2}\ra_{\a,\b,-2}^J & = \la F_n^{\a,\b,-2},G_n^{\a,\b,-2}\ra_{\a,\b,-2}^\circ
   - c_n^{\a,\b} \la G_n^{\a,\b,-2},G_n^{\a,\b,-2}\ra_{\a,\b,-2}^\circ \\ 
    & \quad - c_n \la G_n^{\a,\b,-2},G_n^{\a,\b,-2}\ra_{\a,\b,-2}^B \\
  & = H_{n-1,n}^{\a,\b,-2} - \frac{(2n+\a+\b)^2}{(n+\a)(n+\b)}H_{n,n}^{\a,\b,-2} \\
     & \quad - c_n^{\a,\b} \left( H_{n-1,n}^{\a,\b,-2} + \frac{(2n+\a+\b)^2}{(n+\a)^2}( H_{n,n}^{\a,\b,-2}+h_{n-1}^{\a+1,\b})\right).
\end{align*}
Solving for $c_n^{\a,\b}$ so that the above expression equals to zero completes the proof.  
\end{proof}

If we modify the definition of our definition $\la \cdot,\cdot\ra_{\a,\b,2}^J$ in \eqref{eq:ipd-ab2} to define
\begin{align*}
\la f,g\ra_{\a,\b,-2}^*  =&  \la f, g\ra_{\a,\b,-2}^\circ    + \l_0 f(0,1) g(0,1) \\
        & + \l_{0,1} \partial_1 f(0,1) \partial_1 g(0,1) + \l_{1,1} \partial_2 f(0,1) \partial_2 g(0,1),
\end{align*}
then it is easy to verify that the polynomials $\wt J_{k,n}$ defined by
$$
 \wt J_{0,0}^{\a,\b,-2}(x,y) =1,\quad \wt J_{0,1}^{\a,\b,-2}(x,y) = y-1, \quad \wt J_{1,1}^{\a,\b,-2}(x,y) = x,
$$
and, for $n \ge 2$, 
\begin{align*}
  \wt J_{k,n}^{\a,\b,-2} &= J_{k,n}^{\a,\b,-2}(x,y), \qquad 0 \le k \le n-2,\\
  \wt J_{n-1,n}^{\a,\b,-2} &=  J_{n-1,n}^{\a,\b,-2}(x,y) - J_{n-1,n}^{\a,\b,-2}(0,1)- \partial_x J_{n-1,n}^{\a,\b,-2}(0,1) x - \partial_y J_{n-11,n}^{\a,\b,-2}(0,1) (y-1),\\
  \wt J_{n,n}^{\a,\b,-2} &=  J_{n,n}^{\a,\b,-2}(x,y) - J_{n,n}^{\a,\b,-2}(0,1)- \partial_x J_{n,n}^{\a,\b,-2}(0,1) x - \partial_y J_{n,n}^{\a,\b,-2}(0,1) (y-1),
\end{align*}
form a mutually orthogonal basis with respect to $\la \cdot,\cdot\ra_{\a,\b,-2}$. The polynomials $\wt J_{k,n}$, however,
no longer belong to $\CV(\varpi_{\a,\b,-2})$. 

We now define an inner product that works with $K_{k,n}^{\a,\b,-2}$ when $\a =0$. 
For $\b > -1$ and $\l >0$, define 
\begin{align} \label{eq:ipd-0b2}
 \la f,g\ra_{0,\b,-2}^K: = & \la \partial_1 f, \partial_1 g\ra_{1,\b,-1}^K + \l \la f\vert_{z=0}, g\vert_{z=0}\ra_{-1,\b}, 
\end{align}
where $\la \cdot,\cdot\ra_{1,\b,-1}^K$ is defined in \eqref{eq:ipd-ab1K} and 
$\la \cdot,\cdot\ra$ is the inner product in one variable,   
$$
 \la f\vert_{z=0}, g\vert_{z=0}\ra_{-1,\b}:= \int_0^1 \partial_z f(1-y,y)  \partial_z g(1-y,y) y^{\b+1} dy + \l_0 f(0,1)g(0,1).
$$
It is easy to see that this defines an inner product on the space $W_2^2$. Indeed, if $\la f,f\ra_{0,\b,-2}^K =0$, then
$f(x,y) = g(y)+x h(y)$ by $\partial_x^2 f(x,y) =0$, so that $f(x,y) = g(y)$ by $\partial_x f(1-y,y) =0$, and $f(x,y)=0$ by
$\partial_z f(1-y,y)=0$ and $f(0,1) =0$. 
 
The polynomials $K_{k,n}^{0,\b,-2}$ are given by 
\begin{align*}
K_{k,n}^{0,\b,-2}(x,y) = (1-y)^k J_k^{-2,0}\left(\frac{2x}{1-y}-1\right)J_{n-k}^{2k-1,\b}(2y-1), \quad 0\le k \le n.
\end{align*}
Let $\CV_n(\varpi_{0,\b,-2})_K$ be the space of orthogonal polynomials with respect to $\la\cdot,\cdot\ra_{0,\b,-2}$.

\begin{prop} \label{prop:orthoK0b2}
The system $\{K_{k,n}^{0,\b,-2}, 0 \le k \le n, \,\, n=0,1,2,\ldots\}$ is a mutually orthogonal basis for
$\CV_n(\varpi_{0,\b,-2})_K$. 
\end{prop}

\begin{proof}
For $k \ge 2$, $K_{k,n}^{0,\b,-2}(x,y) = c (1-x-y)^2 K_{k-2,n-2}^{0,\b,2}$, which implies that 
$\partial_z K_{k,n}^{0,\b,-2} (1-y,y) =0$. The last identity also holds for $k=1$ since 
$$
 K_{1,n}^{0,\b,-2}(x,y) = -(1 - x - y) J_{n-1}^{1,\b} (2y-1)
$$
and $\partial_z g(x+y) =0$. In particular, it follows that $K_{k,n}^{0,\b,-2}(0,1) =0$ for $1 \le k \le n$.
Hence, using the fact that $\partial_x K_{k,n}^{0,\b,-2}(x,y) = K_{k-1,n-1}^{1,\b,-1}(x,y)$ and 
$\partial_x K_{k,n}^{0,\b,-2} (1-y,y) =0$, we see that, for $k \ge 1$, 
\begin{align*}
  \la f, K_{k,n}^{0,\b,-2} \ra_{0,\b,-2}^K = \la \partial_x f,  K_{k-1,n-1}^{1,\b,-1}\ra_{1,\b,-1}^K
\end{align*}
by the definition of $\la \cdot,\cdot\ra_{1,\b,-1}$. In particular, this shows that 
$\la K_{k,n}^{0,\b,-2}, K_{l,m}^{0,\b,-2}\ra_{0,\b,-2}^K =0$ when $(k,n) \ne (l,m)$ and either $k \ge 1$ or $l \ge 1$. In the
remaining case of $k= l = 0$, we use the fact that $\partial_x  K_{0,n}^{0,\b,-2}(x,y) = 0$, so that 
\begin{align*}
  \la f, K_{0,n}^{0,\b,-2} \ra_{0,\b,-2}^K = \la   f,  K_{0,n-1}^{0,\b,-2}\ra_{-1,\b}.
\end{align*}
It is easy to see that $ K_{0,n-1}^{0,\b,-2}(x,y) = J_{n}^{-1,\b}(2y-1)$ is precisely the orthogonal polynomial with respect to
$\la \cdot,\cdot \ra_{-1,\b}$. This completes the proof. 
\end{proof}

If $\a \ne 0$ and $\a > -1$, we have 
$$
   K_{1,n}^{\a,\b,-2}(x,y) = \frac{1}{\a}  ( 1- y + \a x) J_{n-1}^{\a+1,\b}(2y-1),
$$
which no longer satisfies $\partial_z K_{1,n}^{\a,\b,-2}(x,y) (1-y,y) =0$ and the orthogonality of $K_{1,n}^{\a,\b,-2}$ fails to hold
for $\la \cdot,\cdot \ra_{\a,\b,-2}^K$ derived from the obvious modification of $\la \cdot,\cdot \ra_{0,\b,-2}^K$.

In exactly the same way, we can define for  $\a > -1$ and $\l > 0$,
\begin{align} \label{eq:ipd-a02}
 \la f,g\ra_{\a,0,-2}^L: = & \la \partial_2 f, \partial_2 g\ra_{\a,1,-1}^L +\la f\vert_{z=0}, g\vert_{z=0}\ra_{-1,\a} 
\end{align}
where $\la \cdot,\cdot\ra_{\a,1,-1}^L$ is defined in \eqref{eq:ipd-ab1L} and 
$\la f\vert_{z=0}, g\vert_{z=0}\ra_{-1,\a}$ is the inner product of one variable 
$$
   \la f\vert_{z=0}, g\vert_{z=0}\ra_{-1,\a}
: = \l_2 \int_0^1  \partial_z f(x,1-x)  \partial_z g(x,1-x) x^{\a+1} dx + \l_3 f(1,0)g(1,0). 
$$
Let $\CV_n(\varpi_{0,\b,-2})_L$ be the space of orthogonal polynomials with respect to $\la\cdot,\cdot\ra_{0,\b,-2}^L$.
As an analogue of Proposition \ref{prop:orthoK0b2}, we have the following: 
 
\begin{prop}
The system $\{ L_{k,n}^{\a,0,-2}, 1 \le k \le n, \,\, n=0,1,2,\ldots\}$ is a mutually orthogonal basis of 
$\CV_n(\varpi_{0,\b,-2})_K$. 
\end{prop}

The proof relies on the identity 
\begin{align*}
  \la f, L_{k,n}^{\a,0,-2} \ra_{\a,0,-2}^L = \la \partial_x f,  L_{k-1,n-1}^{\a,1,-1}\ra_{\a,1,-1}^L.
\end{align*}

\begin{prop}
For $n =0,1,\ldots$, $\CV_n(\varpi_{0,0,-2})_K =  \CV_n(\varpi_{0,0,-2})_L =:  \CV_n(\varpi_{0,0,-2})$. 
\end{prop}

\begin{proof}
Let $g$ be a generic polynomial of degree at most $n-1$. We show that $K_{k,n}^{0,0,-2}$ is orthogonal to $g$
with respect to $\la \cdot,\cdot\ra_{0,0,-2}^L$. Using \eqref{eq:diffK}, it follows immediately from 
Proposition \ref{prop:00-1B} that $\la \partial_2 K_{k,n}^{0,0,-2}, \partial_2 g \ra_{0,1,-1} =0$.  Furthermore, from 
the proof of Proposition \ref{prop:orthoK0b2}, it follows easily that 
$\la K_{k,n}^{0,0,-2} \vert_{z=0},g\vert_{z=0} \ra_{-1,0} =0$ for $1 \le k \le n$, an this identity holds also 
for $k =0$ as can be verified by working with $K_{0,n}^{0,0,-2}(x,1-x) = J_n^{-1,0}(1-2x)$. 
\end{proof}

It is worth to point out that $\{J_{k,n}^{0,0,-2}: 0\le k \le n\}$ does not belong to $\CV_n(\varpi_{0,0,-2})_K$. Indeed,
as in the proof of the above proposition, it is easy to see that $J_{k,n}^{0,0,-2}$ satisfies 
$\la \partial_1 J_{k,n}^{0,0,-2}, \partial_1 g \ra_{0,0,-2}^K =0$. However, since 
$$
\partial_z J_{1,n}^{0,0,-2}(1-y,y) = J_{0,n-1}^{1,1,-2}(1-y,y) = J_n^{0,0}(2y-1)
$$
is not orthogonal with respect to $\la \cdot,\cdot,\ra_{-1,0}$, we see that 
$\la J_{1,n}^{0,0,-2},f\ra_{0,0,-2}^K \ne 0$, so that $J_{1,n}^{0,0,-2} \notin \CV_n(\varpi_{0,0,-2})_K$. 

We can define the projection operator and the Fourier partial sums with respect to these inner products.
As a consequence of the proof, we immediately deduce the following proposition, in which the notations should
be self-explanatory by now.  

\begin{prop}
For $f \in W_2^2$, 
$$
  \partial_1  \proj_{n, \mathrm{K}}^{0,\b,-2} f = \proj_{n-1, \mathrm{K}}^{1,\b,-1} \partial_1 f \quad \hbox{and} \quad   
  \partial_2  \proj_{n, \mathrm{L}}^{\a,0,-2} f = \proj_{n-1,\mathrm{L}}^{\a,1,-1} \partial_2 f.  
$$
\end{prop}

\begin{proof}
For $\partial_1$, this is a direct consequence of the proof of the previous propositions, since 
$\partial_x K_{k,n}^{0,\b,-2}(x,y) =0$.  For $\partial_2$, we use $L_{k,n}^{0,\b,-2}$. 
\end{proof}

\subsection{The space $\CV_n(\varpi_{\a,-2,\g})$ and $\CV_n(\varpi_{-2,\b, \g})$}
As in the Subsection \ref{sect:a-1g}, we again record only basic facts for these two cases. The relations
\eqref{eq:3bases-transform} extend to $\wh J_{k,n}^{\a,\b,-2}$. Thus, we define, for example, 
$$
  \wh K_{k,n}^{\a,-2,\g}(x,y) = \wh J_{k,n}^{\g, \a,-2}(1-x-y,x), \quad 0 \le k \le n. 
$$

For $\CV_n(\varpi_{\a,-2,\g})$ with $\a,\g> -1$, we define
\begin{align} \label{eq:ipd-a2gK}
 \la f,g\ra_{\a,-2,\g}^K: = & \int_{\triangle} (y^2 \partial_z^2 f(x,y) \partial_z^2 g(x,y) +
 2 yz \partial_z\partial_x f(x,y) \partial_z\partial_x g(x,y)  \\
   & \qquad  + z^2 \partial_x^2 f(x,y) \partial_x^2 g(x,y)) x^\a z^\g dxdy \notag \\
     & + \int_0^1 \partial_z f(0,y) \partial_z g(0,y) y^{\a+1}(1-y)^\g dy \notag \\
     & + \int_0^1 \partial_x f(0,y)  \partial_y g(0,y) y^\a (1-y)^{\g+1} dx. \notag
\end{align}
This is not a full-fledged inner production.  
The polynomials $\wh K_{k,n}^{\a,-2,\g}$, $0\le k \le n$, are mutually orthogonal with respect to 
$\la f, g\ra_{\a,-2,\g}^K$ and,  for $n \ge 2$, $\{\wh K_{k,n}^{\a,-2,\g}, 0\le k \le n\}$ is a basis for 
$\CV_n(\varpi_{\a,-2,\g})$. Moreover, we define inner products
\begin{align}\label{eq:ipd-020LJ}
\begin{split}
 \la f,g\ra_{\a,-2,0}^L: & =  \la \partial_y f, \partial_y g \ra_{\a,-1,1}^L + \l  \la f(\cdot,0),g(\cdot,0)\ra_{-1,\a}, \\
 \la f,g\ra_{0,-2,\g}^J: & =  \la \partial_z f, \partial_z g\ra_{1, -1, \g}^J + \l  \la f(\cdot,0),g(\cdot,0)\ra_{\g,-1}.
\end{split}
\end{align}  
where $\la \cdot,\cdot\ra_{\a,-1,1}^L$ and $\la \cdot,\cdot\ra_{\a,-1,1}^J$ are defined in \eqref{eq:ipd-a1g},
$\la \cdot,\cdot\ra_{-1,\a}$ and $\la \cdot,\cdot\ra_{\g,-1}$ are inner products of one variable, \eqref{eq:ipd-ell-b}
with $\ell=1$ and \eqref{eq:ipd-a-m} with $m=1$, respectively.  
%
Then $\{ L_{k,n}^{\a,-2,0}, 1 \le k \le n\}$ and $\{ J_{k,n}^{0,-2,\g}, 1 \le k \le n\}$ are mutually orthogonal 
with respect to $\la \cdot,\cdot\ra_{\a,-2,0}^L$ and $\la \cdot,\cdot\ra_{0,-2,\g}^J$, respectively. 
%

\begin{prop}
For $f \in W_2^2$, 
$$
  \partial_2  \proj_{n,  \mathrm{L}}  f = \proj_{n-1, \mathrm{L}}^{\a,-1,1} \partial_2 f \quad \hbox{and} \quad   
  \partial_3  \proj_{n, \mathrm{J}}^{0,-2,\g} f = \proj_{n-1, \mathrm{J}}^{1,-1,\g} \partial_3 f.  
$$
Moreover, $\CV_n(\varpi_{0,-2,0})_L=\CV_n(\varpi_{0,-2,0})_J$ and $\proj_{n, \mathrm{L}} f = \proj_{n, \mathrm{J}} f$. 
\end{prop} 

For $\CV_n(\varpi_{-2,\b, \g})$ with $\b,\g> -1$, we define 
\begin{align} \label{eq:ipd2bgL}
 \la f,g\ra_{-2,\b,\g}^L: = & \int_{\triangle} (y^2 \partial_z^2 f(x,y) \partial_z^2 g(x,y) +
 2 y z \partial_z\partial_y f(x,y) \partial_z\partial_y g(x,y) \\
   & \qquad  + z^2 \partial_x^2 f(x,y) \partial_x^2 g(x,y)) y^\b z^\g dxdy \notag\\
     & + \int_0^1 \partial_z f(x,0) \partial_z g(x,0) (1-x)^{\b+1}x^\g dx \notag \\
     & + \int_0^1 \partial_y f(x,0)  \partial_y g(x,0) x^\b (1-x)^{\g+1} dx, \notag
\end{align}
%
which is not a full-fledged inner product. Then the polynomials $\wh L_{k,n}^{-2,\b,\g}$, $0\le k \le n$, are mutually 
orthogonal with respect to $\la f, g\ra_{-2,\b,\g}^L$ and,  for $n \ge 2$, $\{\wh L_{k,n}^{-2,\b,\g}, 0\le k \le n\}$ is a basis for 
$\CV_n(\varpi_{-2,\b,\g})$. We further define 
\begin{align}\label{eq:ipd-200JK}
\begin{split}
 \la f,g\ra_{-2,0,\g}^J: = & \la \partial_3 f, \partial_3 g \ra_{-1,1,\g}^J + \la f(\cdot,0), g(\cdot,0) \ra_{\g,-1}, \\
 \la f,g\ra_{-2,\b,0}^K: = & \la \partial_1 f, \partial_1 g\ra_{-1, \b,1}^K + \la f(\cdot,0), g(\cdot,0) \ra_{-1,\b}.
\end{split}
\end{align}  
where  $\la \cdot, \cdot\ra_{-1,1,\g}^J$ and $\la \cdot, \cdot\ra_{-1, \b,1}^K$ are defined in \eqref{eq:ipd-1bg},
$\la \cdot,\cdot\ra_{-1,\b}$ and $\la \cdot,\cdot\ra_{\g, -1}$ are as in \eqref{eq:ipd-020LJ}. 
Then $\{ J_{k,n}^{-2,0,\g}, 1 \le k \le n\}$ and $\{K_{k,n}^{-2,\b,0}, 1 \le k \le n\}$ are mutually orthogonal 
with respect to $\la \cdot,\cdot\ra_{-2,\b,0}^J$ and $\la \cdot,\cdot\ra_{-2,\b,0}^K$, respectively. 

\begin{prop}
For $f \in W_2^2$,
$$
  \partial_3 \proj_{n,\mathrm{J}}^{-2,0,\g} f = \proj_{n-1,\mathrm{J}}^{-1,1,\g} \partial_3 f \quad \hbox{and} \quad   
  \partial_1  \proj_{n,\mathrm{K}}^{-2,\b,0} f = \proj_{n-1,\mathrm{K}}^{-1,\b,1} \partial_1 f.  
$$
Moreover, $\CV_n(\varpi_{-2,0,0})_J=\CV_n(\varpi_{-2,0,0})_K$ and 
$\proj_{n, \mathrm{J}} f = \proj_{n, \mathrm{K}} f$. 
\end{prop}

\section{Sobolev orthogonality with parameters $(-1,-1,-2)$ and its permutations}
\label{sect9}
\setcounter{equation}{0}

In this case we obtain three genuine inner products in $W_2^2$ for the three family of bases. In particular,
our inner product for the space generated by $J_{k,m}^{-1,-1,-2}$ is based on $\la \cdot,\cdot \ra_{0,0,-2}^J$, 
even though the later is not a full-fledged inner product. 

\subsection{The space $\CV_n(\varpi_{-1,-,1,-2})$}
For $f,g\in W_2^2$, we defined an inner product
\begin{align}\label{eq:ipd112J}
 \la f, g\ra_{-1,-1,-2}^J = & \la \partial_3 f, \partial_3 g\ra_{0,0,-2} + \l \la f(0,\cdot), g(0,\cdot)\ra_{-2,-1} + f(1, 0) g(1, 0),
 \end{align}
where $\la \cdot,\cdot\ra_{-2,-1}$ is the inner product of one variable in \eqref{eq:ipd-ell-b} with $\ell =2$ and $m=1$, 
\begin{align*}
 \la f, g\ra_{-2,-1}: =   \int_0^1 f''(t) g''(t)t dy + \l_0 f' (1)g'(1)  +  f(1) g(1).
\end{align*}
That this is an inner product can be seen as follows: $\la \partial_3 f, \partial_3 f \ra_{0,0,-2} = 0$ implies that 
$\partial_3 f (x,y)$
is a constant, so that $f(x,y) = a x + b g(x+y)$, hence, the part involving $\partial_y$ shows that $g(x)$ must be a constant, 
thus $f(x,y) = a (x+y) + b$, which is zero by $f(0,1) =0$ and $f(1,0) =0$. 

We need to work with $J_{k,n}^{-1,-1,-2}$, which are given by
$$
  J_{k,n}^{-1,-1,-2}(x,y)  = (x+y)^{k} J_{k}^{-1,-1}\left(\frac{y-x}{x+y}\right) J_{n-k}^{2k-1,-2}(1-2x-2y), \quad 0 \le k \le n. 
$$
In particular, we have 
\begin{align*}
\begin{split}
J_{n-1,n}^{-1,-1,-2}(x,y) &=(x+y)^{n-1} J_{n-1}^{-1,-1}\left(\frac{y-x}{x+y}\right) \left (\frac{2n-2}{2n-3} - x- y\right), \\
J_{n,n}^{\a,\b,-2}(x,y) &=(x+y)^n J_{n}^{-1,-1}\left(\frac{y-x}{x+y}\right).
\end{split}
\end{align*}
Let $\CV_n(\varpi_{-1,-1,-2})$ be the space spanned by $\{ J_{k,n}^{-1,-1,-2}: 0 \le k \le n\}$. 
In order to obtain a mutually orthogonal basis for $\CV_n(\varpi_{-1,-1,-2})$, we define the following polynomials. 
Let \begin{align*}
F_n^{-1,-1,-2}(x,y) & := J_{n-1,n}^{-1,-1,-2}(x,y) - 2 J_{n,n}^{-1,-1,-2}(x,y), \\
G_n^{-1,-1,-2}(x,y) & := J_{n-1,n}^{-1,-1,-2}(x,y)+ 2 J_{n,n}^{-1,-1,-2}(x,y).  
\end{align*}
We define a system of polynomials $\wh J_{k,n}^{-1,-1,-2}$ as follows: 
\begin{align*}
& \wh J_{0,0}^{-1,-1,-2}(x,y)  =1,\quad \wh J_{0,1}^{-1,-1,-2}(x,y) = x-\frac12, \quad \wh J_{1,1}^{-1,-1,-2}(x,y) = 1-x-y,\\
&  \wh J_{0,2}^{-1,-1,-2}(x,y) =\f12 (1 - x - y)^2,\quad \wh J_{1,2}^{-1,-1,-2}(x,y) = x(x-1), \\
&  \wh J_{2,2}^{-1,-1,-2}(x,y) = x(1-x-2y),
\end{align*}
and, for $n \ge 3$, 
\begin{align*}
    \wh J_{k,n}^{-1,-1,-2}(x,y) &= J_{k,n}^{-1,-1,-2}(x,y), \qquad 0 \le k \le n-2,\\
  \wh J_{n-1,n}^{-1,-1,-2}(x,y) & =  G_n^{-1,-1,-2}(x,y),\\
  \wh J_{n,n}^{-1,-1,-2}(x,y) & =  F_n^{-1,-1,-2}(x,y) - \frac{ n-2 }{ 2n-3} G_n^{-1,-1,-2}(x,y).
\end{align*}
 
\begin{prop}
For $n \ge 0$, the polynomials $\wh J_{k,n}^{-1,-1,-2}$, $0\le k \le n$, are mutually orthogonal with respect to 
$\la f, g\ra_{-1,-1,-2}^J$. Furthermore, for $n \ge 3$, $\{\wh J_{k,n}^{-1,-1,-2}, 0\le k \le n\}$ is a basis for $\CV_n(\varpi_{-1,-1,-2})$. 
\end{prop}

\begin{proof}
The case $n =0,1,2$ can be verified directly. For $n \ge 3$, it is clear that $\wh J_{k,n}^{-1,-1,-2} \in \CV_n(\varpi_{-1,-1,-2})$
for $0 \le k \le n$. Thus, it remains to prove that they are mutually orthogonal. These polynomials are defined so that they 
satisfy the relation
$$
  \partial_z  \wh J_{k,n}^{-1,-1,-2}(x,y) = \wh J_{k-1,n-1}^{0,0,-2} (x,y), \qquad 1 \le k \le n 
$$
and we have $ \partial_z  \wh J_{0,n}^{-1,-1,-2}(x,y) =0$. Furthermore, for $n\ge 3$, 
$\wh J_{k,n}^{-1,-1,-2}$ contains a factor $x y$ for $k =n-1$ and $k = n$, $\wh J_{k,n}^{-1,-1,-2}$ contains a factor
$xy(1-x-y)$ for $2\le k \le n-2$, $\wh J_{1,n}^{-1,-1,-2}$ contains a factor $x (1-x-y)^2$ and $\wh J_{0,n}^{-1,-1,-2}$
contains a factor $(x+y)(1-x-y)^2$. In particular, it follows that 
\begin{equation}\label{eq:J112main}
  \la f, \wh J_{k,n}^{-1,-1,-2}\ra_{-1,-1,-2}^J = \la \partial_z f,  \wh J_{k,n}^{0,0,-2}\ra_{0,0,-2}^J, \quad 1 \le k \le n, 
\end{equation}
from which the orthogonality of $\wh J_{k,n}^{-1,-1,-2}$ and $\wh J_{l,n}^{-1,-1,-2}$ for either $k \ge 1$ or $l\ge 1$ 
follows readily. For the remaining case of $k = l =0$, from $\partial_z  \wh J_{0,n}^{-1,-1,-2}(x,y) =0$, 
$\wh J_{0,n}^{-1,-1,-2}$ vanishing on $(0,1)$ and $(1,0)$,  it follows that 
$$
  \la f, \wh J_{0,n}^{-1,-1,-2}\ra_{-1,-1,-2}^J = \la f(0,\cdot), J_{0,n}^{-1,-1,-2}(0,\cdot) \ra_{-2,-1}.
$$
It is easy to see that $J_{0,n}^{-1,-1,-2}(0, y) = J_{n-1}^{-1,-2}(1-2 y)$ are orthogonal polynomials with respect to 
$\la \cdot,\cdot\ra_{-1,-2}$, the proof is completed. 
\end{proof}

\begin{prop} \label{prop:KL112bases}
For $n \ge 3$, the systems $\{K_{k,n}^{-1,-1,-2}: 0 \le k \le n\}$ and $\{L_{k,n}^{-1,-1,-2}: 0 \le k \le n\}$ are both 
bases of $\CV_n(\varpi_{-1,-1,-2})$. 
\end{prop}

\begin{proof}
We provide the proof only for $K_{k,n}^{-1,-1,-2}$. These polynomials satisfy, by \eqref{eq:Jlm}, that 
\begin{align}\label{eq:K112}
\begin{split}
K_{k,n}^{-1,-1,-2}(x,y) &=  \frac{(n-3)!}{n!} x (1-x-y)^2 K_{k-3,n-3}^{1,-1,2}(x,y), \quad 3 \le k\le n, \\
K_{2,n}^{-1,-1,-2}(x,y) &= \f12 (1-x-y)^2 J_{n-2}^{-2,-1}(2y-1), \\
K_{1,n}^{-1,-1,-2}(x,y) &= \f1n xy J_{n-2}^{0,1}(2y-1), \qquad 
K_{0,n}^{-1,-1,-2}(x,y)  = J_n^{-2,-1}(2y-1). 
\end{split}
\end{align}
Let $g$ be a generic polynomial of degree less than $n$. First we verify the equation 
$\la \partial_3 K_{k,n}^{-1,-1,-2}, \partial_3 g \ra_{0,0,-2}^J =0$. 
From the definition of the inner product in \eqref{eq:ipd-ab2}, we write 
$\la \cdot, \cdot\ra_{0,0,-2}^J = \la \cdot, \cdot \ra_T + \la \cdot,\cdot \ra_B$, where $\la \cdot,\cdot\ra_T$ is 
the integral over $\triangle$ and $\la \cdot,\cdot\ra_B$ consists of the two integrals over $[0,1]$, the boundary part. 
That  $\la \partial_3 K_{k,n}^{-1,-1,-2}, \partial_3 g \ra_T = 0$ follows from \eqref{eq:diffK} and the orthogonality 
of ordinary Jacobi polynomials on the triangle; for example, it is easy to see that $\partial_1^2 \partial_3 K_{k,n}^{-1,-1,-2} 
\in \CV_{n-3}(\varpi_{2,0,0})$ and the integral in $\la \cdot, \cdot\ra_T$ contains a $x^2$ which is the weight 
function $\varpi_{2,0,0}$. Now, by \eqref{eq:K112} and the fact that $\partial_z g(x+y) =0$, we see that 
$\partial_1 \partial_3  K_{k,n}^{-1,-1,-2}(x,1-x) =0$ and $\partial_2 \partial_3  K_{k,n}^{-1,-1,-2}(x,1-x) =0$ for
$2 \le k\le n$, so that $\la \partial_3 K_{k,n}^{-1,-1,-2}, \partial_3 g \ra_B = 0$ for $2 \le k\le n$. For $k=1$, we
derive from \eqref{eq:diffK} that 
\begin{align*} 
\partial_3 \partial_1 K_{1,n}^{-1,-1,-2}(x,y) & = K_{0,n-2}^{1,0,-1}(x,y) = J_{n-2}^{0,1}(2x-1), \\
\partial_3 \partial_2 K_{1,n}^{-1,-1,-2}(x,y) & = -2(n-2) K_{0,n-2}^{0,1,-1}(x,y) + (1-x-y) J_{n-1}^{2,1}(2x-1),
\end{align*}
so that $\partial_3 \partial_2 K_{1,n}^{-1,-1,-2}(x,1-x) = J_{n-2}^{0,1}(1-2x)$. Hence, by the orthogonality of 
$J_{n-2}^{0,1}$ and the definition of $\la \cdot,\cdot \ra_B$, it follows that 
$\la \partial_3 K_{k,n}^{-1,-1,-2}, \partial_3 g \ra_B = 0$ holds also for $k=1$. Moreover, the last equation holds
for $k =0$ as well, since $\partial_1 K_{0,n}^{-1,-1,-2}(x,y) =0$ and 
$\partial_1 \partial_3 K_{0,n}^{-1,-1,-2}(x,1-x) = K_{0,n-2}^{-1,1,0} (x,1-x) = J_{n-2}^{0,1}(1-2x)$. Putting this 
together, we have establish that $\la \partial_3 K_{k,n}^{-1,-1,-2}, \partial_3 g \ra_{0,0,-2}^J =0$ for $0 \le k \le n$.

Next we verify that $\la K_{k,n}^{-1,-1,-2}(0,\cdot), g(0,\cdot) \ra_{-2,-1} = 0$. For $k =1$ or $k \ge 3$, this 
follows immediately from \eqref{eq:K112} because of the $x$ factor in the polynomials. 
When $k=2$, the formula in \eqref{eq:K112} shows that $K_{2,n}^{-1,-1,-2}$ satisfy $\partial_yf(0,1) =0$, 
$f(0,1) =0$. Moreover, by \eqref{eq:diffK}, 
$$
  \partial_y^2 K_{2,n}^{-1,-1,-2}(x,y) = \frac{n(n-1)}{2} K_{0,n-2}^{-1,1,0}(x,y) +
   \frac{2n}{3} K_{1,n-2}^{-1,1,0}(x,y) +   K_{2,n-2}^{-1,1,0}(x,y), 
$$
which impies that $\partial_y^2 K_{2,n}^{-1,-1,-2}(0,y) = \frac{n(n-1)}2 J_{n-2}^{0,1}(2y-1)$ by 
$J_1^{0,-1}(-1) = J_2^{0,-1}(-1) \allowbreak =0$. Hence, by the orthogonality of $J_{n-2}^{0,1}$ we obtain  
$\la K_{2,n}^{-1,-1,-2}(0,\cdot), g(0,\cdot) \ra_{-2,-1} = 0$. Furthermore, the case $k=0$ can be verified directly 
from the explicit formula of $K_{0,n}^{-1,-1,-2}$ in \eqref{eq:K112} and \eqref{derivativeJ}, and from using the 
orthogonality of $J_{n-2}^{0,1}$. Finally, it is easy to verify that $K_{k,n}^{-1,-1,-2}(1, 0) =0$ for $0 \le k \le n$
for $n \ge 4$. 

Putting these together, we have proved that $\la K_{k,n}^{-1,-1,-2}, g \ra_{-1,-1,-2}^J =0$ for $0 \le k \le n$ and
$n\ge 3$. Hence, $K_{k,n}^{-1,-1,-2} \in \CV_n(\varpi_{-1,-1,-2})$. The proof is completed. 
\end{proof}

We now define another inner product that has an mutually orthonormal basis based on $K_{k,n}^{-1,-1,-2}$. 
We define 
\begin{align}\label{ipd112K}
 \la f, g\ra_{-1,-1,-2}^K: = & \la \partial_1 f, \partial_1 g\ra_{0,-1,-1}^K+  \l \la f,g\ra_{-2,-1}^K
\end{align}
where $\la \cdot,\cdot\ra_{0,-1,-1}^K$ is defined in \eqref{eq:ipd-a11} and $\la \cdot,\cdot \ra_{-2,-1}^K$ is 
\eqref{eq:ipd-ell-b} with $\ell =2$ and $m=1$ applied on $f(1-y,y)$ and $g(1-y,y)$, 
$$
  \la f,g\ra_{-2,-1}^K:= \int_0^1 \partial_y^2  f(1-y,y) \partial_y^2  g(1-y,y) y dy + \l_1 \partial_y f(0,1)  \partial_y g(0,1) 
       + \l_2 f(0,1)g(0,1). 
$$
To get a mutually orthogonal basis, we define 
\begin{align*}
 & \wh K_{0,1}^{-1,-1,-2}(x,y) = K_{0,1}^{-1,-1,-2}(x,y) -1, \\
 & \wh K_{k,n}^{-1,-1,-2}(x,y) = K_{k,n}^{-1,-1,-2}(x,y), \quad (k,n) \ne (0,1), \quad k \ge 2\\
 & \wh K_{1,n}^{-1,-1,-2} (x,y) = K_{1,n}^{-1,-1,-2}(x,y) + (n-1) K_{0,n}^{-1,-1,-2}(x,y), \quad n \ge 1.
\end{align*}

\begin{prop}\label{prop:K112}
The system $\{\wh K_{k,n}^{-1,-1,-2}: 0 \le k \le n, \, n=0,1,2,...\}$ is mutually orthogonal with respect to the inner 
product $\la \cdot,\cdot \ra_{-1,-1,-2}^K$. 
\end{prop}

\begin{proof}
As in the proof of the previous proposition, using \eqref{eq:K112}, we obtain
$$
\partial_y \wh K_{k,n}^{-1,-1,-2}(1-y,y) =0 \quad\hbox{and}\quad  \partial_y \partial_z \wh K_{k,n}^{-1,-1,-2}(1-y,y) = 0, \quad 1 \le k \le n,
$$
and $K_{k,n}^{-1,-1,-2}(0,1) =0$ for $1 \le k \le n$. Furthermore, since $\partial_x K_{0,n}^{-1,-1,-2}(x,y) =0$, we see that
$\partial_x \wh K_{k,n}^{-1,-1,-2}(x,y) = \partial_x K_{k,n}^{-1,-1,-2}(x,y)$ for $1 \le k \le n$. As a result, it follows that 
\begin{equation}\label{eq:K112main}
  \la f, \wh K_{k,n}^{-1,-1,-2}\ra_{-1,-1,-2}^K = \la \partial_x f, K_{k,n}^{0,-1,-1}\ra_{0,-1,-1}^K, \quad 1 \le k \le n,
\end{equation}
from which it follows readily that $\la  \wh K_{j,m}^{-1,-1,-2},\wh K_{k,n}^{-1,-1,-2}\ra_{-1,-1,-2}^K =0$ whenever $j > 0$ or $k >0$.
In the remaining case of $j=k=0$, we deduce from 
$$
   K_{0,n}^{-1,-1,-2}(x,y) = J_{n}^{-2,-1}(2y-1) 
$$
that $\partial_y \partial_z K_{0,n}^{-1,-1,-2}(1-y,y) = \partial_y^2 J_n^{-2,-1}(2y-1)$, so that, by 
$\partial_x K_{0,n}^{-1,-1,-2}(x,y) =0$, it follows that 
$$
  \la f,\wh K_{0,n}^{-1,-1,-2}\ra_{-1,-1,-2}^K =  \la f, \wh K_{0,n}^{-1,-1,-2}\ra_{-2,-1}^K. 
$$
It is easy to see that $ \wh K_{0,n}^{-1,-1,-2}(1-y,y) = J_{n}^{-2,-1}(2y-1) $ is precisely the orthogonal polynomials with respect to 
$\la \cdot,\cdot \ra_{-2,-1}$. This completes the proof. 
\end{proof}

Next we define an inner product that works with the polynomials $L_{k,n}^{-1,-1,-2}$,
$$
  L_{k,n}^{-1,-1,-2}(x,y) = (1-x)^k J_k^{-1,-2}\left(\frac{2y}{1-x} -1 \right)J_{n-k}^{2k-2,-1}(2x-1), \quad 0 \le k \le n.
$$
Let $\la \cdot,\cdot\ra_{-1,0,-1}^L$ be the inner product defined in Section 6.2. For $\l_4, \l_5, \l_6 > 0$, define 
\begin{align}\label{eq:ipd112L}
 \la f, g\ra_{-1,-1,-2}^L: = & \la \partial_y f, \partial_y g\ra_{-1,0,-1}^L+ \l  \la f,g\ra_{-2,-1}^L,
 \end{align}
where $\la \cdot,\cdot \ra_{-2,-1}^L$ is \eqref{eq:ipd-ell-b} with $\ell =2$ and $m=1$ applied on $f(x,1-x)$ and $g(x,1-x)$, 
\begin{align*} 
   \la f,g\ra_{-2,-1}^L:= & \int_0^1 \partial_x^2 f(x,1-x) \partial_x^2 g(x,1-x) (1-x)dx  \\
    & +  \l_1 \partial_x f(1,0)\partial_x g(1,0) + \l_2 f(1,0) g(1,0). 
\end{align*}
To get a mutually orthogonal basis, we define 
$$
\wh L_{0,1}^{-1,-1,-2} = L_{0,1}^{-1,-1,-2}-1, \quad \wh L_{k,n}^{-1,-1,-2} = L_{k,n}^{-1,-1,-2}, \quad (k,n) \ne (0,1).
$$
 
\begin{prop}\label{prop:L112}
The system $\{\wh L_{k,n}^{-1,-1,-2}: 0 \le k \le n, \,\, n=0,1,2,\ldots\}$ is mutually orthogonal with respect to the inner
product $\la \cdot,\cdot\ra_{-1,-1,-2}^L$. 
\end{prop}

\begin{proof}
Using the explicit formulas of $L_{k,n}^{-1,-1,-2}$, we can verify that $\partial_y L_{0,n}^{-1,-1,-2}(x,y) =0$ and 
$\partial_x L_{k,n}^{-1,-1,-2}(x,0) =0$ if $1 \le k \le n$. Furthermore, $L_{k,n}^{-1,-1,-2}(1,0) =0$
for $k \ge 1$. Hence, it follows readily that 
\begin{equation}\label{eq:L112main}
  \la f, L_{k,n}^{-1,-1,-2}\ra_{-1,-1,-2}^L = \la \partial_y f, L_{k,n}^{-1,0,-1} \ra_{-1,0,-1}, \qquad 1 \le k \le n,
\end{equation} 
which shows, in particular, that $ \la L_{k,n}^{-1,-1,-2}, L_{l,m}^{-1,-1,-2}\ra_{-1,-1,-2}^L =0$ whenever $k \ge 1$ 
or $l \ge 1$. In the remaining case of $k = l =0$, it is easy to see, using $\partial_y L_{0,n}^{-1,-1,-2}(x,y) =0$, that
$$
\la f, L_{0,n}^{-1,-1,-2}\ra_{-1,-1,-2}^L = \la f,  L_{0,n}^{-1,-1,-2} \ra_{-2,-1}^L, \qquad n \ge 0.
$$
Since $L_{0,n}^{-1,-1,-2}(x,y) = J_n^{-2,-1}(2x-1)$ is precisely the orthogonal polynomial with respect to 
$\la \cdot,\cdot\ra_{-2,-1}^L$, this completes the proof. 
\end{proof}

For $n \ge 3$, $\CV_n(\varpi_{-1,-1,-2})$ has three bases, $\wh J_{k,n}^{-1,-1,-2}$, $\wh K_{k,n}^{-1,-1,-2}$ and
$\wh L_{k,n}^{-1,-1,-2}$, which are mutually orthogonal with respect to their corresponding inner product. 
As a consequence, we have the following: 

\begin{prop}\label{prop:proj111}
For $n \ge 3$, 
$$
\proj_{n,J}^{-1,-1,-2} f = \proj_{n,K}^{-1,-1,-2} f =\proj_{n,L}^{-1,-1,-2} f:=\proj_{n}^{-1,-1,-2} f.
$$
\end{prop}

Let $S_{n,J}^{-1,-1,-2}$ denote the $n$-th partial sum operator of the orthogonal expansion with respect to the
$J_{k,n}^{-1,-1,-2}$ family. Define $S_{n,K}^{-1,-1,-2}$ and $S_{n,L}^{-1,-1,-2}$ likewise. Since 
\begin{equation} \label{eq:f-Sn=proj}
      f - S_n f = \sum_{m=n+1} \proj_m f, 
\end{equation}
it follows from Proposition \ref{prop:proj111} that $S_{n,J}^{-1,-1,-2} f = S_{n,K}^{-1,-1,-2} f = S_{n,L}^{-1,-1,-2} f$ for 
$n \ge 3$, which we denote by $S_n^{-1,-1,-2} f$. 

\begin{thm} \label{thm:comm112}
For $n \ge 3$, 
\begin{align}\label{eq:comm112b}
\begin{split}
  \partial_1  S_n^{-1,-1,-2} f  & = S_{n-1}^{0,-1,-1} \partial_1 f, \\
  \partial_2  S_n^{-1,-1,-2} f &  = S_{n-1}^{-1,0,-1} \partial_2 f. 
\end{split}
\end{align}
\end{thm}

\begin{proof}
By the definition, we have 
$$
   \proj_n^{-1,-1,-2} f(x,y) = \sum_{k=0}^n \frac{\la f, \wh K_{k,n}^{-1,-1,-2}\ra_{-1,-1,-2}^K}
   {\la \wh K_{k,n}^{-1,-1,-2}, \wh K_{k,n}^{-1,-1,-2}\ra_{-1,-1,-2}^K}  \wh K_{k,n}^{-1,-1,-2}(x,y).
$$
Since $\partial_1  \wh K_{k,n}^{-1,-1,-2} = K_{k-1,n-1}^{0,-1,-1}$ for $1 \le k \le n$ and $= 0$ for $k=0$,
the first identity of 
\begin{align} \label{eq:comm112}
\begin{split}
  \partial_1  \proj_n^{-1,-1,-2} f  & = \proj_{n-1}^{0,-1,-1} \partial_1 f, \\
  \partial_2  \proj_n^{-1,-1,-2} f &  = \proj_{n-1}^{-1,0,-1} \partial_2 f 
\end{split}
\end{align}
follows immediately from \eqref{eq:K112main}. Similarly, the second identity follows from \eqref{eq:L112main}. 
The  stated identities follow from \eqref{eq:f-Sn=proj} as a consequence. 
\end{proof}

Using \eqref{eq:J112main}, the same argument also shows that  
$\partial_3  \proj_n^{-1,-1,-2} f = \proj_{n-1}^{0,0,-2} \partial_3 f$. However, $\la \cdot,\cdot\ra_{0,0,-2}$ is 
not a full-fledged inner product and we cannot claim $\partial_3 S_n^{-1,-1,-2} f =S_{n-1}^{0,0,-2} \partial_3 f$. 

\subsection{The space $\CV_n(\varpi_{-2,-1,-1})$ and $\CV_n(\varpi_{-1,-2,-1})$}
Again, we can derive the orthogonality in these two cases by simultaneous permutations. We extend relations in
\eqref{eq:3bases-transform} to modified polynomials such as $\wh J_{k,n}^{-1,-1,-2}$ and record the basic facts. 

For $\CV_n(\varpi_{-2,-1,-1})$, we defined three inner products, 
\begin{align}\label{eq:ipd211L}
\begin{split}
 \la f, g\ra_{-2,-1,-1}^L : &= \la \partial_y f, \partial_y g\ra_{-2,0,0}^L + \l \la f(\cdot,0), g(\cdot,0)\ra_{-2,-1}
     + f(0, 0) g(0, 0), \\
 \la f,g\ra_{-2,-1,-1}^J: & = \la \partial_z f, \partial_z g\ra_{-1,0,-1}^J + \l_1 \la f, g\ra_{-2,-1}^J\\
 \la f,g\ra_{-2,-1,-1}^K: & = \la \partial_x f, \partial_x g\ra_{-1,-1,0}^K + \l_2 \la f,g\ra_{-2,-1}^K, 
\end{split}
\end{align}
where the first terms in the right hand are defined in previous sections, $\la \cdot,\cdot\ra_{-2,-1}$ is the inner product of one variable as before and
\begin{align*}
\la f, g\ra_{-2,-1}^J: =  & \int_0^1 \partial_x^2 f(0,y) \partial_x^2 g(0,y) (1-y) dy   + \l_1 \partial_x f(0,0) \partial_x g(0,0) + \l_2 f(0,0)g(0,0),\\
   \la f,g\ra_{-2,-1}^K :=&  \int_0^1 \partial_z^2 f(0,y) \partial_z^2 g(0,y) y dy   + \l_1  \partial_z f(0,1)\partial_z g(0,1)+\l_2 f(0,1)g(0,1). 
\end{align*}
Then each of $\{\wh L_{k,n}^{-2,-1,-1}\}$, $\{\wh J_{k,n}^{-2,-1,-1}\}$, and $\{\wh K_{k,n}^{-2,-1,-1}\}$ is
mutually orthogonal with respect to their inner product $\la \cdot, \cdot \ra_{-2,-1,-1}^L$, $\la \cdot, \cdot \ra_{-2,-1,-1}^J$, 
$\la \cdot, \cdot \ra_{-2,-1,-1}^K$, respectively. Furthermore, for $n \ge 3$, each of these three families is a basis for $\CV_n(\varpi_{-2,-1,-1})$. 

\begin{prop}
For $n \ge 3$, 
\begin{align}\label{eq:comm211}
\begin{split}
  \partial_1  S_n^{-2,-1,-1} f  & = S_{n-1}^{-1,-1,0} \partial_1 f, \\
  \partial_3 S_n^{-2,-1,-1} f &  =S_{n-1}^{-1,0,-1} \partial_3 f.  
\end{split}
\end{align}
\end{prop}

For $\CV_n(\varpi_{-1,-2,-1})$, we defined three inner products. 
\begin{align}\label{eq:ipd121K}
\begin{split}
 \la f, g\ra_{-2,-1,-1}^K := & \la \partial_x f, \partial_x g\ra_{0,-2,0}^K + \l \la f, g\ra_{-2,-1}^K  +  f(0,0) g(0, 0),\\
  \la f,g\ra_{-1,-2,-1}^L:= & \la \partial_y f, \partial_y g\ra_{-1,-1,0}^L + \l_1 \la f, g\ra_{-2,-1}^L, \\
 \la f,g\ra_{-1,-2,-1}^J:= & \la \partial_z f, \partial_z g\ra_{0,-1,-1}^J + \l_2 \la f,g\ra_{-2,-1}^J,
\end{split}
\end{align}
where where the first terms in the right hand are defined in previous sections, and the second terms are defined by
\begin{align*}
 \la f, g\ra_{-1,-2}^K: = & \int_0^1 \partial_z^2 f (y 1-y) \partial_z^2 g(y,1-y)y dy  + \l_1 \partial_y f(1,0) \partial_y g(1,0) 
   + f(1, 0) g(1, 0),\\
\la f, g\ra_{-2,-1}^L:  = &  \int_0^1  \partial_z^2 f(x,0) \partial_z^2 g(x,0) (1-x) dx  
 + \l_1 \partial_z f(1,0) \partial_z g(1,0) + \l_2 f(1,0)g(1,0), \\
 \la f,g\ra_{-2,-1}^J :=   &  \int_0^1 \partial_y^2 f(x,0) \partial_z^2 g(x,0) (1-x) dx   
   + \l_1 \partial_z f(0,0)\partial_z g(0,0)+\l_2 f(0,0)g(0,0). 
\end{align*}
Then, each of $\{\wh K_{k,n}^{-2,-1,-1}\}$, $\{\wh L_{k,n}^{-2,-1,-1}\}$, and $\{\wh J_{k,n}^{-2,-1,-1}\}$ is mutually 
orthogonal with respect to their inner product $\la \cdot, \cdot \ra_{-2,-1,-1}^K$, $\la \cdot, \cdot \ra_{-2,-1,-1}^L$, 
$\la \cdot, \cdot \ra_{-2,-1,-1}^J$, respectively. Furthermore, for $n \ge 3$, each family is a basis for
 $\CV_n(\varpi_{-1,-2,-1})$. 


\begin{prop}
For $n \ge 3$, 
\begin{align}\label{eq:comm121}
\begin{split}
  \partial_2  S_n^{-1,-2,-1} f &  = S_{n-1}^{-1,-1,0} \partial_2 f, \\
  \partial_3  S_n^{-1,-2,-1} f &  =S_{n-1}^{0,-1,-1} \partial_3 f.  
\end{split}
\end{align}
\end{prop}

\section{Sobolev orthogonality with parameters $(-2,-2,-2)$}
\label{sect10}
\setcounter{equation}{0}

In this final section on our Sobolev orthogonality, we consider the case $(-2,-2,-2)$. The inner products we define
depend on what we have done in the previous sections. The orthogonal subspace is defined as before, but further 
modification is needed for obtaining an orthonormal basis.

For $f,g\in W_2^2$, we defined an inner product
\begin{align} \label{eq:ipd222J}
 \la f, g\ra_{-2,-2,-2}^J := & \la \partial_3 f, \partial_3 g\ra_{-1,-1,-2}^J +  \la f, g\ra_{-3,-2}^J 
 \end{align}
 where 
 \begin{align*}
   \la f, g\ra_{-3,-2}^J: = & \l_1 \int_0^1 \partial_y^2 \partial_x f (0,y) \partial_y^2 \partial_x g(0,y)(1-y) dy \\ 
  & + \l_2 \partial_y \partial_x f (0,0) \partial_y \partial_x g(0,0) + \partial_x f(1, 0) \partial_x g(1, 0) + \l_3 f(0, 1) g(0, 1).
\end{align*}
That this is an inner product can be seen as follows: $\la \partial_z f, \partial_z f \ra_{-1,-1,-2}^J = 0$ implies that $\partial_3 f (x,y)=0$, 
so that $f(x,y) = g(x+y)$, hence, the part involving $\partial_y^2 \partial_x$ shows that $g(x)$ must be a quadratic polynomials, 
thus $f(x,y) = a (x+y)^2 + b(x+y) +c$, which is zero if it vanishes on the three corner points of the triangle. 

We need to work with $J_{k,n}^{-2,-2,-2}$, which are given by
$$
  J_{k,n}^{-2,-2,-2}(x,y)  = (x+y)^{k} J_{k}^{-2,-2}\left(\frac{y-x}{x+y}\right) J_{n-k}^{2k-3,-2}(1-2x-2y), \quad 0 \le k \le n. 
$$
In particular, we have 
\begin{align*}
\begin{split}
J_{0,n}^{-2,-2,-2}(x,y) & = J_n^{-3,-2}(1-2x-2y), \\
J_{1,n}^{-2,-2,-2}(x,y) & = \f12 (y-x) J_{n-1}^{-1,-2}(1-2x-2y), \\
J_{n-1,n}^{-2,-2,-2}(x,y) &= (x+y)^{n-1} J_{n-1}^{-2,-2}\left(\frac{y-x}{x+y}\right) \left (\frac{2n-4}{2n-5} - x- y\right), \\
J_{n,n}^{-2,-2,-2}(x,y) &=(x+y)^n J_{n}^{-2,-2}\left(\frac{y-x}{x+y}\right).
\end{split}
\end{align*}
The fact that the last factor in $J_{n-1,n}^{-2,-2,-2}$ does not vanish on the boundary causes potential problem in making 
this polynomial orthogonal to $J_{m,m}^{-2,-2,-2}$ for, say, $m=n-1$ or $n$. Moreover, the factor $y-x$ in 
$J_{1,n}^{-2,-2,-2}$ is also insufficient for $\la \cdot,\cdot\ra_{-3,-2}^J$. For these reasons we need to modify these 
polynomials to obtain an mutually orthogonal basis.

What we need is to define a new family of polynomials from appropriate linear combinations among $J_{k,n}^{-2,-2,-2}$, 
with $n$ fixed, so that they are mutually orthogonal with respect to the inner product \eqref{eq:ipd222J}. We first define 
following two polynomials: 
\begin{align*}
F_n^{-2,-2,-2}(x,y) & := J_{n-1,n}^{-2,-2,-2}(x,y) - 2 J_{n,n}^{-2,-2,-2}(x,y), \\
G_n^{-2,-2,-2}(x,y) & := J_{n-1,n}^{-2,-2,-2}(x,y)+ 2 J_{n,n}^{-2,-2,-2}(x,y).  
\end{align*}
We then define a system of polynomials $\wh J_{k,n}^{-2,-2,-2}$ as follows: 
\begin{align*}
& \wh J_{0,0}^{-2,-2,-2}(x,y)  =1,\quad \wh J_{0,1}^{-2,-2,-2}(x,y) = x, \quad \wh J_{1,1}^{-2,-2,-2}(x,y) = y+\frac23 x,\\
&  \wh J_{0,2}^{-2,-2,-2}(x,y) =\f12 (x+y)(2-x-y),\quad \wh J_{1,2}^{-2,-2,-2}(x,y) = x^2 -2x-y, \\
&  \wh J_{2,2}^{-2,-2,-2}(x,y) = x^2-2x-y^2, \quad \wh J_{0,3}^{-2,-2,-2}(x,y) =\frac{1}{6}(x+y)( 3-(x+y)^2),\\
&  \wh J_{1,3}^{-2,-2,-2}(x,y) = \frac16(y^3-2x^3+3(1-y)(y+x^2)), \\
& \wh J_{2,3}^{-2,-2,-2}(x,y) = \frac{1}{2}(x^2+(x+y)(1-x^2)), \\
&  \wh J_{3,3}^{-2,-2,-2}(x,y) = \frac{1}{12}(3y(1-x^2)+x(3-x^2))
\end{align*}
and, for $n \ge 4$, 
\begin{align*}
  \wh J_{k,n}^{-2,-2,-2} &= J_{k,n}^{-2,-2,-2}(x,y), \qquad k=0 \quad \hbox{and}\quad 2 \le k \le n-2,\\
  \wh J_{1,n}^{-2,-2,-2} &= J_{1,n}^{-2,-2,-2}(x,y) +\frac{n-2}{2}J_{0,n}^{-2, -2, -2}(x, y); \\
  \wh J_{n-1,n}^{-2,-2,-2} & =  G_n^{-2,-2,-2}(x,y), \\
  \wh J_{n,n}^{-2,-2,-2} & =  F_n^{-2,-2,-2}(x,y) - \frac{ n-3 }{ 2n-5} G_n^{-2,-2,-2}(x,y).
\end{align*}
Let $\CV_n(\varpi_{-2,-2,-2})$ be the space spanned by $\{ J_{k,n}^{-2,-2,-2}, 0\le k \le n\}$.
 
\begin{prop} \label{prop:222ortho}
For $n \ge 0$, the polynomials $\wh J_{k,n}^{-2,-2,-2}$, $0\le k \le n$, are mutually orthogonal with respect to 
$\la f, g\ra_{-2,-2,-2}^J$. Furthermore, for $n \ge 4$, $\{\wh J_{k,n}^{-2,-2,-2}, 0\le k \le n\}$ is a basis for $\CV_n(\varpi_{-2,-2,-2})$. 
\end{prop}

\begin{proof}
The case $n =0,1,2,3$ can be verified directly. For $n \ge 4$, it is clear that $\wh J_{k,n}^{-2,-2,-2} \in \CV_n(\varpi_{-2,-2,-2})$
for $0 \le k \le n$. Thus, it remains to prove that they are mutually orthogonal. These polynomials are defined so that they 
satisfy the relation
\begin{equation*} 
  \partial_z  \wh J_{k,n}^{-2,-2,-2}(x,y) = \wh J_{k-1,n-1}^{-1,-1,-2} (x,y), \qquad 1 \le k \le n 
\end{equation*}
and we have $ \partial_z  \wh J_{0,n}^{-1,-1,-2}(x,y) =0$. Furthermore, for $n\ge 4$, 
$\wh J_{k,n}^{-1,-1,-2}$ contains a factor $x y$ for $k =n-1$ and $k = n$, $\wh J_{k,n}^{-1,-1,-2}$ contains a factor
$x^2$ for $2\le k \le n-2$, and $\wh J_{0,n}^{-1,-1,-2}$. Moreover,  $\wh J_{1,n}^{-1,-1,-2}$ is modified so that we have
$$
  \partial_x \wh J_{1,n}^{-1,-1,-2}(x,y) = - J_{1,n-1}^{-1,-2,-1}(x,y) = x J_{n-1}^{0,-1}(1-2x-2y)
$$
by \eqref{eq:diff1J}, which shows that $\partial_x \wh J_{1,n}^{-1,-1,-2}(0,y) =0$ and 
$\partial_x \wh J_{1,n}^{-1,-1,-2}(1,0)=0$. In particular, it follows that 
\begin{align} \label{eq:J222main}
  \la f, \wh J_{k,n}^{-2,-2,-2}\ra_{-2,-2,-2} = \la \partial_z f,  \wh J_{k,n}^{-1,-1,-2}\ra_{-1,-1,-2}, \quad 1 \le k \le n, 
\end{align}
from which the orthogonality of $\wh J_{k,n}^{-2,-2,-2}$ and $\wh J_{l,n}^{-2,-2,-2}$ for either $k \ge 1$ or $l\ge 1$ follows readily. 
For the remaining case of $k = l =0$, from $\partial_z  \wh J_{0,n}^{-2,-2,-2}(x,y) =0$, we see that 
\begin{align} \label{eq:J222mainB}  
  \la f, \wh J_{0,n}^{-2,-2,-2}\ra_{-2,-2,-2} = \la f,  J_{0,n}^{-2,-2,-2}\ra_{-3,-2}.
\end{align}
Since $\partial_x \wh J_{0,n}^{-1,-1,-2}(0,0) =0$ and $\wh J_{0,n}^{-1,-1,-2}(0,1)=0$, we can verify that 
$J_{0,n}^{-2,-2,-2}(0, y) =
 J_n^{-3,-2}(1-2 y)$ are orthogonal polynomials with respect to $\la \cdot,\cdot\ra_{-1,0}$. The proof is completed. 
\end{proof}

\begin{prop} \label{prop:KL222bases}
For $n \ge 4$, the systems $\{K_{k,n}^{-2,-2,-2}: 0\le k \le n\}$ and $\{L_{k,n}^{-2,-2,-2}: 0\le k \le n\}$ are both 
bases of $\CV_n(\varpi_{-2,-,2,-2})$.
\end{prop}

\begin{proof}
We give a proof only for $K_{k,n}^{-2,-2,-2}$, the other case is similar. These polynomials are given by
$$
K_{k,n}^{-2,-2,-2}(x,y) = (1-y)^k J_k^{-2,-2}\left( \frac{2x}{1-y}-1\right) J_{n-k}^{2k-3,-2}(2y-1). 
$$
Let $g$ be a generic polynomial of degree at most $n-1$. By \eqref{eq:diffK} $\partial_z K_{k,n}^{-2,-2,-2} \in
\CV(\varpi_{-1,-1,-2})$ for $0 \le k \le n$, so that $\la \partial_3 K_{k,n}^{-2,-2,-2}, \partial_3 g \ra_{-1,-1,-2}^J =0$ 
by Proposition \ref{prop:KL112bases}. Next, we verify that $\la K_{k,n}^{-2,-2,-2}, g \ra_{-3,-2}^J =0$ as well. 

From the definition of $K_{k,n}^{-2,-,2,-2}$, we deduce, by \eqref{eq:Jlm} and \eqref{eq:Jab}, that 
\begin{align} \label{eq:K222}
 \begin{split}
    K_{k,n}^{-2,-2,-2}(x,y)& = \frac{(n-4)!}{n!} x^2(1-x-y)^2 K_{k-4,n-4}^{2,-2,2}(x,y), \quad k \ge 4, \\
    K_{3,n}^{-2,-2,-2}(x,y)& = \frac{1}{12}(1+2x-y)(1-x-y)^2 J_{n-3}^{3,-2}(2y-1),\\
    K_{2,n}^{-2,-2,-2}(x,y)& = \frac{1}{2} (1-x-y)^2 J_{n-2}^{1,-2}(2y-1),\\
    K_{1,n}^{-2,-2,-2}(x,y)& = - (1-2x-y)J_{n-1}^{-1,-2}(2y-1). 
\end{split}
\end{align}
For $k \ge 4$, the factor $x^2(1-x-y)^2$ implies that $K_{k,n}^{-2,-2,-2}$ satisfies $\partial_y^i \partial_x f(0,y) =0$,
$\partial_x f(1,0)=0$, and $f(0,1)=0$, so that  $\la K_{k,n}^{-2,-2,-2}, g \ra_{-3,-2}^J =0$ follows trivially. 
For $1 \le k\le 3$, the identities in \eqref{eq:K222} imply, using \eqref{eq:Jlm} if necessary, that $K_{k,n}^{-2,-2,-2}$ 
satisfies $\partial_x f(0,y) =0$ and $f(0,1)=0$. Furthermore, direct computation using \eqref{derivativeJ} and
\eqref{eq:Jlm} also shows that $\partial_x \partial_y f(0,0) =0$ for $1 \le k \le 3$. Hence, we only need to 
consider the integral over $\partial_y \partial_x f(0,y)$ in $\la \cdot,\cdot\ra_{-3,-2}^J$. For $k =3$, we use the identity
$$
  \partial_y \partial_x \left[(1+2x-y)(1-x-y)^2 \phi(y)\right] =6x [f(y) + (-1+x+y) f''(y)]
$$
to conclude that $\partial_y^2 \partial_x K_{3,n}^{-2,-2,-2}(0,y) =0$, which shows that 
$\la K_{3,n}^{-2,-2,-2}, g \ra_{-3,-2}^J =0$. 
For $k =2$, a direct computation by \eqref{derivativeJ} shows that 
$$
    \partial_y \partial_x K_{2,n}^{-2,-2,-2}(0,y) = 2 J_{n-3}^{2,-1}(2y-1) - (1-y) J_{n-4}^{3,0}(2y-1)
    = (n-1) J_{n-3}^{1,0}(2y-1),
$$
where the last identity follows from \eqref{eq:JacobiR2} with $\a = 2$ and $\b =0$, which is the orthogonal
polynomials with respect to $(1-y) dy$ on $[0,1]$, so that $\la K_{2,n}^{-2,-2,-2}, g \ra_{-3,-2}^J =0$. 
The last equation holds for $k=1$ as well, since a quick computation using \eqref{derivativeJ} gives 
$\partial_y \partial_x K_{1,n}^{-2,-2,-2}(0,y) = J_{n-3}^{1,0}(2y-1)$. Finally, if $k =0$, then 
$\partial_x K_{0,n}^{-2,-2,-2}(x,y) =0$ and it is easy to see that $\la K_{0,n}^{-2,-2,-2}, g \ra_{-3,-2}^J =0$.

Putting these together, we see that  $\la K_{k,n}^{-2,-2,-2}, g \ra_{-2-2,-2}^J =0$ for $0 \le k \le n$ and $n \ge 4$. 
Hence, $K_{k,n}^{-2,-2,-2}$ belong to $\CV_n(\varpi_{-2,-2,-2})$. The proof is completed. 
\end{proof}

By permutation, we can also define two more inner products that work with $K_{k,n}^{-2,-2,-2}$ and $L_{k,n}^{-2,-2,-2}$, respectively. 
We define 
\begin{align} \label{eq:ipd222K}
 \la f, g\ra_{-2,-2,-2}^K := & \la \partial_1 f, \partial_1 g\ra_{-1,-2,-1}^K +  \la f, g\ra_{-3,-2}^K,
\end{align}
where 
 \begin{align*}
   \la f, g\ra_{-3,-2}^K: = & \l_1 \int_0^1 \partial_z^2 \partial_y f(x,1-x) \partial_z^2 \partial_y g(x,1-x)(1-x) dx \\ 
  & + \l_2 \partial_z \partial_y f(0,1) \partial_z \partial_y g(0,1) + \partial_y f(0, 0) \partial_y g(0, 0) + \l_3 f(1, 0) g(1, 0),
\end{align*}
and 
\begin{align} \label{eq:ipd222L}
 \la f, g\ra_{-2,-2,-2}^L := & \la \partial_2 f, \partial_2 g\ra_{-2,-1,-1}^L +  \la f, g\ra_{-3,-2}^L, 
 \end{align}
where 
\begin{align*}
   \la f, g\ra_{-3,-2}^L: = & \l_1 \int_0^1 \partial_x^2 \partial_z f (x,0) \partial_x^2 \partial_x g(x,0)x dx \\ 
  & + \l_2 \partial_x \partial_z f (1,0) \partial_x \partial_z g(1,0) + \partial_z f(0,1) \partial_z g(0,1) + \l_3 f(0,0) g(0,0).
\end{align*}
For $n \ge 0$, we define 
\begin{align*}
  \wh K_{k,n}^{-2,-2,-2} (x,y) & = \wh J_{k,n}^{-2,-2,-2} (1-x-y,x), \quad 0 \le k \le n, \\
  \wh L_{k,n}^{-2,-2,-2} (x,y) & = \wh J_{k,n}^{-2,-2,-2} (y, 1-x-y), \quad 0 \le k \le n.
\end{align*}

\begin{prop}
The system $\{K_{k,n}^{-2,-2,-2}: 0 \le k \le n, \, n=0,1,2,...\}$ is mutually orthogonal with respect to the inner 
product $\la \cdot,\cdot \ra_{-2,-2,-2}^K$. The system $\{L_{k,n}^{-2,-2,-2}: 0 \le k \le n, \, n=0,1,2,...\}$ is mutually
orthogonal with respect to the inner product $\la \cdot,\cdot \ra_{-2,-2,-2}^L$.   
\end{prop}

For $n \ge 4$, the space $\CV_n(\varpi_{-2,-2,-2})$ has three bases, $\wh J_{k,n}^{-2,-2,-2}$, 
$\wh K_{k,n}^{-2,-2,-2}$ and $\wh L_{k,n}^{-2,-2,-2}$, which are mutually orthogonal with respect to their 
corresponding inner product. Let 
\begin{equation}\label{eq:proj-222}
  \proj_{n,J}^{-2,-2,-2} f = \sum_{k=0}^n \wh f_{k,n}^{-2,-2,-2} \wh J_{k,n}^{-2,-2,-2}, 
\end{equation}
where $\wh f_{k,n}^{-2,-2,-2}$ is defined in terms of $\wh J_{k,n}^{-2,-2,-2}$, and define $\proj_{n,K}^{-2,-2,-2}$ and 
$\proj_{n,L}^{-2,-2,-2}$ similarly.  

\begin{prop} \label{prop:proj222}
For $n \ge 4$, 
$$
\proj_{n,J}^{-2,-2,-2} f = \proj_{n,K}^{-2,-2,-2} f =\proj_{n,L}^{-2,-2,-2} f =: \proj_n^{-2,-2,-2}f.
$$
\end{prop}

Let $S_{n,J}^{-2,-2,-2}$ denote the $n$-th partial sum operator of the orthogonal expansion with respect to the
$J_{k,n}^{-2,-2,-2}$ family. Define $S_{n,K}^{-2,-2,-2}$ and $S_{n,L}^{-2,-2,-2}$ likewise. As in the previous section, 
it follows from Proposition \ref{prop:proj222} that $S_{n,J}^{-2,-2,-2} f = S_{n,K}^{-2,-2,-2} f = S_{n,L}^{-2,-2,-2} f$ 
for $n \ge 3$, which also implies the following theorem. 

\begin{thm} \label{thm:proj222}
Let $S_n^{-2,-2,-2}$ denote either one of the three partial sum operators. For $n \ge 3$, 
\begin{align*}
  \partial_1  S_n^{-2,-2,-2} f  & = S_{n-1}^{-1,-2,-1} \partial_1 f, \\
  \partial_2  S_n^{-2,-2,-2} f  &  = S_{n-1}^{-2,-1,-1} \partial_2 f, \\
  \partial_3  S_n^{-2,-2,-2} f  &  =S_{n-1}^{-1,-1,-2} \partial_3 f.  
\end{align*}
\end{thm}

\begin{proof}
The proof follows along the same line of proof for Theorem \ref{thm:comm112}. For $\partial_3$, we 
use $\partial_3 \wh J_{k,n}^{-2,-2,-2} = \wh J_{k-1,n-1}^{-1,-1,-2}$ for $1 \le k \le n$ and $=0$ for
$k=0$ and \eqref{eq:J222main} to conclude that $\partial_3 \proj_n^{-2,-2,-2} f = \proj_{n-1}^{-1,-1,-2} \partial_3 f$,
which proves the third identity. For $\partial_1$ and $\partial_2$ we use the expansion of $\proj_n f$ in 
$K_{k,n}^{-2,-2,-2}$ and $L_{k,n}^{-2,-2,-2}$, respectively. 
\end{proof}

\section{Approximation and Orthogonal Expansions for $\varpi_{\a,\b,\g}$}
\label{sect11}
\setcounter{equation}{0}

Starting from this section, we study approximation in the Sobolev spaces. Throughout the rest of this paper, we denote 
$$ 
\|f\|_{\a,\b,\g} : =  \| f\|_{L^2(\varpi_{\a,\b,\g})} \quad \hbox{and} \quad \|f\| = \|f\|_{L^2(\varpi_{0,0,0})}. 
$$ 
We also adopt the convention that $c$ is a positive constant, possibly depending on some fixed parameters, and its value may 
vary at each occurrence, and $c_{\a}$, $c_{\a,\b}$ ... are such constants that depend on the parameters in their subscript. 
Also recall that the relation $A\sim B$ means that $0 < c_1 \le A / B\le c_2$. 

Let $\wh f_{k,n}^{\a,\b,\g}$ and $\proj_n^{\a,\b,\g} f $ be defined as in \eqref{eq:Fourier}. As a consequence of the commuting
relations in \eqref{eq:dproj}, we can state the action of the derivatives in terms of the Fourier coefficients. 
Let $a_{k,n}^{\a,\b}$ be defined as in \eqref{eq:akn-bkn} and \eqref{eq:akn-bkn2}. 

\begin{prop} \label{cor:whf-whDf}
Let $ \wh f_{k,n}^{\a,\b,\g}$ be defined as in \eqref{eq:proj-abc}. Then
\begin{align} \label{eq:whf-whD}
\begin{split}
 \wh {\partial_1 f}_{k,n-1}^{\a+1,\b,\g+1} &\ = - a_{k+1,n}^{\a,\b} \wh f_{k+1,n}^{\a,\b,\g}  -   \wh f_{k,n}^{\a,\b,\g} \\
  \wh {\partial_2 f}_{k,n-1}^{\a,\b+1,\g+1} &\ =  a_{k+1,n}^{\b,\a} \wh f_{k+1,n}^{\a,\b,\g} 
             -  \wh f_{k,n}^{\a,\b,\g}, \\
  \wh {\partial_3 f}_{k,n-1}^{\a+1,\b+1,\g} &\ =  \wh f_{k+1,n}^{\a,\b,\g}.
\end{split}
\end{align}
\end{prop} 
 
\begin{proof}
We use the expression of $\proj_n^{\a,\b,\g} f$ in terms of $J_{k,n}^{\a,\b,\g}$. By \eqref{eq:diff1J}, 
\begin{align*}
  \partial_1 \proj_n^{\a,\b,\g} f & = \sum_{k=0}^n \wh f_{k,n}^{\a,\b,\g} \left(- a_{k,n}^{\a,\b} J_{k-1,n-1}^{\a+1,\b,\g+1}
     -   J_{k,n-1}^{\a+1,\b,\g+1} \right) \\
      & = \sum_{k=0}^{n-1}  \left(- a_{k+1,n}^{\a,\b} \wh f_{k+1,n}^{\a,\b,\g} 
             -   \wh f_{k,n}^{\a,\b,\g} \right) J_{k,n-1}^{\a+1,\b,\g+1}.  
\end{align*}
By the first identity in \eqref{eq:dproj}, this gives the first identity in \eqref{eq:whf-whD}. Considering $\partial_2 \proj_n^{\a,\b,\g} f $
and using the second identity in \eqref{eq:dproj}, we can deduce likewise the second identity in \eqref{eq:whf-whD}. Finally, the 
third identity follows from taking $\partial_3$ derivative of $\proj_n^{\a,\b,\g}$ and using the third identity in \eqref{eq:dproj}.
\end{proof}

\begin{defn}
Let $\CW_2^{r}(\varpi_{\a,\b,\g})$ be the Sobolev space that consists of those functions $f \in L^2(\varpi_{\a,\b,\g})$ such that
$$
\partial^{\mathbf{m}} f \in L^2(\omega_{\a+ \wh m_1, \b+ \wh m_2,\g+ \wh m_3}), \qquad |\mathbf{m}| \le r, 
$$  
where $\wh m_1 = m_1+m_3$, $\wh m_2 = m_2+m_3$ and $\wh m_3 = m_1+m_2$ for $\mathbf{m} \in \NN_0^3$.
\end{defn}

We now state a theorem on approximation for functions in $W_2^1(\varpi_{\a,\b,\g})$. Recall that for $\a,\b,\g > -1$ and $f \in L^2(\varpi_{\a,\b,\g})$, we denote by $E_n(f)_{\a,\b,\g}$, in \eqref{eq:best-approx},
the error of the best approximation of $f$ by polynomials of degree at most $n$. 

\begin{thm} \label{thm:est-aag}
Let $a, \b,\g > -1$. If $f \in \CW_2^1(\varpi_{\a,\a,\g})$, then 
\begin{align} \label{eq:est-aag}
  E_n(f)_{\a,\b,\g} \le   \frac{c_{\a,\b,\g}}{n} & \left[E_{n-1} (\partial_1 f)_{\a+1,\b,\g+1} \right.  \\
          & \left. +E_{n-1} (\partial_2 f)_{\a,\b+1,\g+1} +E_{n-1} (\partial_3 f)_{\a+1,\b+1,\g} \right ]. \notag
\end{align}
\end{thm}

\begin{proof}
By the definition of the partial sum operator, it is easy to see that 
\begin{align} \label{eq:En}
\left[ E_n(f)_{\a,\b,\g} \right]^2 = \left \|f-S_n^{\a,\b,\g}f \right \|_{\a,\b,\g}^2 = \sum_{m=n+1}^\infty \sum_{k=0}^m 
      \left |\wh f_{k,m}^{\a,\b,\g} \right |^2 h_{k,m}^{\a,\b,\g}. 
\end{align}
Furthermore, applying the first two identities of \eqref{eq:whf-whD}, we obtain likewise 
\begin{align}\label{eq:EnD1+D2}
\begin{split}
 \left[ E_{n-1}(\partial_1 f)_{\a+1,\b,\g+1} \right]^2 &  = \sum_{m=n}^\infty \sum_{k=0}^m
    \left |a_{k+1,m+1}^{\a,\b} \wh f_{k+1,m+1}^{\a,\b,\g} + \wh f_{k,m+1}^{\a,\b,\g} \right |^2 h_{k,m}^{\a+1,\b,\g+1}, \\
 \left[ E_{n-1}(\partial_2 f)_{\a,\b+1,\g+1} \right]^2 &  = \sum_{m=n}^\infty \sum_{k=0}^m
    \left |a_{k+1,m+1}^{\b,\a} \wh f_{k+1,m+1}^{\a,\b,\g} - \wh f_{k,m+1}^{\a,\b,\g} \right |^2 h_{k,m}^{\a,\b+1,\g+1}.  
\end{split}
\end{align}
Moreover,  applying the third identity of \eqref{eq:whf-whD}, we also obtain
\begin{align} \label{eq:EnD3}
 \left[ E_{n-1}(\partial_3 f)_{\a+1,\b+1,\g} \right]^2 &  = \sum_{m=n}^\infty \sum_{k=0}^m
    \left |\wh f_{k+1,m+1}^{\a+1,\b+1,\g} \right |^2 h_{k,m}^{\a+1,\b+1,\g}.
\end{align} 

We now use \eqref{eq:EnD1+D2} and \eqref{eq:EnD3} to bound \eqref{eq:En}. If $0 \le k \le m/2$, we have
$$
 \frac{h_{k,m}^{\a,\b,\g}}{h_{k,m-1}^{\a+1,\b,\g+1}} = \frac{(2k + \a + \b+1)(2k+\a+\b+2)}{(k+\a+1)(k + \a +\b + 1) (m-k) (m+k+ \a +\b+ \g+2)} \sim \frac{1}{m^2}
$$
by \eqref{eq:hkn}; similarly, ${h_{k,m}^{\a,\b,\g}}/{h_{k,m-1}^{\a,\b+1,\g+1}} \sim m^{-2}$. Furthermore, by \eqref{eq:akn-bkn},
we see that 
$$
 \wh f_{k,m}^{\a,\b,\g} = \frac{k+\a+1}{2k+\a+\b+2}(a_{k+1,m}^{\a,\b}  \wh f_{k+1,m}^{\a,\b,\g} + \wh f_{k,m}^{\a,\b,\g})
     -  \frac{k+\b+1}{2k+\a+\b+2}(a_{k+1,m}^{\b,\a}  \wh f_{k+1,m}^{\a,\b,\g} - \wh f_{k,m}^{\a,\b,\g})
$$
so that 
$$
 \left |\wh f_{k,m}^{\a,\b,\g} \right |^2 \le 2 \left | a_{k+1,m}^{\a,\b}  \wh f_{k+1,m}^{\a,\b,\g} + \wh f_{k,m}^{\a,\b,\g}\right|^2
  + \left| a_{k+1,m}^{\b,\a}  \wh f_{k+1,m}^{\a,\b,\g} - \wh f_{k,m}^{\a,\b,\g} \right|^2.
$$
Consequently, it follows from \eqref{eq:EnD1+D2} that 
\begin{align*}
\sum_{m=n+1}^\infty \sum_{k=0}^{\lfloor \frac{m}{2} \rfloor } \left |\wh f_{k,m}^{\a,\b,\g} \right |^2 h_{k,m}^{\a,\b,\g}
& \le \frac{c}{n^2} \sum_{m=n+1}^\infty \sum_{k=0}^{m} \left |\wh f_{k,m}^{\a,\b,\g} \right |^2 \min\{h_{k,m-1}^{\a+1,\b,\g+1},
  h_{k,m-1}^{\a,\b+1,\g+1} \} \\
& \le \frac{c}{n^2} \left( \left[ E_{n-1}(\partial_1 f)_{\a+1,\b,\g+1} \right]^2 +  \left[ E_{n-1}(\partial_2 f)_{\a,\b+1,\g+1} \right]^2 \right).
\end{align*}
If $k \ge m/2$, then $k \sim m$ so that, by \eqref{eq:hkn},
$$
 \frac{h_{k,m}^{\a,\b,\g}}{h_{k-1,m-1}^{\a+1,\b+1,\g}} = \frac{1}{k(k+\a+ \b+1)} \sim \frac{1}{m^2},
$$
which implies that the corresponding portion of the sum in \eqref{eq:En} is bounded by $c n^{-2}$ multiple of 
$\left[ E_{n-1}(\partial_3 f)_{\a+1,\b+1,\g} \right]^2$ by \eqref{eq:EnD3}. This completes the proof.
\end{proof}

Iterating the estimate \eqref{eq:est-aag}, we obtain the following estimate for $f \in W_2^r$.  

\begin{cor}
Let $\a, \b,\g> -1$. If $f \in \CW_2^r(\varpi_{\a,\b,\g})$, then 
\begin{equation}\label{eq:cor11}
  E_n(f)_{\a,\b,\g} \le \frac{c_{\a,\b,\g}}{n^r} 
    \sum_{|\mathbf{m}| = r} E_{n-r} (\partial^{\mathbf{m}} f)_{\a+\wh m_1,\b+\wh m_2,\g+\wh m_3}.
\end{equation}
\end{cor}

It is easy to see that $\CW_2^{2r}(\varpi_{\a,\b,\g})$ coincides with $\CW_2^{2r}$ in the Definition \ref{def:1.1}. 
Recall the error $\mathcal{E}_{n-2r}^{(2r)} (f)$ of approximation in $\CW_2^{2r}$, defined in \eqref{eq:calEn}. 
Iterating the estimate in \eqref{eq:cor11} gives the following corollary: 

\begin{cor} \label{cor:calE-iterate}
Let $s$ and $r$ be two positive integers, $s \le r$. For $f \in \CW_2^{2r}$, 
$$
   \mathcal{E}_{n-2s}^{(2s)} (f) \le c\, n^{2s-2r}  \mathcal{E}_{n-2r}^{(2r)}.
$$
\end{cor}
 
\section{Approximation by $S_n^{\a,\b,-1} f$ and its permutations}
\label{sect12}
\setcounter{equation}{0}

With respect to the inner product $\la \cdot, \cdot\ra_{\a,\b,-1}^J$ discussed in Section 5, we consider approximation 
behavior of the corresponding partial sum operators. Let $S_n^{\a,\b,-1}$ be the $n$-th partial sum operator of $f$ 
defined by 
$$
S_n^{\a,\b,-1} f := \sum_{k=0}^n \proj_k^{\a,\b,-1} f = \sum_{m=0}^n \sum_{k=0}^m \wh f_{k,m}^{\a,\b,-1} J_{k,m}^{\a,\b,-1}, 
$$
where $\wh f_{k,m}^{\a,\b,-1}$ is defined as in \eqref{eq:proj-ab1}.  

\begin{thm}\label{thm:norm-normD}
For $f \in \CW_2^1(\varpi_{\a,\b,0})$ and $n=1,2,3,...$, 
\begin{align}\label{eq:norm-normD}
 \left \|f-S_n^{\a,\b,-1}f \right\|_{\a,\b,0} \le  \frac{ c_{\a,\b}}{n}
  \left [ E_{n-1}( \partial_1 f)_{\a+1,\b,0} +   E_{n-1}( \partial_2 f)_{\a,\b+1,0} \right].  
\end{align}
\end{thm}

\begin{proof}
By the identity \eqref{eq:recurJc} with $\g = -1$, we have 
$$
  J_{k,n}^{\a,\b,-1} (x,y) = J_{k,n}^{\a,\b,0} (x,y)+ c_{k,n}^{\a,\b} J_{k,n-1}^{\a,\b,0} (x,y),
$$
where 
$$
 c_{k,n}^{\a,\b}:= \frac{n+k+\a+\b+1}{(2n+\a+\b)(2n+\a+\b+1)}.
 $$
Hence, we see that  
\begin{align*}
\proj_n^{\a,\b,-1}f(x,y)=   \sum_{k=0}^n \wh f_{k,n}^{\a,\b,-1}  \left(J_{k,n}^{\a,\b,0} (x,y)+ c_{k,n}^{\a,\b}J_{k,n-1}^{\a,\b,0} (x,y)\right).
\end{align*}
Consequently, it follows that 
\begin{align*}
& f - S_n^{\a,\b,-1}f  = \sum_{m=n+1}^\infty \proj_m^{\a,\b,-1}f \\
  &    = \sum_{k=0}^n \wh f_{k,n+1}^{\a,\b,-1}c_{k,n+1}^{\a,\b} J_{k,n}^{\a,\b,0}  
    + \sum_{m=n+1}^\infty  \sum_{k=0}^m
       \left( \wh f_{k,m}^{\a,\b,-1} + \wh f_{k,m+1}^{\a,\b,-1} c_{k,m+1}^{\a,\b}\right )J_{k,m}^{\a,\b,0}.
\end{align*}
Recall that $h_{k,n}^{\a,\b,0}$ denotes the $L^2$ norm of $J_{k,n}^{\a,\b,0}$. We obtain immediately that 
\begin{align*}
 \| f - S_n^{\a,\b,-1}f \|_{\a,\b,0}^2 = &
     \sum_{k=0}^n |\wh f_{k,n+1}^{\a,\b,-1} c_{k,n}^{\a,\b}|^2 h_{k,n}^{\a,\b,0} \\
     & + \sum_{m=n+1}^\infty  \sum_{k=0}^m
       \left | \wh f_{k,m}^{\a,\b,-1} + \wh f_{k,m+1}^{\a,\b,-1}c_{k,m+1}^{\a,\b} \right |^2 h_{k,m}^{\a,\b,0}.
\end{align*}
Using the fact that $(A+B)^2 \le 2 A^2 + 2 B^2$, it follows that 
\begin{align} \label{eq:norm-est}
 \| f - S_n^{\a,\b,-1}f \|_{\a,\b,0}^2  
         \le  &  \, 2  \sum_{m=n}^\infty  \sum_{k=0}^m \Big |\wh f_{k,m+1}^{\a,\b,-1} \Big |^2
             \left ((c_{k,m+1}^{\a,\b} )^2  h_{k,m}^{\a,\b,0} +h_{k,m+1}^{\a,\b,0} \right ) \\
                & + 2 \sum_{m=n+1}^\infty \Big | \wh f_{m,m}^{\a,\b,-1} \Big |^2 h_{m,m}^{\a,\b,0}.\notag 
\end{align}
Now, for $0 \le k \le m-1$,  we deduce from Proposition \ref{cor:whf-whDf} that
$$
  \wh f_{k,m+1}^{\a,\b,-1} = -  \left(\frac{k+\a+1}{2k+\a+\b+2} \wh {\partial_1 f}_{k,m}^{\a+1,\b,0}+
      \frac{k+\b+1}{2k+\a+\b+2} \wh {\partial_2 f}_{k,m}^{\a,\b+1,0} \right). 
$$
Using the expression of $h_{k,n}^{\a,\b,\g}$ in \eqref{eq:hkn}, it is easy to see that 
$h_{k,m+1}^{\a,\b,0}/{h_{k,m}^{\a+1,\b,0}} \le c m^{-2}$ and $h_{k,m+1}^{\a,\b,0}/{h_{k,m}^{\a,\b+1,0}} \le c m^{-2}$,
so that 
$$
 \left |\wh f_{k,m}^{\a,\b,-1}\right |^2 h_{k,m}^{\a,\b,0} \le \frac{c}{m^2}\left(  
    \left |\wh f_{k,m}^{\a+1,\b,0}\right |^2 h_{k,m-1}^{\a+1,\b,0} +\left |\wh f_{k,m}^{\a,\b+1,0}\right |^2 h_{k,m-1}^{\a,\b+1,0} \right).
$$
Similarly, $(c_{k,m+1}^{\a,\b})^2 h_{k,m}^{\a,\b,0}/h_{k,m}^{\a+1,\b,0} \le c m^{-2}$ and 
$(c_{k,m+1}^{\a,\b})^2 h_{k,m}^{\a,\b,0}/h_{k,m}^{\a,\b+1,0}\le c m^{-2}$, so that  
$$
 |\wh f_{k,m+1}^{\a,\b,-1}|^2 (c_{k,m+1}^{\a,\b})^2 h_{k,m}^{\a,\b,0} \le c m^{-2} 
   \left( |\wh f_{k,m}^{\a+1,\b,0}|^2 h_{k,m}^{\a+1,\b,0} + |\wh f_{k,m}^{\a,\b+1,0}|^2 h_{k,m-1}^{\a,\b+1,0} \right)
$$
Finally, we can also deduce from \eqref{eq:whf-whD} that 
$$
 \wh f_{n,n}^{\a,\b,-\g} = - \left( 
    \wh {\partial_1 f}_{n-1,n-1}^{\a+1,\b,\g+1} - \wh {\partial_2 f}_{n-1,n-1}^{\a,\b+1,\g+1} \right)
$$
and from \eqref{eq:hkn} that ${h_{m,m}^{\a,\b,0}}/ {h_{m-1,m-1}^{\a+1,\b,0}} \le c m^{-2}$. Together, they show that
\begin{align*}
  \Big| \wh f_{m+1,m+1}^{\a,\b,-1} & \Big|^2 h_{m+1,m+1}^{\a,\b,0} \le \frac{c}{m^2} \left ( \Big | \wh {\partial_1 f}_{m,m}^{\a+1,\b,0}\Big |^2 
     h_{m,m}^{\a,\b+1,0} +   \Big | \wh {\partial_2 f}_{m,m}^{\a,\b+1,0} \Big|^2 h_{m,m}^{\a,\b+1,0} \right).
\end{align*}

Substituting these inequalities into \eqref{eq:norm-est} proves the estimate 
\begin{align*}
 \|f-S_n^{\a,\b,-1}f \|_{\a,\b,0}^2 \le & \frac{2 c_{\a,\b}}{(2n+\a+\b+3)^2 } \sum_{m=n}^\infty \sum_{k=0}^m
       \Big | \wh{\partial_1 f}_{k,m}^{\a+1,\b, 0} \Big|^2 h_{k,m}^{\a+1,\b,0}
  \\
 & + \frac{2 c_{\a,\b}}{(2n+\a+\b+3)^2 }  \sum_{m=n}^\infty \sum_{k=0}^m
       \Big | \wh{\partial_2 f}_{k,m}^{\a,\b+1, 0} \Big|^2 h_{k,m}^{\a,\b+1,0},
\end{align*}
which is exactly the estimate \eqref{eq:norm-normD} by the Parseval identity for the orthogonal expansions in
$L^2(\varpi_{\a+1,\b,0})$ and in $L^2(\varpi_{\a,\b+1,0})$, respectively.  
\end{proof}

A more careful analysis shows that $c_{\a,\b}/n$ in the statement of the theorem can be replaced by a more precise
$2 / (n+\a+\b+3)$ when $n \ge 4$.  

We can also deduce an analogue of  Theorem \ref{thm:norm-normD} for the two permutations of $(\a,\b,-1)$. These 
estimates will be needed later in the paper, so we record them below. 

\begin{thm}\label{thm:norm-normDa0g}
For $f \in \CW_2^1(\varpi_{\a,0,\g})$ and $n=1,2,3,...$, 
\begin{align}\label{eq:norm-normDa0g}
 \left \|f-S_n^{\a,-1,\g} f \right\|_{\a,0,\g}  \le  \frac{c_{\a,\g}}{n} 
     \left [ E_{n-1}(\partial_2 f)_{\a,0,\g+1}+ E_{n-1}(\partial_3 f)_{\a+1,0,\g}\right].  
\end{align}
\end{thm}
 
\begin{thm}\label{thm:norm-normD0bg}
For  $f \in \CW_2^1(\varpi_{0,\b,\g})$ and $n=1,2,3,...$, 
\begin{align}\label{eq:norm-normD0bg}
 \left \|f-S_n^{-1,\b,\g} f \right\|_{0,\b, \g}  \le \frac{c_{\b,\g}}{n^2}
     \left[ E_{n-1}(\partial_3 f)_{0,\b+1,\g} + E_{n-1}(\partial_1 f)_{0,\b,\g+1} \right].
\end{align}
\end{thm}

\section{Approximation by $S_n^{-1,-1,\g} f$ and its permutations}
\label{sect13}
\setcounter{equation}{0}

We define by $S_n^{-1,-1,\g}$ the $n$-th partial sum operator defined by 
$$
  S_n^{-1,-1,\g} f := \sum_{k=0}^n \proj_k^{-1,-1,\g} f = \sum_{m=0}^n \sum_{k=0}^m \wh f_{k,m}^{-1,-1,\g} J_{k,m}^{-1,-1,\g}, 
$$
where $\wh f_{k,m}^{-1,-1,\g}$ is defined as in \eqref{eq:Fourier11g}. From the commuting of the projection operators and
derivatives given in Theorem \ref{thm:proj-1-1g}, we can deduce the fowling relations: 

\begin{prop}
Let $\wh f_{k,n}^{\a,\b,\g}$  be defined with respect to the basis $J_{k,n}^{\a,\b,\g}$. Then
\begin{align}\label{eq:whf-whD2} 
\begin{split}
   \wh {\partial_1 \partial_2 f}_{0,m}^{0,0,\g+2}  & \ = \wh f_{0,m+2}^{-1,-1,\g} + \frac{m+2}{2} \wh  f_{1,m+2}^{-1,-1,\g}  
           -A_{0,m+2} \wh f_{2,m+2}^{-1,-1,\g}, \\
   \wh {\partial_1 \partial_2 f}_{k,m}^{0,0,\g+2}  & \ = \wh f_{k,m+2}^{-1,-1,\g} - A_{k,m+2} \wh f_{k+2,m+2}^{-1,-1,\g},  \quad k \ge 1,
\end{split}
\end{align}
where 
\begin{align}\label{eq:Akm} 
  A_{k,m} := \frac{(m+k)(m+k+1)}{4(2k+1)(2k+3)}.
\end{align}
\end{prop}

\begin{proof}
It is easy to see that \eqref{eq:whf-whD} works even when $\a$ and $\b$ are negative integers. Hence, the identity 
\eqref{eq:whf-whD2} follows from combining the first two identities of \eqref{eq:whf-whD} and using 
$a_{k+1,m-1}^{\a,\a+1} - a_{k+1,m}^{\a,\a}= 0$ for $\a = -1$ and $k \ge 1$, whereas this quantity is equal to 
$m/2$ when $a = -1$ and $k=0$ by \eqref{eq:akn-bkn2}. 
\end{proof}

\begin{prop}
For $f \in \CW_2^1(\varpi_{0,0,\g})$, 
\begin{equation} \label{eq:normD3}
  \left[ E_{n-1}(\partial_3 f)_{0,0,\g}\right]^2  = \sum_{m=n}^\infty \sum_{k=0}^m \left |\wh f_{k+1,m+1}^{-1,-1,\g}\right|^2 h_{k,m}^{0,0,\g}
\end{equation}
and
\begin{align}\label{eq:normD12}
 \left[ E_{n-2}(\partial_1\partial_2 f)_{0,0,\g+2}\right]^2  = &\sum_{m=n-1}^\infty  \left |\wh f_{0,m+2}^{-1,-1,\g} 
    + \frac{m+2}{2} \wh  f_{1,m+2}^{-1,-1,\g}
            - A_{0,m+2} \wh  f_{2,m+2}^{-1,-1,\g}\right|^2 h_{0,m}^{0,0,\g+2}\notag  \\
 & + \sum_{m=n-1}^\infty \sum_{k=1}^m \left |\wh f_{k,m+2}^{-1,-1,\g} -A_{k,m+2} \wh f_{k+2,m+2}^{-1,-1,\g} \right|^2 h_{k,m}^{0,0,\g+2}.
\end{align}
\end{prop}

\begin{proof}
Since the proof of Proposition \ref{cor:whf-whDf} is a formal manipulation of formulas, it holds also when $\a = \b = -1$.
Since $\left[E_{n-1}(\partial_3 f)_{0,0,\g}\right]^2 =\|\partial_3 f-S_{n-1}^{0,0,\g}(\partial_3 f)\|_{0,0,\g}^2$, the 
identity \eqref{eq:normD3} follows from the third identity of \eqref{eq:whf-whD}. Likewise, \eqref{eq:normD12} follow from 
\eqref{eq:whf-whD2}. 
\end{proof}

\begin{thm}\label{thm:norm-11g}
For $f \in W_2^1(\varpi_{0,0,\g})\cap W_2^2(\varpi_{0,0,\g+2})$ and $n =2,3,\ldots$, 
\begin{align}\label{eq:norm-11g}
 \left \|f-S_n^{-1,-1,\g} f \right\|_{0,0,\g} \le  \frac{c_{1,\g}}{n}
  E_{n-1}( \partial_3 f)_{0,0,\g} +  \frac{c_{2,\g}}{n^2} E_{n-2}( \partial_1 \partial_2 f)_{0,0,\g+2}.  
\end{align}
\end{thm}

\begin{proof}
By Proposition \ref{prop:Ja+1b+1} with $\a = \b = -1$, we see that, for $k \ne 1$, 
\begin{align} \label{eq:J11g-00g}
 J_{k,m}^{-1,-1,\g} (x,y) =   & J_{k,m}^{0,0,\g}(x,y) +  e_{0,k,m}^\g J_{k-2,m}^{0,0,\g}(x,y) \\
                               & + d_{1,k,m}^\g J_{k,m-1}^{0,0,\g}(x,y)+ e_{1,k,m}^\g  J_{k-2,m-1}^{0,0,\g}(x,y) \notag\\
                                & +d_{2,k,m}^\g  J_{k,m-2}^{0,0,\g}(x,y) +e_{2,k,m}^\g  J_{k-2,m-2}^{0,0,\g}(x,y), \notag
\end{align}
where the coefficients are given explicitly by
\begin{align*}
  d_{1,k,m}^\g:= &\ \frac{- 2 (m-k+\g)}{(2m + \g - 1) (2m+ \g + 1)},  \\
  d_{2,k,m}^\g: =&  \frac{(m-k+\g-1)_2}{(2m+ \g-2)_2(2m+ \g-1)_2},\\
  e_{0,k,m}^\g:=&\  \frac{-(m-k+1)_2}{4 (2k-3)(2k-1)}, \\ 
  e_{1,k,m}^\g:= & \frac{(m-k+1)(m+k-1)(m-k+\g-1)}{2(2k-3)(2k-1)(2m+ \g-1)(2m+ \g + 1)},\\
  e_{2,k,m}^\g:=& \frac{-(m+k-2)_2 (m+k+\g-2)_2} {4(2k-3)(2k-1)(2m+ \g-2)_2(2m+ \g - 1)_2},
\end{align*}
whereas for $k =1$,
\begin{align} \label{eq:J11g-00gB}
 J_{1,m}^{-1,-1,\g} (x,y) =  & J_{1,m}^{0,0,\g}(x,y)+ \frac{m}{2} J_{0,m}^{0,0,\g}(x,y) \\
    &  + d_{1,1,m}^\g J_{1,m-1}^{0,0,\g}(x,y)+  \frac{m}{2} d_{1,0,m}^\g  J_{0,m-1}^{0,0,\g}(x,y) \notag \\ 
    & +  d_{2,1,m}^\g J_{1,m-2}^{0,0,\g}(x,y) + \frac{m}{2} d_{2,0,m}^\g J_{0,m-2}^{0,0,\g}(x,y), \notag
\end{align}
where we have used the fact that $\tau_m^{-1,0} - \tau_m^{-1,-1} =0$ when $m \ne 1$ and it is equal to $-1/2$ when $m=1$.

Let $d_{0,k,m}^{\g}:=1$. Substituting the relations \eqref{eq:J11g-00g} and  \eqref{eq:J11g-00gB} into $f-S_n^{-1,-1,\g} f$ 
and rearranging the sums, we conclude that 
\begin{align*}
 f- & S_n^{-1,-1,\g} f   =  \sum_{m=n+1}^\infty \sum_{k=0}^m \wh f_{k,m}^{-1,-1,\g} J_{k,m}^{-1,-1,\g} \\
    = &  \sum_{m=n+1}^\infty  \sum_{i=0}^2 
       \sum_{m=n+1}^\infty \bigg[\left(d_{i, 0,m}^\g \Big ( \wh f_{0,m}^{-1,-1,\g}+ \frac{m}{2}\wh f_{1,m}^{-1,-1,\g}\Big)
        e_{i, 2,m}^\g \wh f_{ 2,m}^{-1,-1,\g}\right) J_{0,m-i}^{0,0,\g}  \\
  &\qquad\qquad\qquad \qquad\qquad
  + \sum_{k=1}^{m-i} \left(d_{i, k,m}^\g \wh f_{k,m}^{-1,-1,\g}+  e_{i,k+2,m}^\g\wh f_{k+2,m}^{-1,-1,\g}\right)
         J_{k,m-i}^{0,0,\g}\bigg].
\end{align*}
Recall that $h_{k,m}^{0,0,\g}$ denotes the $L^2$ norm of $J_{k,m}^{0,0,\g}$. It follows immediately that
\begin{align}\label{eq:est11g-1}
 \| f- & S_n^{-1,-1,\g}  f\|_{0,0,\g}^2  \notag \\
   & \le  \sum_{i=0}^2  \sum_{m=n+1}^\infty 
       \sum_{m=n+1}^\infty \bigg[ \left | d_{i, 0,m}^\g \Big ( \wh f_{0,m}^{-1,-1,\g}+ \frac{m}{2}\wh f_{1,m}^{-1,-1,\g}\Big)
        e_{i, 2,m}^\g \wh f_{ 2,m}^{-1,-1,\g}\right |^2 h_{0,m-i}^{0,0,\g} \notag \\
  &\qquad\qquad\qquad \quad 
  + \sum_{k=1}^{m-i} \left | d_{i, k,m}^\g \wh f_{k,m}^{-1,-1,\g}+  e_{i,k+2,m}^\g\wh f_{k+2,m}^{-1,-1,\g}\right |^2
         h_{k,m-i}^{0,0,\g}\bigg]. 
\end{align}         
We need to bound the right hand side of \eqref{eq:est11g-1} by these two quantities. 

For the sum with $i =0$ in the right hand side of \eqref{eq:est11g-1}, we divided the sum over $k$ into three parts. For $1\le k \le m/2$,  
we have the identity
$$
\wh f_{k,m}^{-1,-1,\g} +  e_{0,k+2,m}^\g \wh f_{k+2,m}^{-1,-1,\g} =
   \left(\wh f_{k,m}^{-1,-1,\g} - A_{k,m} \wh f_{k+2,m}^{-1,-1,\g}\right)
   +\frac{m}{2(2k+3)} \wh f_{k+2,m}^{-1,-1,\g}. 
$$
For $1 \le k \le m/2$, we have $m-k\sim m$, which shows, by \eqref{eq:hkn}, that $h_{k,m}^{0,0,\g} \le c m^{-4}h_{k,m-2}^{0,0,\g+2}$ 
and we also have $\frac{m^2}{k^2}h_{k,m}^{0,0,\g} \le c m^{-2} h_{k+1,m-1}^{0,0,\g}$ by \eqref{eq:hkn}. Hence, it follows that
\begin{align*}
   \sum_{m=n+1}^\infty \sum_{k=1}^{\lfloor \f{m}2 \rfloor} & \left \vert \wh f_{k,m}^{-1,-1,\g}+  e_{0,k+2,m}^\g 
      \wh f_{k+2,m}^{-1,-1,\g}\right \vert^2 h_{k,m}^{0,0,\g} 
\le c \sum_{m=n+1}^\infty m^{-2} \sum_{k=1}^{\lfloor \f{m}2 \rfloor}  \left |\wh f_{k+1,m}^{-1,-1,\g}\right |^2 h_{k,m-1}^{0,0,\g} \\
  & \qquad \qquad \quad + c \sum_{m=n+1}^\infty m^{-4}  \sum_{k=1}^{\lfloor \f{m}2 \rfloor} 
        \left | \wh f_{k,m}^{-1,-1,\g} - A_{k,m}\wh f_{k+2,m}^{-1,-1,\g}\right |^2 h_{k,m-2}^{0,0,\g+2}  \\
 &  \qquad \qquad \le  c n^{-2} \left[ E_{n-1}(\partial_3 f)_{0,0,\g}\right]^2+ c n^{-4} \left[ E_{n-2}(\partial_1\partial_2 f)_{0,0,\g+2}\right]^2.  
\end{align*}
The same argument evidently applies to the case $k=0$ as well. For $m/2 \le k \le m$, we have $|e_{0,k+2,m}^\g| \le c$ using 
$k \sim m$, and the sum is bounded by $c n^{-2} \left[ E_{n-1}(\partial_3 f)_{0,0,\g}\right]^2$, using $h_{k,m}^{0,0,\g} \le 
c m^{-2} h_{k -1,m-1}^{0,0,\g}$, which requires $k \sim m$, and  $h_{k,m}^{0,0,\g} \le c m^{-2} h_{k +1,m-1}^{0,0,\g}$, 
which holds for all $k$. 

The other two sums in the right hand side of \eqref{eq:est11g-1} are worked out along the same line. We shall be brief. 
For the second sum, with $i=1$, we have 
\begin{align*}
d_{1,k,m}^\g & \wh f_{k,m}^{-1,-1,\g} + e_{1,k+2,m}^\g \wh f_{k+2,m}^{-1,-1,\g} \\
    =&  d_{1,k,m}^\g  (\wh f_{k,m}^{-1,-1,\g} - A_{k,m}\wh f_{k+2,m}^{-1,-1,\g}) - B_{k,m}\wh f_{k+2,m}^{-1,-1,\g},  
\end{align*}
where 
$$
 B_{k,m} =  \frac{(\g+1)(m+k+1)}{2(2k+3)(2m+\g-1)(2m+\g+1)}.
$$
This is used for $1 \le k \le m/2$, for which we bounded the sum on the first part in the right hand side by 
$c n^{-4} \left[ E_{n-2}(\partial_1\partial_2 f)_{0,0,\g+2}\right]^2$, using $|d_{1,k,m}^\g | \le m^{-1}$ and 
$h_{k,m}^{0,0,\g} \le c m^{-2} h_{k,m-1}^{0,0,\g+2}$, which holds under $m-k \sim m$, whereas we bounded 
the sum on the second part in the right hand side by $c n^{-2} \left[ E_{n-1}(\partial_3 f)_{0,0,\g}\right]^2$, 
using $B_{k,m} \le c\frac{1}{m^2} \frac{m}{k}$ and 
$\frac{m^2}{k^2}h_{k+1,m}^{0,0,\g} \le c m^{-2} h_{k+1,m-1}^{0,0,\g+2} \le c$, again under $m-k \sim m$.
For $ m/2 \le k \le m$, we use $|d_{1,k,m+1}^\g | \le c m^{-1}$ 
and $|e_{1,k+2,m+1}^\g | \le c m^{-1}$, as well as $h_{k,m}^{0,0,\g} \le c h_{k \pm 1,m}^{0,0,\g}$, which holds under 
$k\sim m$, to bound the sum by $c n^{-2} \left[ E_{n-1}(\partial_3 f)_{0,0,\g}\right]^2$.

For the third term in the right hand side of \eqref{eq:est11g-1}, we have 
\begin{align*}
d_{2,k,m}^\g & \wh f_{k,m}^{-1,-1,\g} + e_{2,k+2,m}^\g \wh f_{k+2,m}^{-1,-1,\g} \\
    = & d_{2,k,m}^\g   (\wh f_{k,m}^{-1,-1,\g} - A_{k,m} \wh f_{k+2,m}^{-1,-1,\g})  
    -  B_{k,m}\wh f_{k+2,m}^{-1,-1,\g}. 
\end{align*}
where
$$
B_{k,m} = \frac{(m+\g) (m+k)_2}{2(2k+3)(2m+\g-2)(2m+\g-1)^2(2m+\g)}.
$$
Notice that $B_{k,m}$ is again bounded by $c \frac{1}{m^2} \frac{m}{k}$. In this case, the case for $0 \le k \le m/2$
follows by using $h_{k,m}^{0,0,\g} \le c h_{k,m}^{0,0,\g+2}$ and $\frac{m^2}{k^2}h_{k,m}^{0,0,\g} \le c h_{k+1,m}^{0,0,\g}$
when $m-k \sim m$. Finally, the sum for $m/2 \le k \le m$ follows by using $|d_{2,k,m+2}|^\g \le c m^{-2}$ and 
$|e_{2,k+2,m+2}|^\g \le c m^{-2}$ as well as $h_{k,m}^{0,0,\g} \le c m^2 h_{k \pm 1,m+1}^{0,0,\g}$ when $k \sim m$. 
\end{proof}

We can also derive analogues of Theorem \ref{thm:norm-11g} for the two permutations of $(-1,-1,\g)$. They are recorded
below for latter use.

\begin{thm}\label{thm:norm-1b1}
For $f \in W_2^1(\varpi_{0,\b,0})\cap W_2^2(\varpi_{0,\b+2,0})$ and $n =2,3,\ldots$, 
\begin{align}\label{eq:norm-1b1}
 \left \|f-S_n^{-1,\b,-1} f \right\|_{0,\b,0} \le  \frac{c_{1,\b}}{n}
  E_{n-1}( \partial_1 f)_{0,\b,0} +  \frac{c_{2,\b}}{n^2} E_{n-2}( \partial_2 \partial_3 f)_{0,\b+2,0}.  
\end{align}
\end{thm}

\begin{thm}\label{thm:norm-a11}
For $f \in W_2^1(\varpi_{\a,0,0})\cap W_2^2(\varpi_{\a+2,0,0})$, and $n=2,3,\ldots$, 
\begin{align}\label{eq:norm-a11}
 \left \|f-S_n^{\a,-1,-1} f \right\|_{\a,0,0} \le  \frac{c_{1,\a}}{n}
  E_{n-1}( \partial_2 f)_{\a,0,0} +  \frac{c_{2,\a}}{n^2} E_{n-2}( \partial_3 \partial_1 f)_{\a+2,0,0}.  
\end{align}
\end{thm}

\section{Approximation by $S_n^{-1,-1,-1} f$ and the proof of Theorem 1.1.}
\label{sect14}
\setcounter{equation}{0}

Let $S_n^{-1,-1,-1}$ be the $n$-th partial sum of the orthogonal expansion defined by
$$
       S_n^{-1,-1,-1}f: = \sum_{k=0}^n \proj_n^{-1,-1,-1}f = \sum_{m=0}^n \sum_{k=0}^m \wh f_{k,m}^{-1,-1,-1} J_{k,m}^{-1,-1,-1}, 
$$
where $\wh f_{k,m}^{-1,-1,-1}$ are defined as in \eqref{eq:Fourier111}. Our main result in this section is as follows: 

\begin{thm} \label{thm:est111}
Let $f \in \CW_2^2$. Then
\begin{align} \label{eq:D-est-111}
\begin{split}
& \left \|\partial_1 f - \partial_1 S_n^{-1,-1,-1} f \right \| \le \frac{c}{n} 
 \big ( E_{n-2}(\partial_2 \partial_1 f)_{0,0,1} + E_{n-2}(\partial_1 \partial_3 f)_{1,0,0}\big), \\
& \left \|\partial_2 f - \partial_2 S_n^{-1,-1,-1} f \right \| \le \frac{c}{n} 
 \big ( E_{n-2}(\partial_1 \partial_2 f)_{0,0,1}+E_{n-2}(\partial_2 \partial_3 f)_{0,1,0} \big), \\
& \left \|\partial_3 f - \partial_3 S_n^{-1,-1,-1} f \right \| \le \frac{c}{n} 
 \big (E_{n-2}(\partial_1 \partial_3 f)_{1,0,0}+E_{n-2}(\partial_2 \partial_3 f)_{0,1,0} \big), 
\end{split}
\end{align}
and 
\begin{align} \label{eq:f-eat-111}
  \left \| f -  S_n^{-1,-1,-1} f \right \|  \le \frac{c}{n^2}  
  \big ( E_{n-2}(\partial_1 \partial_2 f)_{0,0,1} & +E_{n-2}(\partial_2 \partial_3 f)_{0,1,0} \\
   &  + E_{n-2}(\partial_3 \partial_1 f)_{1,0,0} \big). \notag
\end{align}
\end{thm} 

We note that that Theorem 1.1 follows immediately from this theorem. In fact, the estimate on the derivatives in 
\eqref{eq:D-est-111} is more refined than what we stated in Theorem \ref{thm:1.1}. 

For the proof, we start with the following proposition:

\begin{prop}
For $f\in \CW_2^2(\varpi_{-1,-1,-1})$. 
\begin{align} \label{eq:est111a}
   & \left[ E_n (\partial_1\partial_3 f)_{1,0,0}\right]^2 + \left[ E_n (\partial_2\partial_3 f)_{0,1,0}\right]^2 \\
   & \qquad\qquad  =  2 \sum_{m=n+1}^\infty \sum_{k=0}^m \left( \left |\wh f_{k,m+2}^{-1,-1,-1} \right|^2 + \f{(m+k+3)^2}{4(2k+3)^2} 
         \left|\wh  f_{k+2,m+2}^{-1,-1,-1} \right |^2 \right) h_{k,m}^{1,0,0}, \notag
\end{align}
and, with $A_{k,m}$ given in \eqref{eq:Akm}, 
\begin{align}\label{eq:est111b}
 \left[ E_{n-2}(\partial_1\partial_2 f)_{0,0,1}\right]^2  = 
    &\sum_{m=n-1}^\infty  \left |\wh f_{0,m+2}^{-1,-1,-1} + \frac{m+2}{2} \wh  f_{1,m+2}^{-1,-1,-1}
            -A_{0,m+2} \wh f_{2,m+2}^{-1,-1,-1}\right|^2 h_{0,m}^{0,0,1}   \notag \\
    & + \sum_{m=n-1}^\infty \sum_{k=1}^m \left |\wh f_{k,m+2}^{-1,-1,-1} -A_{k,m+2} \wh f_{k+2,m+2}^{-1,-1,-1} \right|^2 h_{k,m}^{0,0,1}.\end{align}
\end{prop}

\begin{proof}
Since \eqref{eq:whf-whD} holds if one or all of $\a, \b,\g$ are equal to $-1$, applying \eqref{eq:whf-whD} twice and using the 
explicit formulas for $a_{k,n}^{\a,\b}$ in \eqref{eq:akn-bkn}, we obtain
\begin{align*}
  \wh{\partial_1\partial_3 f}_{k,m}^{1,0,0} & \ = - \wh f_{k,m+2}^{-1,-1,-1} - a_{k+1,m+1}^{0,0}  \wh f_{k+2,m+2}^{-1,-1,-1},\\
  \wh{\partial_2\partial_3 f}_{k,m}^{0,1,0}& \ = - \wh f_{k,m+2}^{-1,-1,-1} + a_{k+1,m+1}^{0,0}  \wh f_{k+2,m+2}^{-1,-1,-1}, 
\end{align*}
Since $E_n(f)_{\a,\b,\g} = \|f-S_n^{\a,\b,\g} f\|_{\a,\b,\g}$, it then follows immediately from the definition of $S_n^{\a,\b,\g}$ that 
\begin{align*}
\left[ E_n (\partial_1\partial_3 f)_{1,0,0}\right]^2 &\ = 
    \sum_{m=n+1}^\infty \sum_{k=0}^m \left | \wh f_{k,m+2}^{-1,-1,-1} + a_{k+1,m+1}^{0,0}\wh f_{k+2,m+2}^{-1,-1,-1}\right |^2 h_{k,m}^{1,0,0} \\
\left[ E_n (\partial_2\partial_3 f)_{0,1,0}\right]^2  &\ = 
    \sum_{m=n+1}^\infty \sum_{k=0}^m \left |\wh f_{k,m+2}^{-1,-1,-1} - a_{k+1,m+1}^{0,0}\wh f_{k+2,m+2}^{-1,-1,-1}\right |^2 h_{k,m}^{0,1,0}.
\end{align*}
Since $h_{k,m}^{1,0,0} = h_{k,m}^{0,0,1}$ by \eqref{eq:hkn}, we combine the two terms and use the explicit formula
of $a_{k+1,m+1}^{0,0}$ to obtain \eqref{eq:est111a}.  The identity \eqref{eq:est111b} is exactly \eqref{eq:normD12},
which holds for $\g =-1$ since \eqref{eq:whf-whD2} does.
\end{proof}

\medskip

\noindent
{\it Proof of Theorem \ref{thm:est111}.}
The derivative estimates are direct consequences of Theorem \ref{thm:Dproj-1-1-1} and the estimates in Section \ref{sect5}.

For the proof of the function estimate, we first derive, from \eqref{eq:recurJc}, the identity 
$$
   J_{k,m}^{-1,-1,-1} = J_{k,m}^{-1,-1,0} + \tau_{m-k}^{-1,2k-1} J_{k,m-1}^{-1,-1,0},
$$
where 
$$
     \tau_{m-k}^{-1,2k-1} = \frac{m+k-1}{2(m-1)(2m-1)}= \frac{1}{2(2m-1)} \left(1+ \f{k}{m-1}\right). 
$$
By the triangle inequality, we see that 
\begin{align} \label{eq:est111}
\|f-  S_n^{-1,-1,-1} f\|  =  & \Big \|\sum_{m=n+1}^\infty \sum_{k=0}^m \wh f_{k,m}^{-1,-1,-1} J_{k,m}^{-1,-1,-1} \Big \| \notag \\
  \!\!\!  \le   \Big \|\sum_{m=n+1}^\infty & \sum_{k=0}^m \wh f_{k,m}^{-1,-1,-1} J_{k,m}^{-1,-1,0}  \Big \| +
   \Big \|\sum_{m=n+1}^\infty \sum_{k=0}^m \wh f_{k,m}^{-1,-1,-1}  \tau_{m-k}^{-1,2k-1}J_{k,m-1}^{-1,-1,0}  \Big \|.
\end{align}
We can then apply the expansion of $J_{k,m}^{-1,-1,\g}$ in \eqref{eq:J11g-00g}, with $\g = 0$, and bound the right hand 
side of \eqref{eq:est111} as in the proof of Theorem \ref{thm:norm-11g} by the two quantities in \eqref{eq:est111a} and \eqref{eq:est111b}.

We first work with the first sum in the right hand side of \eqref{eq:est111}. 
The proof is parallel to that of Theorem \ref{thm:norm-11g}. Indeed, we have an analog of \eqref{eq:est11g-1} in which
$\wh f_{k,m}^{-1,-,1,\g}$ becomes $\wh f_{k,m}^{-1,-1,-1}$ but $h_{k,m}^{0,0,\g}$ becomes $h_{k,m}^{0,0,0}$; that is,
\begin{align}\label{eq:est111-A}
 \| f-  S_n^{-1,-1,0}  f\|^2 &  \notag \\
       \le \sum_{i=0}^2  \sum_{m=n+1}^\infty &
       \sum_{m=n+1}^\infty \bigg[ \left | d_{i, 0,m}^{0} \Big ( \wh f_{0,m}^{-1,-1,-1}+ \frac{m}{2}\wh f_{1,m}^{-1,-1,-1}\Big)
        e_{i, 2,m}^{0} \wh f_{ 2,m}^{-1,-1,-1}\right |^2 h_{0,m-i}^{0,0,0} \notag \\
  &\qquad\quad
  + \sum_{k=1}^{m-i} \left | d_{i, k,m}^0 \wh f_{k,m}^{-1,-1,-1}+  e_{i,k+2,m}^0 \wh f_{k+2,m}^{-1,-1,-1}\right |^2
         h_{k,m-i}^{0,0,0}\bigg]. 
\end{align}         
Since 
\eqref{eq:est111b} is exactly \eqref{eq:normD12} with $\g =-1$ and $h_{k,m}^{0,0,1}\sim h_{k,m}^{0,0,2}$ for $0\le k \le m/2$,
by \eqref{eq:hkn}, the essential part of the proof that uses \eqref{eq:est111b} follows exactly as in the proof of 
Theorem \ref{thm:norm-11g}. Moreover, when we need to use \eqref{eq:est111a} instead 
of \eqref{eq:normD3}, we gain a $m^{-2}$ as can be easily seen from $h_{k,m}^{0,0,0} \sim m^{-2} \frac{k^2}{m^2}
h_{k+1,m-1}^{1,0,0}$ and $h_{k,m}^{0,0,0} \sim m^{-2} (\frac{m^2} {k^2} h_{k-1,m-1}^{1,0,0})$, both of which follow 
from \eqref{eq:hkn}. Hence, the right hand side of \eqref{eq:est111-A} is bounded $c n^{-2}$ multiple of the quantities
in \eqref{eq:est111a} and \eqref{eq:est111b} exactly as in the proof of Theorem \ref{thm:norm-11g}. 
This takes care of the first term. 

For the second sum in the right hand side of \eqref{eq:est111}, we write it as a sum of two according to the sum 
$\tau_{m-k}^{-1,2k-1} = \frac{1}{2(2m-1)} +  \frac{1}{2(2m-1)} \frac{k}{m-1}$. The first sum is exactly as before,
apart from the factor $\frac{1}{2(2m-1)}$, and it follows exactly as before since $h_{k,m-1}^{0,0,0}  \sim m^{2} h_{k,m}^{0,0,0}$,
taking care of the fact that the second sum is over $J_{k,m-1}^{-1,-1,0}$ instead of $J_{k,m}^{-1,-1,0}$. For the second sum, 
we need to modify the proof since $\wh f_{k,m}^{-1,-,1,\g}$ becomes $\frac{1}{2(2m-1)} \wh f_{k,m}^{-1,-1,-1} \frac{k}{m-1}$.
Clearly, the factor $\frac{1}{2(2m-1)}$ can be used with $h_{k,m-1}^{0,0,0} 
 \sim m^{2} h_{k,m}^{0,0,0}$ again. For the sum of $i=0$ in the right hand side of \eqref{eq:est111-A}, 
with $\frac{k}{m-1}$ factor, we write, for $k \ge 1$,  
\begin{align*}
\f{k}{m-1} \wh f_{k,m}^{-1,-1,-1}  & +  e_{0,k+2,m}^0 \f{k+1}{m-1}\wh f_{k+2,m}^{-1,-1,-1} \\
     = & \frac{k}{m-1} \left(  \wh f_{k,m}^{-1,-1,-1}  +  e_{0,k+2,m}^0\wh f_{k+2,m}^{-1,-1,-1}\right)
      +    \frac{1}{m-1} e_{0,k+2,m}^0 \wh f_{k+2,m}^{-1,-1,-1}.  
\end{align*}
Since $k/(m-1)$ is bounded, the sum over the first term in the right hand side works as before, and we are left with bounding
the sum over $|\frac{1}{m-1} e_{0,k+2,m}^0 \wh f_{k+2,m}^{-1,-1,-1}|^2 h_{k,m}^{0,0,0}$, which is easily seen to be bounded,
using $h_{k,m-1}^{0,0,0} \sim m^{2} h_{k,m}^{0,0,0}$ again, by $c n^{-6}$ multiple of \eqref{eq:est111a}, with
$n^{-2}$ to spare. The same proof works for $k=0$, where the additional $m/2$ accompanying $\wh f_{1,m}^{-1,-1,-1}$
is \eqref{eq:est111-A} gives a bound $c n ^{-4}$ multiple of \eqref{eq:est111a}. The two sums of $i=1$ and $i=2$ can be 
estimated in exactly the same way. We omit the details. 
\qed

\section{Approximation by $S_n^{-1,-1,-2}$ and its permutations}
\label{sect15}
\setcounter{equation}{0}

We assume $m \ge 3$ below. Recall that an orthonormal basis for $\CV_m(W_{-1,-1,-2})$ under $\la \cdot,\cdot\ra_{-1,-1,-2}^J$ in
\eqref{eq:ipd112J} consists of $\wh J_{k,m}^{-1,-1,-2}$, which are equal to $J_{k,m}^{-1,-1,-2}$ if $0 \le k \le m-2$, and 
\begin{align*}
  \wh J_{m-1,m}^{-1,-1,-2} = & J_{m-1,m}^{-1,-1,-2} +2 J_{m,m}^{-1,-1,-2} \\
  \quad  \wh J_{m,m}^{-1,-1,-2} = & \frac{m-1}{2m-3} J_{m-1,m}^{-1,-1,-2} -  \frac{2(3m-5)}{2m-3} J_{m,m}^{-1,-1,-2}. 
\end{align*}
The Fourier coefficients $\wh f_{k,m}^{-1,-1,-2}$ are defined with respect to this basis. For $0 \le k \le m$, we define 
$\wh g_{k,m}$ by 
\begin{align}\label{eq:gkm-112}
\begin{split}
  \wh g_{k,m} = \begin{cases} \wh f_{k,m}^{-1,-1,-2}, & 0 \le k \le m-2, \\
    \wh f_{m-1,m}^{-1,-1,-2}+ \frac{m-1}{2m-3} \wh f_{m,m}^{-1,-1,-2}, & k = m-1, \\
    2 \wh f_{m-1,m}^{-1,-1,-2} - \frac{2(3m-5)}{2m-3} \wh f_{m,m}^{-1,-1,-2}, & k = m.
      \end{cases}
\end{split}
\end{align}
It is then easy to see that 
$$
  \proj_m^{-1,-1,-2} f = \sum_{k=0}^m \wh f_{k,m}^{-1,-1,-2} \wh J_{k,m}^{-1,-1,-2} 
        =  \sum_{k=0}^m \wh g_{k,m}  J_{k,m}^{-1,-1,-2}.  
$$

\begin{lem}
Let $\wh g_{k,m}$ be defined as in \eqref{eq:gkm-112}. Then
\begin{align}\label{eq:est112b}
 \left[ E_{n-2}(\partial_1\partial_2 f)_{0,0,0}\right]^2  = 
  & \sum_{m=n-1}^\infty  \left |\wh g_{0,m+2} + \frac{m+2}{2} \wh g_{1,m+2} 
        - A_{0,m+2} \wh g_{2,m+2} \right|^2 h_{0,m}^{0,0,0}   \notag \\
  & + \sum_{m=n-1}^\infty \sum_{k=1}^m \left |\wh g_{k,m+2}  -A_{k,m+2} \wh g_{k+2,m+2} \right|^2 h_{k,m}^{0,0,0},
\end{align}
and, with $A_{k,m}$ given in \eqref{eq:Akm},
\begin{align}\label{eq:est112a}
& \left[ E_n(\partial_1^2 \partial_3 f)_{2,0,0}\right]^2   +\left[ E_n(\partial_2^2 \partial_3 f)_{0,2,0}\right]^2  \\
    = 2 \sum_{m=n+1}^\infty & \sum_{k=0}^m  \left( \left | \mathfrak{d}_{k+2,m+3} \wh g_{k+2,m+3}\right|^2
    + \left | \wh g_{k+1,m+3}+ \mathfrak{c}_{k+3,m+3}\wh g_{k+3,m+3} \right|^2 \right)
       h_{k,m}^{2,0,0} \notag ,
\end{align}
where 
$$
\mathfrak{d}_{k+2,m+3} = \frac{m+k+4}{2(k+2)}, \quad \mathfrak{c}_{k+3,m+3} =\frac{(k+1)(m+k+4)(m+k+5)}{4(k+2)(2k+3)(2k+5)}.
$$
\end{lem}

\begin{proof}
Since the proof of \eqref{eq:whf-whD2} can be derived from the derivatives of $J_{k,n}^{\a,\b,\g}$ and the identity 
$\partial_1\partial_2 \proj_n^{-1,-1,\g} f = \proj_{n-2}^{0,0,\g+2} (\partial_1 \partial_2 f)$, it also holds for $\g = -2$ with 
$\wh g_{k,m}$ in place of $\wh f^{-1,-1,-2}$. Consequently, \eqref{eq:est112b} follows. Now, using \eqref{eq:diff1J} 
repeatedly, it is easy to see that 
\begin{align*}
  \partial_1^2 \partial_3 J_{k,m}^{-1,-1,-2}  =&   \mathfrak{c}_{k,m} J_{k-3,m-3}^{2,0,0} 
    +  \mathfrak{d}_{k,m} J_{k-2,m-3}^{2,0,0} + J_{k-1,m-3}^{2,0,0},\\
  \partial_2^2 \partial_3 J_{k,m}^{-1,-1,-2}   = &   \mathfrak{c}_{k,m} J_{k-3,m-3}^{0,2,0} \ 
         -  \mathfrak{d}_{k,m} J_{k-2,m-3}^{0,2,0} + J_{k-1,m-3}^{0,2,0}. 
\end{align*}
where $\mathfrak{c}_{k,m} = a_{k-1,m-1}^{0,0}  a_{k-2,m-2}^{1,0}$ and 
$\mathfrak{d}_{k,m} =a_{k-1,m-1}^{0,0} + a_{k-1,m-2}^{1,0}$, which can be given explicitly by \eqref{eq:akn-bkn}.

By \eqref{eq:comm112} and relations in Section \ref{sect5}, $\partial_1^2 \partial_3 \proj_m^{-1,-1,-2} f = 
\proj_{m-3}^{2,0,0}(\partial_1^2 \partial_3  f)$ and \allowbreak $\partial_2^2 \partial_3 \proj_m^{-1,-1,-2} f = 
\proj_{m-3}^{2,0,0}(\partial_2^2 \partial_3  f)$. Together, these allow us to derive that 
\begin{align*}
\wh{\partial_1^2 \partial_3 f}_{k,m}^{2,0,0} = & \mathfrak{c}_{k+3,m+3}\wh g_{k+3,m+3}
     +  \mathfrak{d}_{k+2,m+3}\wh g_{k+2,m+3}  + \wh g_{k+1,m+3},\\
\wh{\partial_2^2 \partial_3 f}_{k,m}^{0,2,0} = & \mathfrak{c}_{k+3,m+3} \wh g_{k+3,m+3}
     -  \mathfrak{d}_{k+2,m+3} \wh g_{k+2,m+3}  + \wh g_{k+1,m+3},
\end{align*}
which can be used in the series of  $\left[ E_n(\partial_1^2 \partial_3 f)_{2,0,0}\right]^2$ and 
$\left[ E_n(\partial_2^2 \partial_3 f)_{0,2,0}\right]^2$, respectively. Since $h_{k,m}^{2,0,0} = h_{k,m}^{0,2,0}$, 
combining the two series leads to \eqref{eq:est112a}. 
\end{proof}

Let $S_n^{-1,-1,-2}$ be the $n$-th partial sum of the orthogonal expansion defined by
$$
       S_n^{-1,-1,-2}f: = \sum_{k=0}^n \proj_n^{-1,-1,-2}f. 
$$

\begin{thm} \label{thm:Est112}
For $f \in W_2^2$ and $n= 3,4,\ldots$,
\begin{align} \label{eq:Est112}
\| f- S_n^{-1,-1,-2}f\|   \le \, & \frac{c_1}{n^2} E_{n-2}(\partial_1 \partial_2 f)_{0,0,0}  \\
     +  & \frac{c_2}{n^3} \big[ E_{n-3}(\partial_1^2 \partial_2 f)_{2,0,0}+  E_{n-3}(\partial_2^2 \partial_3 f)_{0,2,0} \big]. \notag
\end{align}
\end{thm}

\begin{proof}
We assume $m \ge 3$ below. By \eqref{eq:recurJc}, we deduce that 
\begin{equation} \label{eq:Jaa2-Jaa0}
 J_{k,m}^{\a,\a,-2} = J_{k,m}^{\a,\a,0} +  \mathfrak{u}_{k,m}^\a J_{k,m-1}^{\a,\a,0} 
     + \mathfrak{v}_{k,m}^\a J_{k,m-2}^{\a,\a,0}, 
\end{equation}
where 
\begin{align*}
 \mathfrak{u}_{k,m}^\a = & \tau_{m-k}^{-1,2k+2\a+1}+\tau_{m-k}^{-2,2k+2\a+1} = \frac{2(m+k+2\a+1)}{(2m+2\a-1)(2m+2\a+1)},\\
 \mathfrak{v}_{k,m}^\a = & \tau_{m-k}^{-2,2k+2 \a +1}\tau_{m-k-1}^{-1,2k+2\a+1} = \frac{(m+k+2\a)(m+k+2\a+1)}
      {4(m+\a-1)(m+\a)(2m+2\a-1)^2}. 
\end{align*}
For $\a = -1$, we conclude, by the triangle inequality, that 
\begin{align} \label{eq:est112}
\|f-  S_n^{-1,-1,-2} f\|  = & \, \Big \|\sum_{m=n+1}^\infty \sum_{k=0}^m \wh g _{k,m} J_{k,m}^{-1,-1,-2} \Big \| \notag \\
     \le &\, \Big \|\sum_{m=n+1}^\infty \sum_{k=0}^m \wh g_{k,m} J_{k,m}^{-1,-1,0}  \Big \| + 
 \Big \|\sum_{m=n+1}^\infty \sum_{k=0}^m \ \mathfrak{u}_{k,m}^{-1} \wh g_{k,m} J_{k,m-1}^{-1,-1,0}  \Big \| \\
   & +   \Big \|\sum_{m=n+1}^\infty \sum_{k=0}^m \ \mathfrak{v}_{k,m}^{-1} \wh g_{k,m} J_{k,m-2}^{-1,-1,0}  \Big \|  \notag
\end{align}
We need to bound the three sums in the right hand side of  \eqref{eq:est112} by the quantities in 
\eqref{eq:est112a} and \eqref{eq:est112b}.
 
The estimate of the first sum is parallel to that of Theorem \ref{thm:est111}, which goes back to the proof of Theorem
\ref{thm:norm-11g}. First of all, \eqref{eq:est112b} is the same as \eqref{eq:normD12}, if we replace 
$\wh f_{k,m}^{-1,-,1,\g}$ by $\wh g_{k,m} \sim \wh f^{-1,-1,-2}$ and $h_{0,m}^{0,0,\g+2}$ by $h_{0,m}^{0,0,0}$. 
We have three sums as in the right hand side of \eqref{eq:est111-A} with $\wh f_{k,m}^{-1,-,1,\g}$ replaced by 
$\wh g_{k,m}$. Since $h_{k,m}^{0,0,0}\sim h_{k,m}^{0,0,2}$ for $0\le k \le m/2$, by \eqref{eq:hkn}, the part of the proof 
of Theorem \ref{thm:norm-11g} that uses \eqref{eq:est112b} works exactly in the current case. For the part 
$|\frac{m}{k} \wh g_{k+2,m}|^2 h_{k,m}^{0,0,0}$ for $0 \le k \le m/2$, we use $h_{k,m}^{0,0,0} \le c m^{-6} \frac{m^2}{k^2} 
h_{k,m-3}^{2,0,0}$ and $h_{k,m}^{0,0,0} \le c m^{-6} \frac{m^2}{k^2} h_{k-2,m-3}^{2,0,0}$ to bound it by $c n^{-6}$ 
multiple of the right hand side of \eqref{eq:est112a}. The parts that involve $B_{k,m}$ in the sums of $i=1$ and $i=2$
are handled similarly. The case $k \ge m/2$ can be bounded straightforwardly as in the proof of Theorem \ref{thm:norm-11g}.
For example, for the sum with $i=1$, we use the fact that $|d_{1,k,m+1}| \le c m^{-1}$ and $|e_{1,k,m+1}| \le c m^{-1}$, and
$h_{k,m}^{0,0,0} \le c m^{-4} \frac{m^2}{k^2} h_{k,m-2}^{2,0,0}$ and $h_{k,m}^{0,0,0} \le c m^{-4} \frac{m^2}{k^2} 
h_{k-2,m-2}^{2,0,0}$ to show that term is bound by $c n^{-6}$ multiplied by the right hand side of \eqref{eq:est112a}. 
The other two cases, $i=0$ and $i=2$, are handled similarly. 

For the second and the third sums in the right hand side of \eqref{eq:est112}, we need to modify our proof since 
$\wh g_{k,m}$ is replaced by $\wh g_{k,m} \mathfrak{u}_{k,m}^{-1}$ and $\wh g_{k,m} \mathfrak{v}_{k,m}^{-1}$,
respectively. Since 
\begin{align*}
 \mathfrak{u}_{k,m}^{-1} & = \frac{2 (m-1)}{(2m-3) (2m-1)}  +  \frac{2 k}{(2m-3) (2m-1)}  \\
 \mathfrak{v}_{k,m}^{-1} & = \frac{1}{4(2m-3)^2} + \frac{k}{4(2m-3)(m-2)(m-1)}+ \frac{k^2}{4(2m-3)^2(m-2)(m-1)}.  
\end{align*}
the modification can be carried out as in the proof of Theorem \ref{thm:est111}. The sums that have the main terms of 
$\mathfrak{u}_{k,m}^{-1}$ and $\mathfrak{u}_{k,m}^{-1}$ can be bounded similarly, since $h_{k,m}^{0,0,0}
\sim m^2 h_{k,m+1}^{0,0,0} \sim m^4  h_{k,m+2}^{0,0,0}$. The remaining terms contain extra factors of $\frac{k}{m}$ 
in the case of $\mathfrak{u}_{k,m}^{-1}$, $\frac{k}{m}$ and $\frac{k^2}{m^2}$ in the case of $\mathfrak{v}_{k,m}^{-1}$.
They can be easily handled because of the extra negative powers of $m$. In the case of $\mathfrak{u}_{k,m}^{-1}$, this 
follows just as in the modification in the proof of Theorem \ref{thm:est111}, so is the $\frac{k}{4(2m-3)(m-2)(m-1)}$ part
of  $\mathfrak{v}_{k,m}^{-1}$. In the remaining case of $\mathfrak{v}_{k,m}^{-1}$, we can write, for example, 
\begin{align*}
 \frac{k^2}{m^2} \wh g_{k,m} + & \frac{(k+1)^2}{m^2} e_{i,k+2,m}^0 \wh g_{k+2,m} \\
  & =  \frac{k^2}{m^2}  \left(  \wh g_{k,m} +  e_{i,k+2,m}^0 \wh g_{k+2,m}\right) +  \frac{2k+1}{m^2} e_{i,k+2,m}^0 \wh g_{k+2,m}
\end{align*}
Again the sum for the first term works as before since $k^2/m^2 \le 1$, and the sum over the remaining cases can be 
easily bounded by \eqref{eq:est112a}. We omit the details. 
\end{proof}

We can also derive estimate on the approximation of the first order derivatives by using Theorem \ref{thm:comm112} and
the estimates in Section \ref{sect5}. Since these are just intermediate results in between Theorem \ref{thm:1.1} and 
Theorem \ref{thm:1.2}, we shall not state them. 

We can also derive analogues of Theorem \ref{thm:Est112} for the two permutations of $(-1,-1,\g)$. They are 
recorded below for latter use.

\begin{thm} \label{thm:Est121}
For $f \in W_2^2$ and $n= 3,4,\ldots$,
\begin{align} \label{eq:Est121}
\| f- S_n^{-1,-2,-1}f\|   \le \, & \frac{c_1}{n^2} E_{n-2}(\partial_2 \partial_3 f)_{0,0,0}  \\
     +  & \frac{c_2}{n^3} \big[ E_{n-3}(\partial_1 \partial_2^2 f)_{2,0,0}+  E_{n-3}(\partial_1 \partial_3^2 f)_{0,2,0} \big]. \notag
\end{align}
\end{thm}

\begin{thm} \label{thm:Est211}
For $f \in W_2^2$ and $n= 3,4,\ldots$,
\begin{align} \label{eq:Est211}
\| f- S_n^{-2,-1,-1}f\|   \le \, & \frac{c_1}{n^2} E_{n-2}(\partial_1 \partial_3 f)_{0,0,0}  \\
     +  & \frac{c_2}{n^3} \big[ E_{n-3}(\partial_2 \partial_3^2 f)_{2,0,0}+  E_{n-3}(\partial_2 \partial_1^2 f)_{0,2,0} \big]. \notag
\end{align}
\end{thm}

\section{Approximation by $S_n^{-2,-2,-2} f$  and proof of Theorem \ref{thm:1.2}}
\label{sect16}
\setcounter{equation}{0}

Let $S_n^{-2,-2,-2}$ be the $n$-th partial sum of the orthogonal expansion defined by
$$
       S_n^{-2,-2,-2}f: = \sum_{k=0}^n \proj_n^{-2,-2,-2}f  
$$
where $\proj_n^{-2,-2,-2}$ are defined as in \eqref{eq:proj222}. Our main result in this section is as follows: 

\begin{thm}\label{thm:Est222}
For $f \in W_2^2$, let $p_n = S_n^{-2,-2,-2} f$. Then 
\begin{align*}
  \|   f - p_n \|  \le & \frac{c_1}{n^3}  E_{n-3}(\partial_1\partial_2 \partial_3) \\
  & + \frac{c_2}{n^4} \left[ E_{n-4}(\partial_1^2\partial_2^2)_{2,0,0}+E_{n-4}(\partial_2^2\partial_3^2)_{0,2,0}+
         E_{n-4}(\partial_3^2\partial_1^2)_{0,0,2} \right], \\
\| \partial_1 f - \partial_1 p_n \| \le  & \frac{c_1}{n^2}  E_{n-3}(\partial_1\partial_2 \partial_3) 
   + \frac{c_2}{n^3} \left[ E_{n-4}(\partial_1^2\partial_2^2)_{0,0,2}+E_{n-4}(\partial_3^2\partial_1^2)_{2,0,0} \right],  \\
\| \partial_2 f - \partial_1 p_n \|  \le &  \frac{c_1}{n^2}  E_{n-3}(\partial_1\partial_2 \partial_3) 
   + \frac{c_2}{n^3} \left[ E_{n-4}(\partial_2^2\partial_3^2)_{0,2,0}+E_{n-4}(\partial_2^2\partial_1^2)_{0,0,2} \right],  \\
   | \partial_3 f - \partial_3 p_n \|  \le &  \frac{c_1}{n^2}  E_{n-3}(\partial_1\partial_2 \partial_3) 
   + \frac{c_2}{n^3} \left[ E_{n-4}(\partial_1^2\partial_3^2)_{2,0,0}+E_{n-4}(\partial_3^2\partial_2^2)_{0,2,0} \right],  
\end{align*}
and
\begin{align*}
  \| \partial_1\partial_2 f - \partial_1 \partial_2 p_n\|  \le &  \frac{c_1}{n}  E_{n-3}(\partial_1\partial_2 \partial_3) 
   + \frac{c_2}{n^2}  E_{n-4}(\partial_1^2\partial_2^2)_{0,0,2}, \\
    \| \partial_2\partial_3 f - \partial_2 \partial_3 p_n\| \le  &  \frac{c_1}{n}  E_{n-3}(\partial_1\partial_2 \partial_3) 
   + \frac{c_2}{n^2}  E_{n-4}(\partial_2^2\partial_3^2)_{0,2,0}, \\
    \| \partial_3\partial_1 f - \partial_3 \partial_1 p_n\| \le &  \frac{c_1}{n}  E_{n-3}(\partial_1\partial_2 \partial_3) 
    + \frac{c_2}{n^2} E_{n-4}(\partial_3^2\partial_1^2)_{2,0,0}.
\end{align*}
\end{thm} 

We note that Theorem 1.2 follows immediately from this theorem by applying \eqref{eq:est-aag} on 
$E_{n-3}(\partial_1\partial_2 \partial_3)$. Furthermore, the estimate on the derivatives in Theorem \ref{thm:Est222} 
is more refined than what we stated in Theorem \ref{thm:1.2}.

We assume $m \ge 4$ below. First we recall that a mutually orthogonal basis for 
$\CV_m(W_{-1,-1,-2})$ consists of 
$\wh J_{k,m}^{-2,-2,-2}$, which are equal to $J_{k,m}^{-2,-2,-2}$ if $k=0$ or $2 \le k \le m-2$, and 
\begin{align*}
  \wh J_{1,m}^{-2,-2,-2} = & J_{1,m}^{-2,-2,-2} + \frac{m-2}{2} J_{0,m}^{-2,-2,-2} \\
  \wh J_{m-1,m}^{-2,-2,-2} = & J_{m-1,m}^{-2,-2,-2} +2 J_{m,m}^{-2,-2,-2} \\
  \quad  \wh J_{m,m}^{-2,-2,-2} = & \frac{m-2}{2m-5} J_{m-1,m}^{-2,-2,-2} -  \frac{2(3m-8)}{2m-5} J_{m,m}^{-2,-2,-2}. 
\end{align*}
Let $\wh f_{k,n}^{-2,-2,-2}$ be defined as in \eqref{eq:proj222}
For $0 \le k \le m$, we define $\wh g_{k,m}$ by 
$$
  \wh g_{k,m} = \begin{cases} \wh f_{0,m}^{-2,-2,-2}+  \frac{m-2}{2} \wh f_{1,m}^{-2,-2,-2}, & k=0 \\
    \wh f_{k,m}^{-2,-2,-2}, & 0 \le k \le m-2, \\
    \wh f_{m-1,m}^{-2,-2,-2}+ \frac{m-2}{2m-5} \wh f_{m,m}^{-2,-2,-2}, & k = m-1, \\
    2 \wh f_{m-1,m}^{-2,-2,-2} - \frac{2(3m-8)}{2m-5} \wh f_{m,m}^{-2,-2,-2}, & k = m.
      \end{cases}
$$
It is easy to see then
\begin{equation} \label{eq:proj222}
  \proj_m^{-2,-2,-2} f = \sum_{k=0}^m \wh f_{k,m}^{-2,-2,-2} \wh J_{k,m}^{-1,-1,-2} 
        =  \sum_{k=0}^m \wh g_{k,m}  J_{k,m}^{-2,-2,-2}.  
\end{equation}

\begin{lem}
Let $\wh g_{k,m}$ be defined as in \eqref{eq:proj222}. Then
\begin{align} \label{eq:est222a}
\left[ E_n (\partial_1\partial_2 \partial_3 f)_{0,0,0} \right]^2 
   = \sum_{m=n+1}^\infty & \left( \left | \wh g_{1,m+3} + \frac{m+2}{2} \wh g_{2,m+3} - A_{3,m+3} \wh g_{3,m+3}\right|^2 
      h_{0,m}^{0,0,0}\right.  \notag \\
      & \left. \quad +  \sum_{k=1}^m   \left|\wh g_{k+1,m+3}  - A_{k+3,m+3} \wh g_{k+3,m+3} \right|^2  h_{k,m}^{0,0,0}\right),
\end{align}
where 
$$
  A_{k+3,m+3} = \frac{(m+k+2)(m+k+3)}{4(2k+1)(2k+3)}, 
$$ 
and
\begin{align} \label{eq:est222b}
\left[ E_n (\partial_1^2 \partial_3^2 f)_{2,0,0} \right]^2 + \left[E_n (\partial_2^2 \partial_3^2 f)_{0,2,0} \right]^2 
  & = \sum_{m=n+1}^\infty \sum_{k=0}^m  \left( \left | \mathfrak{p}_{k+3,m+4} \wh g_{k+3,m+4}\right|^2 h_{k,m}^{2,0,0}
     \right. \notag\\
      & \left. \quad +  \left|\wh g_{k+2,m+4} + \mathfrak{q}_{k+4,m+4} \wh g_{k+4,m+4} \right|^2 h_{k,m}^{2,0,0}\right), 
\end{align}
where 
$$
\mathfrak{p}_{k+3,m+4}=\frac{m+k+4}{2(k+2)}, \quad \mathfrak{q}_{k+4,m+4}=\frac{(k+1)(m+k+4)(m+k+5)}{4(k+2)(2k+3)(2k+5)},
$$
and, furthermore, 
\begin{align} \label{eq:est222c}
 \left[ E_n  (\partial_1^2 \partial_2^2 f)_{0,0,2} \right]^2 & \notag \\
  &\!\!\!\!\!\!\!\!\!\!\!\! \!\!\!\!\!\!\!\!\!\!\!\!= \sum_{m=n+1}^\infty  \left | \wh g_{0,m+4} + 
    \mathfrak{s}_{2,m+4} \left(\wh g_{2,m+4}- \frac{m+4}{4}\wh g_{3,m+4}\right)
        + \mathfrak{t}_{4,m+4}\wh g_{4,m+4}\right|^2 h_{0,m}^{0,0,2}
     \notag \\
 &  \!\!\!\!\!\!+ \left( \left|\wh g_{1,m+4} + \frac{m+3}{2}\wh g_{2,m+4}
    + \mathfrak{s}_{3,m+4}\wh g_{3,m+4}+ \mathfrak{t}_{5,m+4}\wh g_{5,m+4}\right|^2 h_{1,m}^{0,0,2}  \right.  \\  
 & \!\!\!\!\!\! \left. + \sum_{k=2}^m \left|\wh g_{k,m+4} +\mathfrak{s}_{k+2,m+4}\wh g_{k+2,m+4}+
        \mathfrak{t}_{k+4,m+4}\wh g_{k+4,m+4} \right|^2 h_{k,m}^{0,0,2}\right), \notag
\end{align}
where
\begin{align*}
    \mathfrak{s}_{k+2,m+4}&= -\frac{(m+k+2)(m+k+3)}{2(2k-1)(2k+3)}, \\ 
\mathfrak{t}_{k+4,m+4}& =\frac{ (m+k+2)(m+k+3)(m+k+4)(m+k+5)}{16(2k+1)(2k+3)^2(2k+5)}.
\end{align*}
\end{lem}

\begin{proof}
Using \eqref{eq:diff1J} three times, it is easy to see that 
\begin{align*}
  \partial_1 \partial_2 \partial_3 J_{2,m}^{-2,-2,-2} & = J_{1,m-3}^{0,0,0}+ \frac{m-1}{2} J_{0,m-3}^{0,0,0}, \\
  \partial_1 \partial_2 \partial_3 J_{k,m}^{-2,-2,-2} & = J_{k-1,m-3}^{0,0,0}-A_{k,m} J_{k-3,m-3}^{0,0,0}, \quad k \ne 2,
\end{align*}
where $A_{k,m} = a_{k,m}^{-2,-2}a_{k-1,n-1}^{-2,-1}$. Hence, by $\partial_1 \partial_2 \partial_3 \proj_n^{-2,-2,-2}f = \proj_{n-3}^{0,0,0}(\partial_1 \partial_2 \partial_3 f)$, it follows that 
\begin{align*}
\wh {\partial_1\partial_2\partial_3 f}_{0,m}^{0,0,0} & = \wh g_{1,m+3} + \frac{m+2}{2} \wh g_{2,m+3} - A_{3,m+3} \wh g_{3,m+3},\\
\wh { \partial_1 \partial_2 \partial_3 f}_{k,m}^{0,0,0} & = \wh g_{k+1,m+3}  - A_{k+3,m+3} \wh g_{k+3,m+3}, \quad k \ge 1,
\end{align*}
from which \eqref{eq:est222a} follows immediately. The proof of \eqref{eq:est222b} follows along the same line, using 
\begin{align*}
   \partial_1^2 \partial_3^2 J_{k,m}^{-2,-2,-2}  & = J_{k-2,m-4}^{2,0,0}+ \mathfrak{p}_{k,m} J_{k-3,m-4}^{2,0,0} 
         + \mathfrak{q}_{k,m} J_{k-4,m-4}^{2,0,0},  \\
   \partial_2^2 \partial_3^2 J_{k,m}^{-2,-2,-2}  & = J_{k-2,m-4}^{0,2,0}- \mathfrak{p}_{k,m} J_{k-3,m-4}^{0,2,0}
         + \mathfrak{q}_{k,m} J_{k-4,m-4}^{0,2,0},  
\end{align*}
where $\mathfrak{p}_{k,m} = a_{k,m}^{-2,-2}a_{k,n-1}^{-1,-2}$ and $\mathfrak{q}_{k,m} = a_{k,m}^{-2,-2}a_{k-1,n-1}^{-1,-2}$,
which leads to  
\begin{align*}
\wh {\partial_1^2\partial_3^2 f}_{k,m}^{2,0,0} & = \wh g_{k+2,m+4} +\mathfrak{p}_{k+3,m+4}\wh g_{k+3,m+4}+
  \mathfrak{q}_{k+4,m+4}\wh g_{k+4,m+4},\\
\wh {\partial_2^2\partial_3^2 f}_{k,m}^{0,2,0} & = \wh g_{k+2,m+4} - \mathfrak{p}_{k+3,m+4}\wh g_{k+3,m+4}+
  \mathfrak{q}_{k+4,m+4}\wh g_{k+4,m+4}.
\end{align*}
The identity \eqref{eq:est222b} follows from these relations since $h_{k,m}^{2,00} = h_{k,m}^{0,2,0}$. 
Finally, the proof of \eqref{eq:est222c} relies on  
\begin{align*}
   \partial_1^2 \partial_2^2 J_{1,m}^{-2,-2,-2} & = J_{1,m-4}^{0,0,2}, \quad 
    \partial_1^2 \partial_2^2 J_{2,m}^{-2,-2,-2}  = J_{2,m-4}^{0,0,2}+ \frac{m-1}{2} J_{1,m-4}^{0,0,2} +
       \mathfrak{s}_{2,m} J_{0,m-4}^{0,0,2},  \\
   \partial_1^2 \partial_2^2 J_{3,m}^{-2,-2,-2} & = J_{3,m-4}^{0,0,2}+ \mathfrak{s}_{3,m} J_{1,m-4}^{0,0,2} -
      \frac{m}{4} \mathfrak{s}_{2,m}  J_{0,m-4}^{0,0,2},  \\
   \partial_1^2 \partial_2^2 J_{k,m}^{-2,-2,-2} & = J_{k,m-4}^{0,0,2}+ \mathfrak{s}_{k,m} J_{k-2,m-4}^{0,0,2} +
       \mathfrak{t}_{k,m} J_{k-4,m-4}^{0,0,2},  \quad k \ge 4,
\end{align*}
where $\mathfrak{s}_{2,m} = - a_{2,m-2}^{-1,-1}a_{1,m-3}^{-1,0}$,  $\mathfrak{s}_{k,m}=-(a_{k,m-2}^{-1,-1}a_{k-1,m-3}^{-1,0}
+a_{k,m}^{-2,-2}a_{k-1,m-1}^{-2,-1})$ and $\mathfrak{t}_{k,m} =a_{k,m}^{-2,-2}a_{k-1,m-1}^{-2,-1} 
a_{k-2,m-2}^{-1,-1}a_{k-3,m-3}^{-1,0}$ for $k \ge 3$. 
These identities are derived from applying \eqref{eq:diff1J} repeatedly; for $k\ge 3$, we used the identity 
$a_{k,m-1}^{\a,\a+1} = a_{k,m}^{\a,\a}$, which holds when $\a, k, m$ satisfy the condition (a) in  
Proposition \ref{prop:partialJ}, and we need to use \eqref{eq:diff1J} under the condition (b) when $k=1, 2$ and $3$. 
From $\partial_1^2 \partial_2^2 \proj_n^{-2,-2,-2}f = \proj_{n-4}^{0,0,2}(\partial_1^2 \partial_2^2 f)$, it follows that
\begin{align*}
\wh {\partial_1^2\partial_2^2 f}_{0,m}^{0,0,2} &  = \wh g_{0,m+4} +\mathfrak{s}_{2,m+4}\left(\wh g_{2,m+4}
 + \frac{m+4}{4}\wh g_{3,m+4} \right)+ \mathfrak{t}_{4,m+4}\wh g_{4,m+4},\\
\wh {\partial_1^2\partial_2^2 f}_{1,m}^{0,0,2} & = \wh g_{1,m+4} + \frac{m+3}{2}\wh g_{2,m+4}
    + \mathfrak{s}_{3,m+4}\wh g_{3,m+4}+ \mathfrak{t}_{5,m+4}\wh g_{5,m+4},\\
\wh {\partial_1^2\partial_2^2 f}_{k,m}^{0,0,2} & = \wh g_{k,m+4} +\mathfrak{s}_{k+2,m+4}\wh g_{k+2,m+4}+
        \mathfrak{t}_{k+4,m+4}\wh g_{k+4,m+4},
\end{align*}
from which \eqref{eq:est222c} follows readily. 
\end{proof}

\medskip\noindent
{\it Proof of Theorem \ref{thm:Est222}}.
Our main task is to estimate $\|f-p_n\|$. The estimate for the first order partial derivatives follow directly from Theorem 
\ref{thm:proj222} and the estimate in \eqref{eq:Est112},  \eqref{eq:Est121} and  \eqref{eq:Est211}. Moreover, the
estimate for the second order partial derivatives follow form the commuting relation derived from Theorem \ref{thm:proj222}, 
\eqref{eq:comm112b}, \eqref{eq:comm211} and \eqref{eq:comm121}, and the estimates in Section \ref{sect6}.

We assume $m \ge 4$ below. By \eqref{eq:Jaa2-Jaa0}, we can deduce as in the proof of 
Theorem \ref{thm:Est112} that 
\begin{align} \label{eq:est222L}
 \|f-S_n^{-2,-2,-2}  f\|^2 = & \left \| \sum_{m=n+1}^\infty \sum_{k=0}^m \wh g_{k,m}J_{k,m}^{-2,-2,-2} \right \|  \notag\\
   \le & \left \| \sum_{m=n+1}^\infty \sum_{k=0}^m \wh g_{k,m}J_{k,m}^{-2,-2,0} \right \| 
         +  \left \| \sum_{m=n+1}^\infty \sum_{k=0}^m \wh g_{k,m} \mathfrak{u}_{k,m}^{-2} J_{k,m-1}^{-2,-2,0} \right \| \\
   & +  \left \| \sum_{m=n+1}^\infty \sum_{k=0}^m \wh g_{k,m}\mathfrak{v}_{k,m}^{-2} J_{k,m-2}^{-2,-2,0} \right \|. \notag
\end{align}
We need to bound the right hand side of \eqref{eq:est222L} by the three quantities in \eqref{eq:est222a}, \eqref{eq:est222b}
and \eqref{eq:est222c}. Since $\mathfrak{u}_{k,m}^{-2}$ and $\mathfrak{v}_{k,m}^{-2}$ behavior similarly as 
$\mathfrak{u}_{k,m}^{-1}$ and $\mathfrak{v}_{k,m}^{-1}$, the second and the third sums can be bounded similarly as
the first sum, by appropriate modifications of the proof as in the proof of Theorem \ref{thm:Est112}. We shall omit the 
details of the modification and concentrating on the estimate of the first sum. 

Now, by Proposition \ref{prop:Ja+1b+1} and \eqref{eq:J11g-00g}, we derive that, for $k \ne 2, 3$,  
\begin{align*} 
  J_{k,m}^{-2,-2,0}   (x,y) =  & J_{k,m}^{0,0,0}(x,y) + v_{0,k,m} J_{k-2,m}^{0,0,0}(x,y)
         + w_{0,k,m} J_{k-4,m}^{0,0,0}(x,y)   \\
     + u_{1,k,m} & J_{k,m-1}^{0,0,0}(x,y) + v_{1,k,m} J_{k-2,m-1}^{0,0,0}(x,y)
           +w_{1,k,m} J_{k-4,m-1}^{0,0,0}(x,y)  \\
     + u_{2,k,m} & J_{k,m-2}^{0,0,0}(x,y) + v_{2,k,m} J_{k-2,m-2}^{0,0,0}(x,y)
           +w_{2,k,m} J_{k-4,m-2}^{0,0,0}(x,y)   \\ 
      + u_{3,k,m} & J_{k,m-3}^{0,0,0}(x,y) + v_{3,k,m} J_{k-2,m-3}^{0,0,0}(x,y)
           +w_{3,k,m} J_{k-4,m-3}^{0,0,0}(x,y)  \\
       + u_{4,k,m} & J_{k,m-4}^{0,0,0}(x,y) + v_{4,k,m} J_{k-2,m-4}^{0,0,0}(x,y)
           +w_{4,k,m} J_{k-4,m-4}^{0,0,0}(x,y), 
 \end{align*}
for $k =3$,
\begin{align*} 
  J_{3,m}^{-2,-2,0}   (x,y) =  &\,  J_{3,m}^{0,0,0}(x,y) + v_{0,3,m} J_{1,m}^{0,0,0}(x,y)
         + \tfrac{2}{m+1} w_{0,3,m}J_{0,m}^{0,0,0}(x,y)   \\
      + u_{1,3,m} &  J_{3,m-1}^{0,0,0}(x,y) + v_{1,3,m} J_{1,m-1}^{0,0,0}(x,y)
           + \tfrac{2}{m} w_{1,3,m} J_{0,m-1}^{0,0,0}(x,y)  \\
     + u_{2,3,m} & J_{3,m-2}^{0,0,0}(x,y) + v_{2,3,m} J_{1,m-2}^{0,0,0}(x,y)
           + \tfrac{2}{m-1} w_{2,3,m} J_{0,m-2}^{0,0,0}(x,y)   \\ 
      + u_{3,3,m} & J_{3,m-3}^{0,0,0}(x,y) + v_{3,3,m} J_{1,m-3}^{0,0,0}(x,y)
           + \tfrac{2}{m-2} w_{3,3,m} J_{0,m-3}^{0,0,0}(x,y)  \\
       + u_{4,3,m} & J_{3,m-4}^{0,0,0}(x,y) + v_{4,3,m} J_{1,m-4}^{0,0,0}(x,y)
           +\tfrac{2}{m-3} w_{4,3,m} J_{0,m-4}^{0,0,0}(x,y), 
 \end{align*}
and for $k =2$,
\begin{align*} 
  J_{2,m}^{-2,-2,0}   (x,y) =  &\,  J_{2,m}^{0,0,0}(x,y) + \tfrac{m-1}{2} J_{1,m}^{0,0,0}(x,y)
         + v_{0,2,m} J_{0,m}^{0,0,0}(x,y)   \\
      + u_{1,2,m}&  J_{2,m-1}^{0,0,0}(x,y) + \tfrac{m-1}{2} u_{1,1,m} J_{1,m-1}^{0,0,0}(x,y)
           +  v_{1,2,m} J_{0,m-1}^{0,0,0}(x,y)  \\
     + u_{2,2,m} & J_{2,m-2}^{0,0,0}(x,y) + \tfrac{m-1}{2} u_{2,1,m} J_{1,m-2}^{0,0,0}(x,y)
           +   v_{2,2,m} J_{0,m-2}^{0,0,0}(x,y)   \\ 
      + u_{3,2,m} & J_{2,m-3}^{0,0,0}(x,y) +\tfrac{m-1}{2} u_{3,1,m} J_{1,m-3}^{0,0,0}(x,y)
           +  v_{3,2,m} J_{0,m-3}^{0,0,0}(x,y)  \\
       + u_{4,2,m} & J_{2,m-4}^{0,0,0}(x,y) + \tfrac{m-1}{2} u_{4,1,m} J_{1,m-4}^{0,0,0}(x,y)
           +v_{4,2,m} J_{0,m-4}^{0,0,0}(x,y), 
 \end{align*}
where
\begin{align*}
  v_{0,k,m}:= &\ - \frac{ (m-k+1)_2}{2(2k-1) (2k-5)},  \,
  w_{0,k,m}: = \frac{(m-k+1)_4}{16(2k-7)(2k-5)^2(2k-3)},\\
  u_{1,k,m} :=&\  -\frac{4(m-k)}{(2m-3)(2m+1)}, \, 
  v_{1,k,m}:= \frac{(m-k+1)((2m-1)^2+(2k-1)(2k-5))}{2 (2k-1)(2k-5)(2m-3)(2m+1)},\\
  w_{1,k,m}:=& \ - \frac{(m-k+1)_3(m+k-3)^2} {4(2k-7)(2k-5)^2(2k-3)(2m-3)(2m+1)},\\
  u_{2,k,m} :=&\ \frac{3(m-k-1)_2}{2 m(m-2)(2m-3)(2m-1)}, \\  
  v_{2,k,m}:= &\ - \frac{3(k-2)^2(k-1)^2-m(m-2)(2(k-2)(k-1)+3(m-1)^2)}{4(2k-1)(2k-5)m(m-2)(2m-3)(2m-1)}, \\
  w_{2,k,m}:=&\ \frac{3(m-k+1)_2 (m+k-4)^2(m+k-3)^2} {32(2k-7)(2k-5)^2(2k-3)m(m-2)(2m-3)(2m-1)},\\
  u_{3,k,m} :=&\ - \frac{(m-k-2)_3}{(m-2)(m-1)(2m-5)(2m-3)^2(2m-1)}, \\ 
  v_{3,k,m}:= &\ \frac{(m-k)(m+k-3)^2 (7 + 2 k(k - 3) + 2 m(m-3))}{4(2k-1)(2k-5)(m-2)(m-1)(2m-5)(2m-3)^2(2m-1)},\\
  w_{3,k,m}:=&\ -\frac{(m-k-1)(m+k-5)^2 (m+k-4)^2(m+k-3)^2}{16(2k-7)(2k-5)^2(2k-3)(m-2)(m-1)(2m-5)(2m-3)^2(2m-1)},\\
  u_{4,k,m} :=&\  \frac{-(m-k-3)_4}{16(m-3)(m-2)^2(m-1)(2m-5)^2(2m-3)^2}, \\ 
  v_{4,k,m}:= & \ - \frac{(m-k-1)_2(m+k-4)^2(m+k-3)^2}{32(2k-1)(2k-5)(m-3)(m-2)^2(m-1)(2m-5)^2(2m-3)^2},\\
  w_{4,k,m}:=& \ \frac{(m+k-6)_4^2} {256 (2k-7) (2k-5)^2 (2 k-3) (m-3) (m-2)^2 (m-1) (2m-5)^2 (2m-3)^2},
\end{align*}
Let $u_{0,k,m}:=1$. Writing $J_{k,m}^{-2,-2,-2}$ in terms of $J_{k,m}^{0,0,0}$ and rearranging the terms in the summation,
it is not difficult to see that 
\begin{align*}
 f - & S_n^{-2,-2,-2} f  =  \sum_{m=n+1}^\infty \sum_{k=0}^m \wh g_{k,m} J_{k,m}^{-2,-2,-2} \\
    & =\sum_{m=n+1}^\infty \sum_{i=0}^4 \left[ \left (u_{i,0,m}\wh g_{0,m} +v_{i,2,m}\wh g_{2,m}
      + \tfrac{2}{m+1-i}w_{i,3,m}\wh g_{3,m}+w_{i,4,m}\wh g_{4,m} \right)J_{0,m-i}^{0,0,0}  \right. \\
    & \qquad\qquad\qquad + \left (u_{i,1,m}  (\wh g_{1,m} +\tfrac{m-1}{2} \wh g_{2,m}) 
       +v_{i,3,m}\wh g_{3,m}+w_{i,5,m}\wh g_{5,m}\right)J_{1,m-i}^{0,0,0} \\
    & \qquad\qquad\qquad + \left. \sum_{k=2}^{m-i} \left(u_{i,k,m} \wh g_{k,m} +v_{i,k+2,m}\wh g_{k+2,m}
             +w_{i,k+4,m}\wh g_{k+4,m}\right) J_{k,m-i}^{0,0,0}
    \right].   
\end{align*}
Consequently, by the triangle inequality and the orthogonality of $J_{k,m}^{0,0,0}$, we conclude that
\begin{align}\label{eq:est222}
\|f -   S_n^{-2,-2,-2} f\|^2 & \notag \\
    \le \sum_{i=0}^4 \sum_{m=n+1}^\infty & \left[ \left | u_{i,0,m}\wh g_{0,m} +v_{i,2,m}\wh g_{2,m}
      + \tfrac{2}{m+1-i}w_{i,3,m}\wh g_{3,m}+w_{i,4,m}\wh g_{4,m} \right |^2 h_{0,m-i}^{0,0,0}  \right. 
      \notag \\
    & + \left |u_{i,1,m}  (\wh g_{1,m} +\tfrac{m-1}{2} \wh g_{2,m}) 
       +v_{i,3,m}\wh g_{3,m}+w_{i,5,m}\wh g_{5,m}\right|^2 h_{1,m-i}^{0,0,0} \\
    &  + \left. \sum_{k=2}^{m-i} \left | u_{i,k,m} \wh g_{k,m} +v_{i,k+2,m}\wh g_{k+2,m}
             +w_{i,k+4,m}\wh g_{k+4,m}\right|^2 h_{k,m-i}^{0,0,0}
   \notag \right].   
\end{align}
We now bound the right hand side of \eqref{eq:est222} by the sum of the three quantities in 
\eqref{eq:est222a}, \eqref{eq:est222b} and \eqref{eq:est222c}.

For each fixed $i$, there are three cases: $k =0$, $k=1$ and $2 \le k \le m$. The estimates below are tedious, 
but by no means difficult. We shall be brief and only give full detail for the case $i= 0$. 

For $i =0$ and $2 \le k \le m$, a direct verification shows that 
\begin{align}\label{eq:pf222a}
  \wh g_{k,m} +v_{0,k+2,m}\wh g_{k+2,m} + & w_{0,k+4,m}\wh g_{k+4,m}
  = (\wh g_{k,m} +\mathfrak{s}_{k+2,m}\wh g_{k+2,m} +\mathfrak{t}_{k+4,m}\wh g_{k+4,m}) \notag \\
  & + \tfrac{m-1}{2k+3} \left( \wh g_{k+2,m}+ A_{k+4,m} \wh g_{k+4,m}\right) + 
 \tfrac{m(m-1)}{5(2k+3)(2k+5)} \wh g_{k+4,m}. 
\end{align}
Thus, the sum over $k$ is bounded by three sums whose summands are the three terms, under absolute value,
in the right hand side of \eqref{eq:pf222a}. Together with the summation over $m$, we see that the first sum is 
bounded by $c m^{-8}$ multiple of the quantity in \eqref{eq:est222c} when $0 \le k \le m/2$, since 
$h_{k,m}^{0,0,0} \sim  m^{-8} h_{k,m-4}^{0,0,2}$ when $m-k \sim m$ according to \eqref{eq:hkn}, and it is bounded 
by $c m^{-8}$ multiple of the quantity in \eqref{eq:est222b} when $k \ge m/2$, where we work with the sum over
$\mathfrak{p}_{k,m}$, since $h_{k,m}^{0,0,0} \sim  m^{-8} h_{\ell,m-4}^{2,0,0}$ for $\ell = k-3, k-1$ and $k+1$ 
when $k \sim m$; the second sum is bounded by $c m^{-6}$ multiple of the quantity in \eqref{eq:est222a}, since 
$(m^2/k^2) h_{k,m}^{0,0,0} \sim m^{-6}  h_{k+1,m-3}^{0,0,0}$; and the third sum is bounded by $c m^{-8}$ multiple 
of the quantity in \eqref{eq:est222b}, working with the sum of $\mathfrak{p}_{k,m}$, since 
$(m^2/k^2) h_{k,m}^{0,0,0} \sim m^{-6}  h_{k+1,m-4}^{2,0,0}$. 

For $i =0$ and $k=1$, a direct verification shows that 
\begin{align}\label{eq:pf222b}
  \wh g_{1,m} +\tfrac{m-1}{2} \wh g_{2,m} + v_{0,3,m}\wh g_{3,m} +   w_{0,5,m}\wh g_{5,m}
 & = (\wh g_{1,m} +\tfrac{m-1}{2}+ \mathfrak{s}_{3,m}\wh g_{3,m} +\mathfrak{t}_{5,m}\wh g_{5,m}) \notag \\
  & + \tfrac{m-1}{5} \left( \wh g_{3,m}+ A_{5,m} \wh g_{5,m}\right) + 
   \tfrac{m(m-1)}{140} \wh g_{5,m}. 
\end{align}
Similarly, for $i =0$ and $k=0$, a direct verification shows that 
\begin{align}\label{eq:pf222c}
&  \wh g_{0,m} + v_{0,2,m}\wh g_{2,m} +\tfrac{2}{m+1} w_{0,3,m} \wh g_{3,m} +   w_{0,4,m}\wh g_{4,m} \\
  = (\wh g_{0,m} + \mathfrak{s}_{2,m}& \left(g_{2,m}- \tfrac{m}{4}\wh g_{3,m}\right ) +\mathfrak{t}_{5,m}\wh g_{5,m}) 
   + \tfrac{m-1}{3} \left( \wh g_{2,m}+ A_{4,m} \wh g_{4,m}\right) +  \tfrac{m(m-1)}{60} \wh g_{4,m}. \notag
\end{align}
It is easy to see that both these cases follow exactly as in the case of $2 \le k \le m$. This completes the estimate
for the case $i =0$. 

The other four cases are estimated similarly. In place of \eqref{eq:pf222a}, we have the identity 
\begin{align}\label{eq:pf222d}
 u_{i,k,m} \wh g_{k,m}  + & v_{i,k+2,m}\wh g_{k+2,m} + w_{i,k+4,m}\wh g_{k+4,m}  \\
   = & \, u_{i,k,m}(\wh g_{k,m} +\mathfrak{s}_{k+2,m}\wh g_{k+2,m} +\mathfrak{t}_{k+4,m}\wh g_{k+4,m}) \notag \\ 
      & + \lambda_{i,k,m} \left( \wh g_{k+2,m}+ A_{k+4,m} \wh g_{k+4,m}\right) + \mu_{i,k,m} \wh g_{k+4,m}, \notag
\end{align}
where $\l_{i,k,m}$ and $\mu_{i,k}$ are rational functions of $k$ and $m$. In all four cases, there is no need to divide 
the sum over $k$ into two sums over $k \le m/2$ and $k \ge m/2$, respectively. For the three terms corresponding 
to the right hand side of \eqref{eq:pf222d}, the first sum works since $u_{i,k,m}^2 h_{k,m-i}^{0,0,0} \le c m^{-8} h_{k,m-4}^{0,0,2}$,
the second sum works since $\l_{i,k,m}^2 h_{k,m-i}^{0,0,0} \le c m^{-6} h_{k+1,m-3}^{0,0,0}$,  and the third sum works since
$\mu_{i,k,m}^2 h_{k,m-i}^{0,0,0} \le c m^{-8} \frac{m^2}{k^2} h_{k+1,m-4}^{2,0,0}$, where $\frac{m^2}{k^2}$ comes from
$\mathfrak{p}_{k,m}$. Similarly, in pace of \eqref{eq:pf222b}, we have 
\begin{align*}
   b_{i,1,m}(\wh g_{1,m}   + &  \tfrac{m-1}{2} \wh g_{2,m}) + v_{i,3,m}\wh g_{3,m} +   w_{i,5,m}\wh g_{5,m} \\
 = & b_{i,1,m} (\wh g_{1,m} +\tfrac{m-1}{2}+ \mathfrak{s}_{3,m}\wh g_{3,m} +\mathfrak{t}_{5,m}\wh g_{5,m}) \notag \\
  & + \l_{i,1,m}\left( \wh g_{3,m}+ A_{5,m} \wh g_{5,m}\right) + 
  \mu_{i,1,m} \wh g_{5,m},  
\end{align*}
and, in pace of \eqref{eq:pf222c}, we have 
\begin{align*}
   u_{i,0,m} \wh g_{0,m} + & v_{i,2,m}\wh g_{2,m} +\tfrac{2}{m+1-i} w_{i,3,m} \wh g_{3,m} +   w_{i,4,m}\wh g_{4,m} \\
   = & u_{i,0,m} (\wh g_{0,m} + \mathfrak{s}_{2,m}  \left(g_{2,m}- \tfrac{m}{4}\wh g_{3,m}\right ) +\mathfrak{t}_{5,m}\wh g_{5,m}) \\
 & +\l_{i,0,m} \left( \wh g_{2,m}+ A_{4,m} \wh g_{4,m}\right) +  \mu_{i,0,m} \wh g_{4,m}. \notag
\end{align*}
Both cases can be handled as in the case of $2 \le k \le m$. We omit the details. 
\qed

\bigskip
\noindent 
{\bf Acknowledgement}: The author thanks Dr. Huiyuan Li for helpful discussions in early stage of this work, and thanks two 
anonymous referees for their helpful comments.

\end{document}